\documentclass[12pt]{iopart}

\usepackage{amssymb}
\usepackage{graphicx}
\usepackage{amsthm}
\usepackage{booktabs}
\newtheorem{proposition}{Proposition}[section]

\newtheorem{theorem}[proposition]{Theorem}
\newtheorem{lemma}[proposition]{Lemma}

\newcommand{\RNum}[1]{\uppercase\expandafter{\romannumeral #1\relax}}

\begin{document}

\title[Reconstruct Acoustic Sources in Time Domain]{A novel sampling method for time domain acoustic inverse source problems}

\author{Jiaru Wang \& Bo Chen* \& Qingqing Yu \& Yao Sun*}

\address{College of science, Civil Aviation University of China, Tianjin, China}
\ead{Bo Chen: charliecb@163.com\\
~~~~~~~Yao Sun: sunyao10@mails.jlu.edu.cn}
\vspace{10pt}
\begin{indented}
\item[]July 2023
\end{indented}

\begin{abstract}
This paper is concerned with the inverse acoustic scattering problems of reconstructing time-dependent multiple point sources and sources on a curve $L$ of the form $\lambda(t)\tau(x)\delta_L(x)$. A direct sampling method with a novel indicator function is proposed to reconstruct the sources. The sampling method is easy to implement and low in calculation cost. Based on the sampling method, numerical algorithms are provided to get the positions and intensities of the sources. Both theoretical analysis of the indicator function and numerical experiments are provided to show the effectiveness of the proposed method.

\noindent\textbf{Keywords:}  time domain, inverse source problem, wave equation, sampling method

\end{abstract}

%
%
%
%
%

\section{Introduction}

The inverse source problems of acoustic waves in the time domain have been widely used in radar detection, gesture recognition, and many other areas of science and engineering. Two typical models of acoustic scattering problems are frequency domain problems and time domain problems. For a long time, people paid more attention to the frequency domain problems \cite{arens2003linear,colton1998inverse,sun2016reciprocity,zhang2013novel}. However, time-dependent data and multi-frequency wave fields have great value in practical applications, and a series of achievements have been published in recent years \cite{chen2017time,guo2016time,sayas2011retarded}.

The reconstruction of point sources in the frequency domain has been well-studied in the past decades \cite{bao2010multi,ji2020identification}. The uniqueness and local stability analysis of multiple point sources for the time-harmonic case are considered in \cite{el2011inverse}. Newton's iteration method to solve the nonlinear equation of the point sources for an inverse source problem of time-harmonic acoustic waves can be seen in \cite{alves2009iterative}. Reconstruction of unknown sources of the acoustic field with multi-frequency measurement data by a recursive algorithm is studied in \cite{bao2015recursive}, and a series of analyses of other multi-frequency inverse source problems are shown in \cite{li2016increasing,zhang2015fourier}. Determining an obstacle and the corresponding incident point sources in the Helmholtz equation from near-field data can be seen in \cite{chang2021simultaneous}.

The reconstruction of point sources in the time domain has received more and more attention in recent years. An inverse source problem with a moving point source is studied in \cite{wang2017mathematical} and a sampling method is provided. In \cite{chen2020numerical}, the authors propose a modified method of fundamental solutions (MMFS) based on the time convolution of the Green's function and the signal function to reconstruct multiple point sources. The inverse source problem in nonhomogeneous background media is analyzed in \cite{devaney2007inverse}.

In comparison with the extensive studies on inverse source problems with point sources, as far as we know, the identifications of other kinds of sources are relatively rare. In \cite{de2015inverse}, the authors provide an analysis of inverse source problems with the source terms that are delta-like in time and of limited oscillation in space. In \cite{2003An}, an inverse source problem with the source term $q(t)\delta(\partial G)$ is considered. Reconstruction of point sources and extended sources at sparse sensors can be seen in \cite{ji2021reconstruction}. In \cite{lan48method}, the authors propose a numerical method for the inverse heat source problem with a line source.

Among the various methods, the sampling method plays an important role in the numerical calculation of the inverse source problems, such as the linear sampling method \cite{guo2013toward,li2008multilevel} and the direct sampling method \cite{liu2017novel,zhang2018locating}.

In this paper, a direct sampling method with a novel indicator function is proposed to reconstruct time-dependent acoustic multiple-point sources and sources of the form $\lambda(t)\tau(x)\delta_L(x)$, where $\lambda(t)$ is the signal function, $\tau(x)$ is the intensity function of the sources on the curve $L$, and
$$\delta_L(x)=\int_{L}\delta(x-y){\rm d}s(y),$$
in which $\delta(\cdot)$ is the Dirac delta distribution. The indicator function is constructed based on the time convolution of the Green's function and the signal function. In \cite{chen2020numerical}, the MMFS is proposed to reconstruct multiple point sources, for which equations have to be solved to obtain the reconstruction of the sources. Nevertheless, the direct sampling method proposed in this paper provides a reconstruction scheme of the sources, either multiple point sources or the sources on a curve $L$, without solving any equations, and the numerical implementation of the algorithm is simpler.

The outline of this paper is as follows. A brief introduction to the inverse source problems is provided in Section 2. A novel direct sampling method to reconstruct multiple point sources and sources of the form $\lambda(t)\tau(x)\delta_L(x)$ is proposed in Section 3 and the theoretical feasibility of the method is proved. In Section 4, numerical experiments are provided to verify the effectiveness of the proposed method. The concluding remarks are provided in Section 5.

\section{Problem setting}

The Green's function for the d'Alembert operator $c^{-2}\partial_{tt}-\triangle$ is
\begin{equation*}
G(x,t;s)=\frac{\delta(t-c^{-1}|x-s|)}{4\pi|x-s|},
\end{equation*}
where $c>0$ denotes the sound speed of the homogeneous background medium, $\displaystyle{\partial_{tt}u=\frac{\partial^2{u}}{\partial{t^2}}}$ and $\triangle$ is the Laplacian in $\mathbb{R}^3$.

The method of fundamental solutions expands the solution of the scattering problem utilizing the Green's function. However, since the Green's function for the d'Alembert operator involves the Dirac delta distribution, we usually use
\begin{equation*}
G*\lambda(x,t;s)=G(x,t;s)*\lambda(t)=\frac{\lambda(t-c^{-1}|x-s|)}{4\pi|x-s|},
\end{equation*}
the time convolution of the Green's function and the signal function, to expand the solution of the time domain acoustic scattering problem.

\subsection{The case of multiple point sources}

Let $\Omega\subset\mathbb{R}^3$ be a bounded convex open region and $u(x,t)$ be the wave field generated by multiple point sources which satisfies the wave equation
\begin{equation}\label{wave equation}
c^{-2}\partial_{tt}u(x,t)-\triangle{u(x,t)}=\sum_{j=1}^{K} \lambda(t)\tau_{j}\delta(x-s_j),\quad x\in\mathbb{R}^3,\,t\in\mathbb{R},
\end{equation}
where $K\in\mathbb{N}^*$ is a positive integer, $s_j\in\Omega$ are the locations of the source points, and~$\tau_j>0$ are the intensities of the sources. The source points $s_j$ are assumed to be mutually distinct.

The signal function $\lambda(t)$ is assumed to be causal, which means the source term $\displaystyle{\sum_{j=1}^{K} \lambda(t)\tau_{j}\delta(x-s_j)}$ vanishes for $t<0$ and the initial condition
\begin{equation*}
u(\cdot,0)=\partial_{t}u(\cdot,0)=0 \quad \textup{in} \; {\mathbb{R}^3}
\end{equation*}
is a direct conclusion of the causality.

The forward scattering problem is to solve (\ref{wave equation}) for the wave field $u$ with the known source term. The solution to the forward scattering problem can be expressed as
\begin{equation}\label{point solution}
u(x,t)=\sum_{j=1}^{K}\tau_jG(x,t;s_j)\ast\lambda(t),\quad x\in\mathbb{R}^3,\,t\in\mathbb{R}.
\end{equation}


The inverse source problem (P1) under consideration is: Reconstruct the locations and intensities of the point sources in (\ref{wave equation}) utilizing the measured wave field data
\begin{equation}\label{field data}
u(x,t),\quad x\in\partial\Omega,\,t\in\mathbb{R},
\end{equation}
where $\partial \Omega$ is the boundary of $\Omega$.

\subsection{The case of sources on a curve}

Let $u(x,t)$ be the wave field generated by the sources located on a curve $L$ which satisfies the wave equation
\begin{equation}\label{wave equation2}
c^{-2}\partial_{tt}u(x,t)-\triangle{u(x,t)}= \lambda(t)\tau(x)\delta_L(x),\quad x\in\mathbb{R}^3,\,t\in\mathbb{R},
\end{equation}
where $\tau(x)$ is the intensity function of the sources and $\displaystyle{\delta_L(x)=\int_{L}\delta(x-y){\rm d}s(y)}$. Then we have
\begin{equation*}
 \delta_L(x)=
 \left\{
 \begin{array}{ll}
 +\infty, &x\in L, \\
 0, &\textup{otherwise}
 \end{array}
 \right.
 \end{equation*}
and $\displaystyle{\int_{\mathbb{R}^3}\delta_L(x){\rm d}x=l}$, in which $l$ is the length of $L$.

The initial condition is still a direct conclusion of the causality. The forward scattering problem is to solve (\ref{wave equation2}) for the wave field $u$ with the known source term  $\lambda(t)\tau(x)\delta_L(x)$. The solution to the forward problem can be expressed as
\begin{equation}\label{line solution}
u(x,t)=\displaystyle{\int_{L}\tau(y)G(x,t;y)\ast\lambda(t){\rm d}s(y),\quad x\in\mathbb{R}^3,\,t\in\mathbb{R}}.
\end{equation}

The inverse source problem (P2) under consideration is: Reconstruct the locations of the sources in (\ref{wave equation2}) from the measured wave field data (\ref{field data}).

\section{The direct sampling method}

The direct sampling methods are independent of the prior information of the unknown sources and are easy to implement. In this section, we are going to establish a direct sampling method to reconstruct multiple point sources and sources on a curve.

Let $D\subset\Omega$ be a bounded sampling region that contains the sources in it. That is, $S\subset D$ for the case of point sources, where $S:=\left\{s_j,\,j=1,2,\ldots,K\right\}$ is the collection of the source points, and $L\subset D$ for the case of sources on the curve $L$.

For the sampling point $z\in D,$ define the indicator function
\begin{equation}\label{indicator function}
I(z)=\displaystyle{\int_{\mathbb{R}}\int_{\partial\Omega}u\left(x,t\right) \Big(G(x,t;z)\ast{\lambda(t)}\Big){\rm d}s(x){\rm d}t}.
\end{equation}
The direct sampling method is based on the properties of the indicator function $I(z)$, which will be proved for both the case of multiple point sources and the case of sources on a curve.

\subsection{The reconstruction of multiple point sources}
\begin{figure}[!ht]
 \center
 \scriptsize
 \begin{tabular}{cc}
  \includegraphics[width=4.5cm]{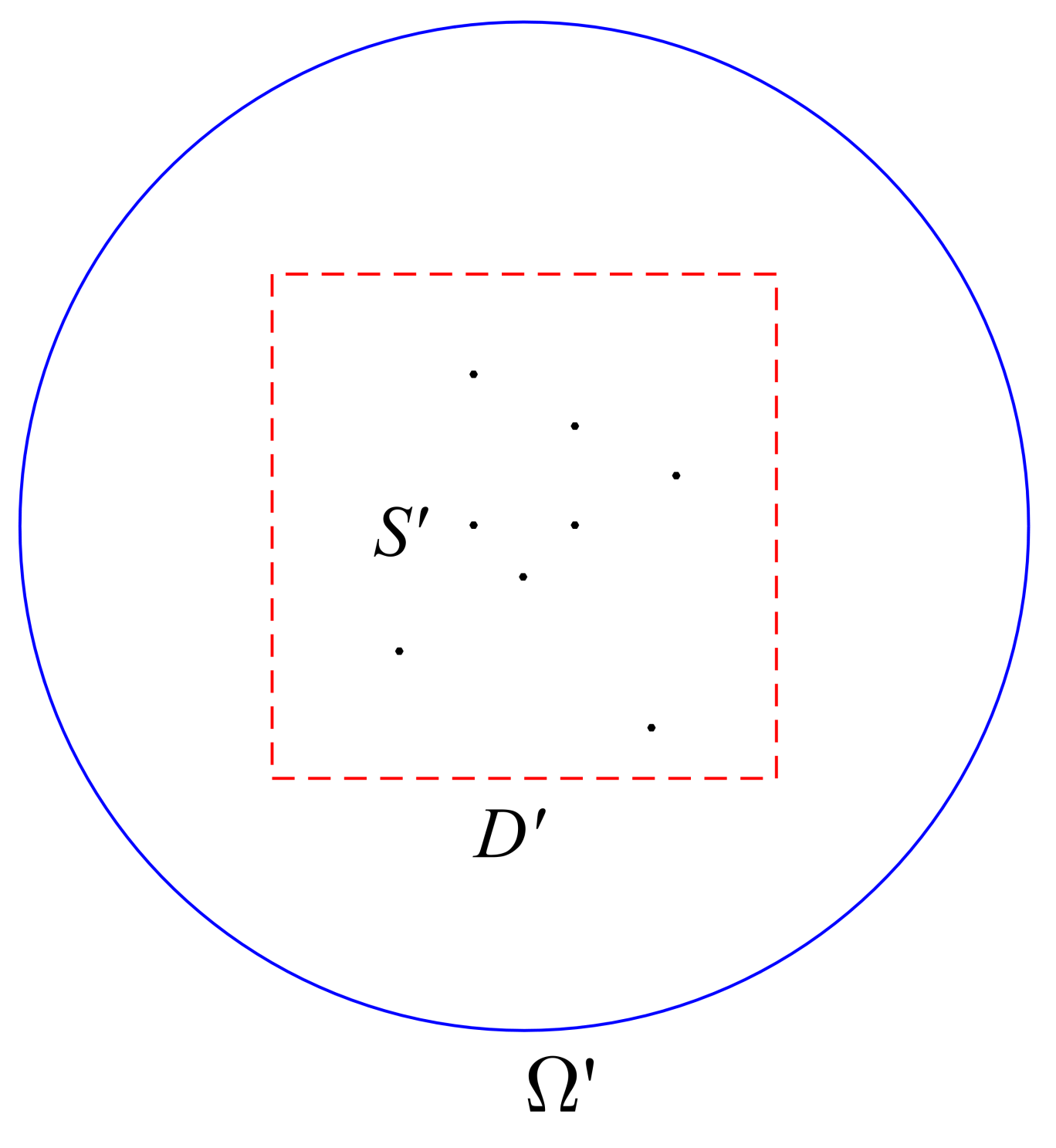}~~~ & ~~~ \includegraphics[width=4.5cm]{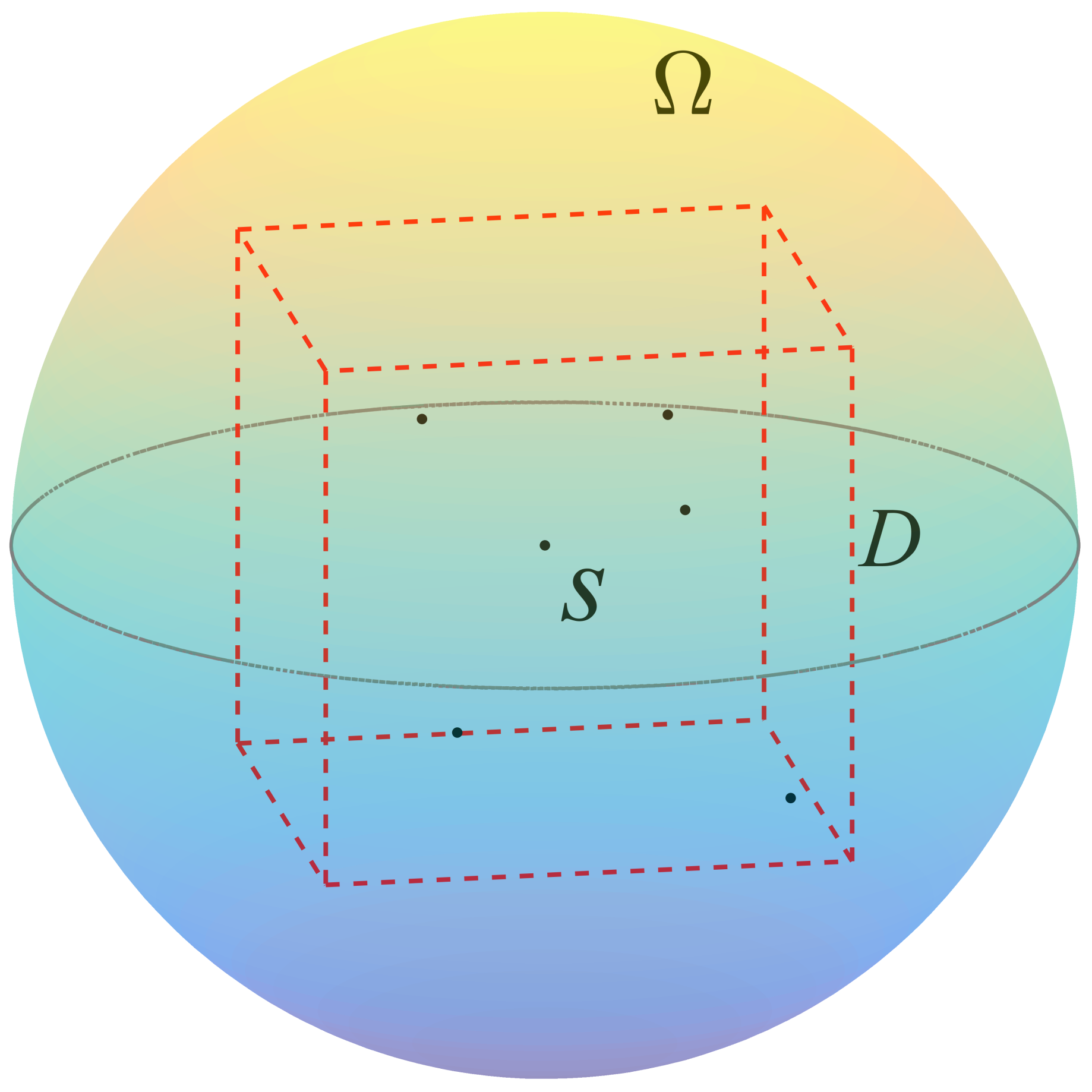} \\
  (a) & (b)  \\
 \end{tabular}
 \caption{Sketch of the measurement curve/surface, sampling region and source points. (a) A two-dimensional cross section. (b) The three-dimensional diagram.}
\end{figure}

We start with the proof of a reduced problem on a two-dimensional cross section. The reduced problem is easier to analyze and the corresponding numerical experiments are more intuitive.

Assume that the collection of source points $S'=\{s_j, 1\leq j\leq K\}$ is on a plane $\mathbb{P}\subset\mathbb{R}^3$. Based on this a priori information, let $\Omega'\subset \mathbb{P}$ be a bounded convex open region such that $S'\subset \Omega'$. Furthermore, we can choose a sampling region $D'\subset \Omega'$ such that $\partial\Omega'\cap D'=\varnothing$ and $S'\subset D'$.

Then the problem is reduced to: Determine the location of the source points on the two-dimensional cross section $\mathbb{P}$ with the indicator function
\begin{equation}\label{indicator function2}
I'(z)=\displaystyle{\int_{\mathbb{R}}\int_{\partial\Omega'}u\left(x,t\right) \Big(G(x,t;z)\ast{\lambda(t)}\Big){\rm d}s(x){\rm d}t}.
\end{equation}

\begin{lemma}
Assume that $\Omega'$ is a circle with radius $R$ on a two-dimensional cross section $\mathbb{P}\in\mathbb{R}^3$, $s$ and $z$ are two different points in $\Omega'$ and $c>0$ is the sound speed. Then we have
\begin{equation*}
0<F(z,s):=\displaystyle{\int_{\partial\Omega'}\delta \left(c^{-1}|x-s|-c^{-1}|x-z|\right){\rm d}s(x)}< +\infty
\end{equation*}
and
\begin{equation*}
\lim\limits_{z\rightarrow s}F(z,s)=+\infty.
\end{equation*}
Moreover, for any point $s\in\Omega'$, there exists a constant $A_1(s)>0$ such that
\begin{equation*}
\lim\limits_{z\rightarrow s}|z-s|F(z,s)=A_1(s).
\end{equation*}
\end{lemma}

\begin{proof}
Noting that $z\neq s$, without loss of generality, we can adopt Cartesian coordinates $(x_1,x_2)$ in the two-dimensional cross section $\mathbb{P}$ such that $s=(a, b_1)$, $z=(a, b_2)$,  $-R<a<R$, $-\sqrt{R^2-a^2}<b_1<b_2<\sqrt{R^2-a^2}$
and $\Omega'$ can be parameterized as
\begin{equation*}
 \Omega':
 \left\{
 \begin{array}{lll}
x_1=R\cos\theta, \\
x_2=R\sin\theta,
 \end{array}
\quad\quad 0\leq\theta\leq 2\pi. \right.
 \end{equation*}

Denote $g(\theta)=|x-s|-|x-z|$, $x\in\partial\Omega'$. A direct computation implies $g(\theta)$ has two zero points $\theta_1$ and $\theta_2$ in $[0,2\pi)$ which satisfy
\begin{equation*}
\begin{array}{rl}
\sin \theta_1 = &\sin \theta_2 = \displaystyle{\frac{b_1+b_2}{2R}},\\
\cos \theta_1 = &\displaystyle{\frac{1}{2R}\sqrt{4R^2-(b_1+b_2)^2},\quad \cos \theta_2 = -\frac{1}{2R}\sqrt{4R^2-(b_1+b_2)^2}}.
\end{array}
\end{equation*}
Moreover, we have
\begin{equation*}
g'(\theta_1)=\displaystyle{\frac{(b_2-b_1)\sqrt{4R^2-(b_1+b_2)^2}} {2\sqrt{R^2+a^2-b_1b_2-a\sqrt{4R^2-(b_1+b_2)^2}}}}
\end{equation*}
and
\begin{equation*}
g'(\theta_2)=-\displaystyle{\frac{(b_2-b_1)\sqrt{4R^2-(b_1+b_2)^2}} {2\sqrt{R^2+a^2-b_1b_2+a\sqrt{4R^2-(b_1+b_2)^2}}}}.
\end{equation*}
Then the property of the Dirac delta function implies
\begin{equation*}
\begin{array}{rl}
F(z,s) = &\displaystyle{c\int_{\partial\Omega'}\delta \left(|x-s|-|x-z|\right){\rm d}s(x)} = \displaystyle{c\int_0^{2\pi}\delta \left(g(\theta)\right)R{\rm d}\theta}\\
= &\displaystyle{cR\int_0^{2\pi} \frac{\delta \left(\theta-\theta_1\right)}{|g'(\theta_1)|}{\rm d}\theta} + \displaystyle{cR\int_0^{2\pi} \frac{\delta \left(\theta-\theta_2\right)}{|g'(\theta_2)|}{\rm d}\theta}\\
= & \displaystyle{\frac{C_1+C_2}{C_3}}\cdot \displaystyle{\frac{1}{b_2-b_1}},
\end{array}
\end{equation*}
where
\begin{equation*}
C_1=2cR\sqrt{R^2+a^2-b_1b_2-a\sqrt{4R^2-(b_1+b_2)^2}},
\end{equation*}
\begin{equation*}
C_2=2cR\sqrt{R^2+a^2-b_1b_2+a\sqrt{4R^2-(b_1+b_2)^2}}
\end{equation*}
and
\begin{equation*}
C_3=\sqrt{4R^2-(b_1+b_2)^2}
\end{equation*}
are constants depending on $a$, $b_1$ and $b_2$.

It is obvious that $0<C_1,C_2<2\sqrt{2}cR^2$ and $0<C_3<2R$. Then we have
\begin{equation*}
0<F(z,s)< +\infty
\end{equation*}
and
\begin{equation*}
\lim\limits_{z\rightarrow s}F(z,s)=+\infty.
\end{equation*}
Moreover, for any point $s\in\Omega'$, we have
\begin{equation*}
\lim\limits_{z\rightarrow s}|z-s|F(z,s)= \lim\limits_{z\rightarrow s}\displaystyle{\frac{C_1+C_2}{C_3}},
\end{equation*}
where $\lim\limits_{z\rightarrow s}\displaystyle{\frac{C_1+C_2}{C_3}}$ is a positive constant depending only on the location of $s$. Denote $A_1(s)=\lim\limits_{z\rightarrow s}\displaystyle{\frac{C_1+C_2}{C_3}}$ and the proof is completed.
\end{proof}

\begin{theorem}
Assume that $u(x,t)$ is the solution to the wave equation
\begin{equation*}
c^{-2}\partial_{tt}u(x,t)-\triangle{u(x,t)}=\sum_{j=1}^{K} \lambda(t)\tau_{j}\delta(x-s_j)
\end{equation*}
with a set of source points $S'=\left\{s_j,\,j=1,2,\ldots,K\right\}$ on a two-dimensional cross section $\mathbb{P}\subset\mathbb{R}^3$, $K\in\mathbb{N}^{*}$, $\tau_j>0$ are the intensities of the point sources and $\lambda(t)=\delta(t)$ is the signal function. Assume that $\Omega'\subset\mathbb{P}$ is a circle with radius $R$ such that $S'\subset \Omega'$. The sampling region $D'\subset \Omega'$ satisfies $S'\subset D'$ and $D'\cap\partial\Omega'=\varnothing$. Then the indicator function $I'(z)$ defined by (\ref{indicator function2}) satisfies
\begin{equation*}
 I'(z)
 \left\{
 \begin{array}{ll}
 =+\infty, &z\in S', \\
 < +\infty, &z\in D'\setminus S'.
 \end{array}
 \right.
 \end{equation*}
Moreover, we have
\begin{equation*}
\lim\limits_{d(z,S')\rightarrow 0} I'(z)=+\infty,
\end{equation*}
where $d(z,S')=\min\limits_{j=1,2,\ldots,K} |z-s_j|$.
\end{theorem}

\begin{proof}
\noindent For $z\in S'$, denote $z=s_k, 1\leq k\leq K$. Considering (\ref{point solution}) and (\ref{indicator function2}), we can get
\begin{equation*}
\begin{array}{rl}
I'(z)&=\displaystyle{\int_{\mathbb{R}}\int_{\partial\Omega'}\left( \sum\limits_{j=1}^{K} \tau_j G\left(x,t,s_j\right)\ast\lambda(t)\right)\Big(G(x,t;s_k) \ast\lambda(t)\Big){\rm d}s(x){\rm d}t}\\
&=\displaystyle{\int_{\mathbb{R}}\int_{\partial\Omega'}\left(\sum\limits_{j=1}^{K} \tau_j\frac{\delta\left(t-c^{-1}|x-s_j|\right)}{4\pi|x-s_j|}\right) \frac{\delta\left(t-c^{-1}|x-s_k|\right)}{4\pi|x-s_k|}{\rm d}s(x){\rm d}t}\\
&\geq\displaystyle{\int_{\mathbb{R}}\int_{\partial\Omega'}\tau_k\frac{\delta^2\left(t-c^{-1}|x-s_k|\right)}{16\pi^2|x-s_k|^2}{\rm d}s(x){\rm d}t}\\
&=+\infty.
\end{array}
\end{equation*}

Denote $d_1=\inf\limits_{x\in\partial\Omega',z\in D'}|x-z|$. For $z\in D'\setminus S'$, we have
\begin{equation*}
\begin{array}{rl}
I'(z)&=\displaystyle{\int_{\mathbb{R}}\int_{\partial\Omega'}\left(\sum\limits_{j=1}^{K}\tau_j\frac{\delta\left(t-c^{-1}|x-s_j|\right)}{4\pi|x-s_j|}\right)\frac{\delta\left(t-c^{-1}|x-z|\right)}{4\pi|x-z|}{\rm d}s(x){\rm d}t}\\
&=\displaystyle{\sum\limits_{j=1}^{K}\tau_j\int_{\partial\Omega'}\frac{\delta\left(c^{-1}|x-s_j|-c^{-1}|x-z|\right)}{16\pi^2|x-s_j||x-z|}{\rm d}s(x)}\\
&\leq\displaystyle{\frac{1}{16\pi^2d_1^2}\sum\limits_{j=1}^{K}\tau_j \int_{\partial\Omega'}\delta\left(c^{-1}|x-s_j|-c^{-1}|x-z|\right){\rm d}s(x)}.
\end{array}
\end{equation*}
Then Lemma 3.1 implies $I'(z)<+\infty$.

If $d(z,S')\rightarrow 0$, without lose of generality, assume that $z\rightarrow s_k\in S'$. Denote $d_2=\sup\limits_{x\in\partial\Omega',z\in D'}|x-z|$. Combining with Lemma 3.1, we have
\begin{equation*}
\begin{array}{rl}
\lim\limits_{d(z,S')\rightarrow 0}I'(z)&\geq\displaystyle{\tau_k\lim\limits_{z\rightarrow s_k}\int_{\partial\Omega'} \frac{\delta\left(c^{-1}|x-s_k|-c^{-1}|x-z| \right)}{16\pi^2|x-s_k||x-z|}{\rm d}s(x)}\\
&\geq\displaystyle{\frac{\tau_k}{16\pi^2 d_2^2}\lim\limits_{z\rightarrow s_k}\int_{\partial\Omega'} \delta\left(c^{-1}|x-s_k|-c^{-1}|x-z| \right){\rm d}s(x)}\\
&=+\infty,
\end{array}
\end{equation*}
which completes the proof.
\end{proof}

Then we consider the proof in $\mathbb{R}^3$ without a priori information. Conclusions similar to the two-dimensional cross section case can be reached.

\begin{lemma}
Assume that $\Omega$ is a spherical region of radius $R$, $s$ and $z$ are two different points in $\Omega$ and $c>0$ is the sound speed. Then we have
\begin{equation*}
0<F(z,s):=\displaystyle{\int_{\partial\Omega}\delta \left(c^{-1}|x-s|-c^{-1}|x-z|\right){\rm d}s(x)}<+\infty
\end{equation*}
and
\begin{equation*}
\lim\limits_{z\rightarrow s}F(z,s)=+\infty.
\end{equation*}
Moreover, for any point $s\in\Omega$, there exists a constant $A_2(s)>0$ such that
\begin{equation*}
\lim\limits_{z\rightarrow s}|z-s|F(z,s)=A_2(s).
\end{equation*}
\end{lemma}

\begin{proof}
Noting that $s\neq z$, adopt Cartesian coordinates $(x_1,x_2,x_3)$ such that $s=(a,b,c_1)$, $z=(a,b,c_2)$, $a^2+b^2<R^2$, $-\sqrt{R^2-a^2-b^2}<c_1<c_2<\sqrt{R^2-a^2-b^2}$ and $\Omega$ can be parameterized as
\begin{equation*}
 \Omega:
 \left\{
 \begin{array}{ll}
x_1=R\sin\varphi\cos\theta,  \\
x_2=R\sin\varphi\sin\theta,  \\
x_3=R\cos\varphi,
 \end{array}
 \right.\quad \quad 0\leq\varphi\leq \pi, 0\leq\theta\leq 2\pi.
 \end{equation*}

Similar to the proof of Lemma 3.1, denote $g(\varphi,\theta)=|x-s|-|x-z|$. Then we have $g(\varphi,\theta)=0$ when $\varphi=\varphi_0$ which satisfies
\begin{equation*}
\cos \varphi_0=\displaystyle{\frac{c_1+c_2}{2R}},\quad\sin \varphi_0=\displaystyle{\frac{1}{2R}\sqrt{4R^2-(c_1+c_2)^2}}.
\end{equation*}
Moreover, we have
\begin{equation*}
\displaystyle{\frac{\partial g}{\partial \varphi}(\varphi_0,\theta) = \frac{(c_1-c_2)\sqrt{4R^2-(c_1+c_2)^2}} {2\sqrt{R^2+a^2+b^2-c_1c_2 +(a\cos\theta-b\sin\theta)\sqrt{4R^2-(c_1+c_2)^2}}}}.
\end{equation*}
The property of the Dirac delta function implies
\begin{equation*}
\begin{array}{rl}
F(z,s) = & \displaystyle{c\int_0^{2\pi}\int_0^{\pi}\delta \left(g(\varphi,\theta)\right)R^2\sin\varphi{\rm d}\varphi {\rm d}\theta}\\
= &\displaystyle{cR^2\int_0^{2\pi} \int_0^{\pi}\frac{\delta \left(\varphi-\varphi_0\right)}{\left|\frac{\partial g}{\partial \varphi}(\varphi_0,\theta)\right|}\sin{\varphi}{\rm d}\varphi {\rm d}\theta}\\
= & \displaystyle{C_4}\cdot \displaystyle{\frac{1}{c_2-c_1}},
\end{array}
\end{equation*}
where
\begin{equation*}
C_4=cR\int_0^{2\pi} \displaystyle{\sqrt{R^2+a^2+b^2-c_1c_2 +(a\cos\theta-b\sin\theta)\sqrt{4R^2-(c_1+c_2)^2}}} {\rm d}\theta
\end{equation*}
is a positive constant depending on $a$, $b$, $c_1$ and $c_2$. Then a similar discussion as that in the proof of Lemma 3.1 implies the conclusion.
\end{proof}

\begin{theorem}
Assume that $u(x,t)$ is the solution to the wave equation
\begin{equation*}
c^{-2}\partial_{tt}u(x,t)-\triangle{u(x,t)}=\sum_{j=1}^{K} \lambda(t)\tau_{j}\delta(x-s_j)
\end{equation*}
with a set of source points $S=\left\{s_j,\,j=1,2,\ldots,K\right\}$ in $\mathbb{R}^3$, $K\in\mathbb{N}^{*}$, $\tau_i>0$ are the intensities of the point sources and $\lambda(t)=\delta(t)$ is the signal function. Assume that $\Omega$ is a spherical region of radius $R$ such that $S\subset \Omega$. The sampling region $D\subset \Omega$ satisfies $S\subset D$ and $D\cap\partial\Omega=\varnothing$. Then the indicator function $I(z)$ defined by (\ref{indicator function}) satisfies
\begin{equation*}
 I(z)
 \left\{
 \begin{array}{ll}
 =+\infty, &z\in S, \\
 < +\infty, &z\in D\setminus S.
 \end{array}
 \right.
 \end{equation*}
Moreover, we have
\begin{equation*}
\lim\limits_{d(z,S)\rightarrow 0} I(z)=+\infty,
\end{equation*}
where $d(z,S)=\min\limits_{j=1,2,\ldots,K} |z-s_j|$.
\end{theorem}

\begin{proof}
For $z=s_k\in S, 1\leq k\leq K$, similar to the proof of Theorem 3.2, we have
\begin{equation*}
I(z)\geq\displaystyle{\int_{\mathbb{R}}\int_{\partial\Omega}\tau_k\frac{\delta^2\left(t-c^{-1}|x-s_k|\right)}{16\pi^2|x-s_k|^2}{\rm d}s(x){\rm d}t}=+\infty.
\end{equation*}

Denote $d_1=\inf\limits_{x\in\partial\Omega, z\in D}|x-z|$ and $d_2=\sup\limits_{x\in\partial\Omega, z\in D}|x-z|$. For $z\in D\setminus S$, combining with Lemma 3.3, we can get
\begin{equation*}
\begin{array}{rl}
I(z)&=\displaystyle{\sum\limits_{j=1}^{K}\int_{\partial\Omega}\tau_j\frac{\delta\left(c^{-1}|x-s_j|-c^{-1}|x-z|\right)}{16\pi^2|x-s_j||x-z|}{\rm d}s(x)}\\
&\leq\displaystyle{\frac{1}{16\pi^2d_1^2}\sum\limits_{j=1}^{K}\tau_j\int_{\partial\Omega}\delta\left(c^{-1}|x-s_j|-c^{-1}|x-z|\right){\rm d}s(x)}\\
&<+\infty.
\end{array}
\end{equation*}
Moreover, if $d(z,S)\rightarrow 0$, assume that $z\rightarrow s_k\in S$. Then we have
\begin{equation*}
\lim\limits_{d(z,S)\rightarrow 0}I(z)\geq\displaystyle{\frac{\tau_k}{16\pi^2 d_2^2}\lim\limits_{z\rightarrow s_k}\int_{\partial\Omega} \delta\left(c^{-1}|x-s_k|-c^{-1}|x-z| \right){\rm d}s(x)} =+\infty,
\end{equation*}
which completes the proof.
\end{proof}

\subsection{The reconstruction of the sources on a curve}
\begin{figure}[!ht]
 \center
 \scriptsize
 \begin{tabular}{cc}
  \includegraphics[width=4.5cm]{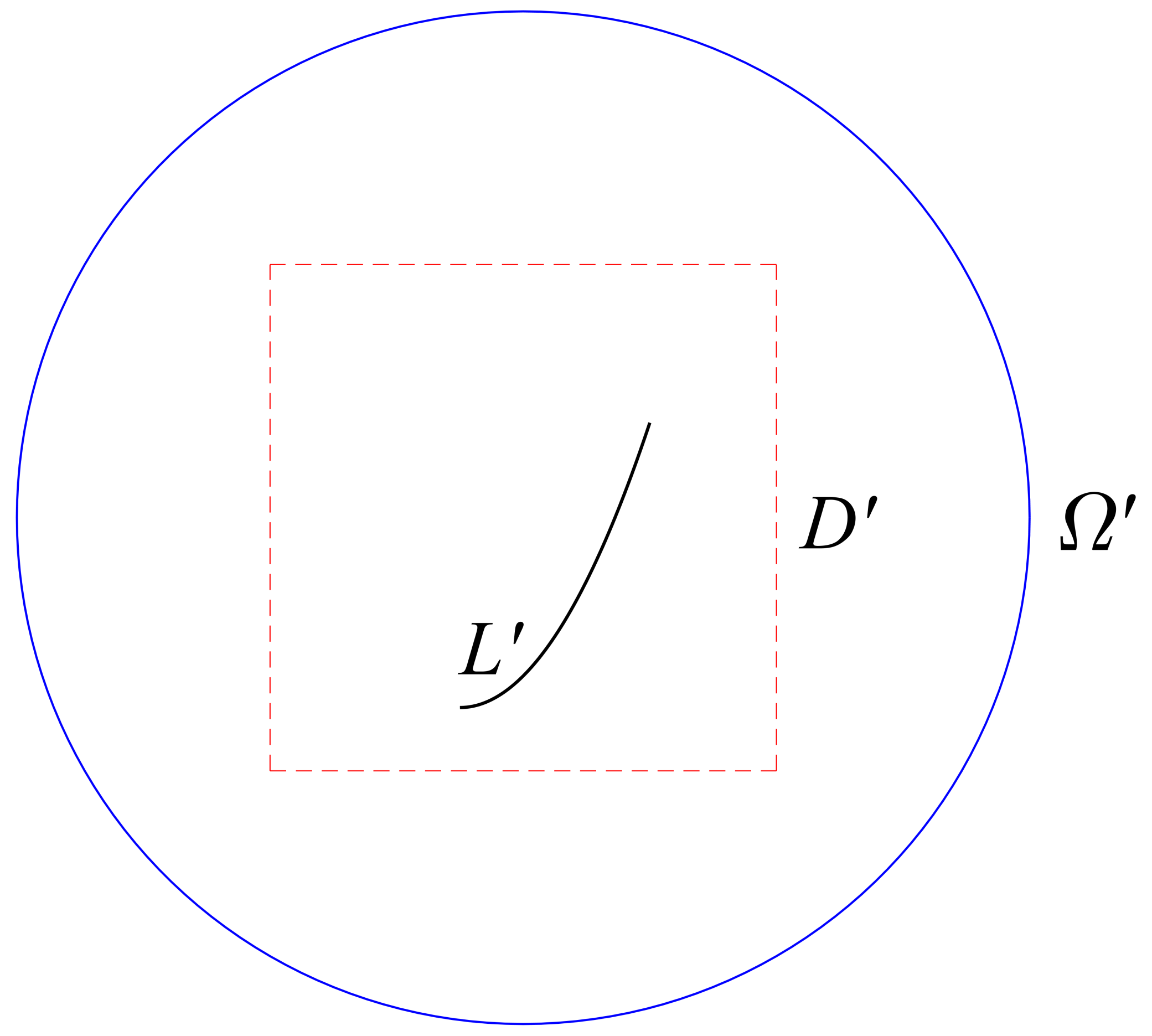}~~~ & ~~~ \includegraphics[width=4.5cm]{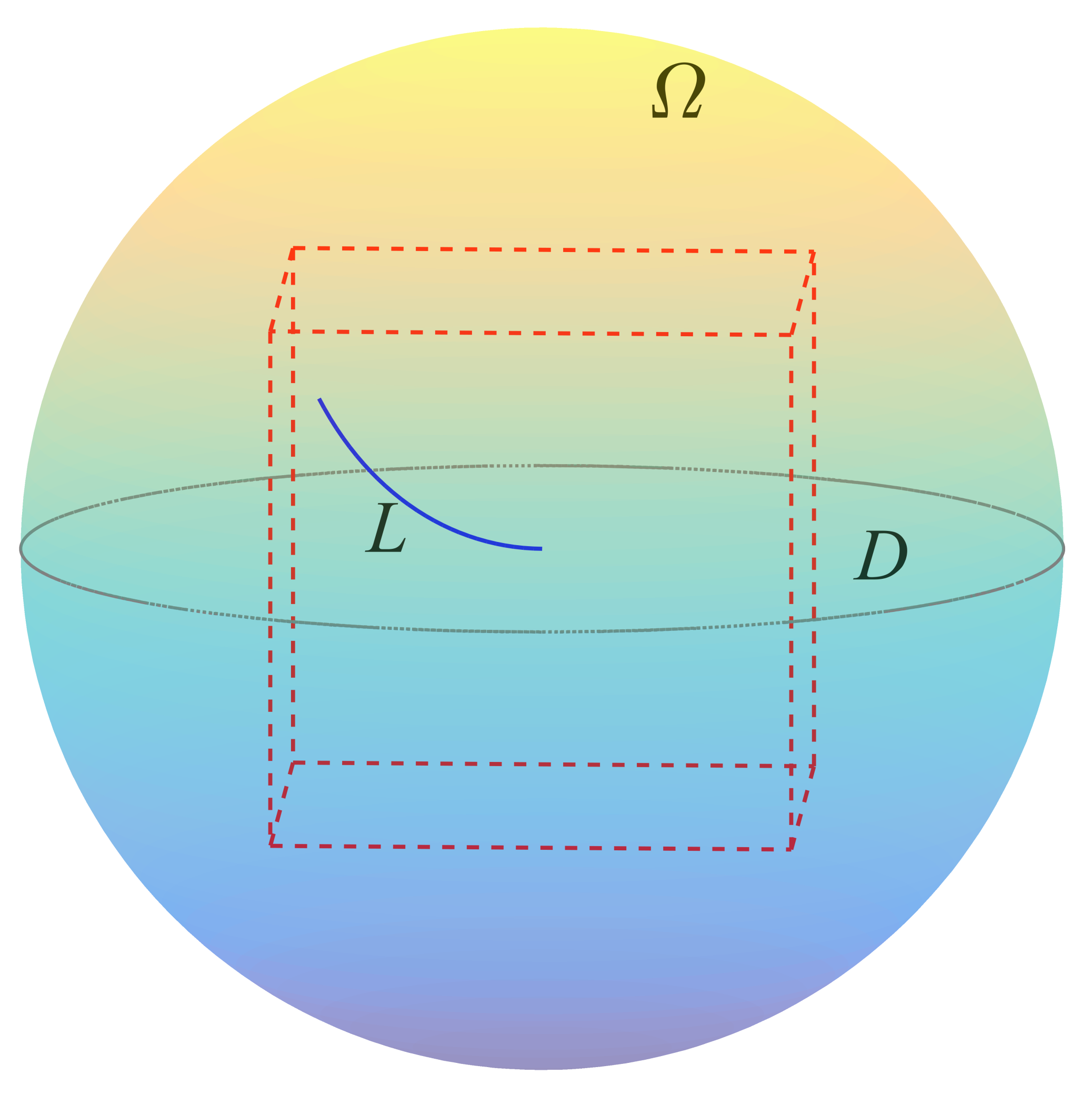} \\
  (a) & (b)  \\
 \end{tabular}
 \caption{Sketch of the measurement curve/surface, sampling region and source curve. (a) A two-dimensional cross section. (b) The three-dimensional diagram.}
\end{figure}

For the scattering problem with the sources on a curve $L$, we still start with a proof in a two-dimensional cross section.

\begin{theorem}
Assume that $u(x,t)$ is the solution to the wave equation
\begin{equation*}
c^{-2}\partial_{tt}u(x,t)-\triangle{u(x,t)}= \lambda(t)\tau(x)\delta(L')
\end{equation*}
with a piecewise smooth source curve $L'$ on a two-dimensional cross section $\mathbb{P}\subset\mathbb{R}^3$, $\tau(x)\geq\tau_0> 0$ is a bounded intensity function and $\lambda(t)=\delta(t)$ is the signal function. Assume that $\Omega'\subset\mathbb{P}$ is a circle with radius $R$ such that $L'\subset\Omega'$. The sampling region $D'\subset \Omega'$ satisfies $L'\subset D'$ and $D'\cap\partial\Omega'=\varnothing$. Then the indicator function $I'(z)$ defined by (\ref{indicator function2}) satisfies
\begin{equation*}
 I'(z)
 \left\{
 \begin{array}{ll}
 = +\infty, &z\in L', \\
 < +\infty, &z\in D'\setminus L'.
 \end{array}
 \right.
 \end{equation*}
\end{theorem}

\begin{proof}
Denote $d_1=\inf\limits_{x\in\partial\Omega', z\in D'}|x-z|$ and $d_2=\sup\limits_{x\in\partial\Omega', z\in D'}|x-z|$. For $z=y_0\in L'$, combining with (\ref{line solution}) and (\ref{indicator function2}), we have
\begin{equation*}
\begin{array}{rl}
I'(z)&=\displaystyle{\int_{\mathbb{R}}\int_{\partial\Omega'}\int_{L'}  \tau(y)\Big(G\left(x,t;y\right)\ast\lambda(t)\Big) {\rm d}s(y)\cdot\Big(G(x,t;y_0) \ast\lambda(t)\Big){\rm d}s(x){\rm d}t}\\
&=\displaystyle{\int_{L'} \tau(y) \int_{\partial\Omega'} \int_{\mathbb{R}} \frac{\delta(t-c^{-1}|x-y|)}{4\pi|x-y|} \frac{\delta(t-c^{-1}|x-y_0|)}{4\pi|x-y_0|} {\rm d}t} {\rm d}s(x) {\rm d}s(y)\\
&=\displaystyle{\int_{L'} \tau(y) \int_{\partial\Omega'} \frac{\delta(c^{-1}|x-y|-c^{-1}|x-y_0|)}{16\pi^2|x-y||x-y_0|} {\rm d}s(x) {\rm d}s(y)}\\
&\geq\displaystyle{\frac{\tau_0}{16\pi^2d_2^2}\int_{L'}  \int_{\partial\Omega'} \delta(c^{-1}|x-y|-c^{-1}|x-y_0|) {\rm d}s(x) {\rm d}s(y)}.
\end{array}
\end{equation*}
Lemma 3.1 implies
\begin{equation*}
\displaystyle{F(y):=\int_{\partial\Omega'} \delta(c^{-1}|x-y|-c^{-1}|x-y_0|) {\rm d}s(x)}
\end{equation*}
has the same singularity as $\displaystyle{\frac{1}{|y-y_0|}}$ when $y\rightarrow y_0$. Then we can get
\begin{equation*}
\displaystyle{\int_{L'} F(y) {\rm d}s(y)=+\infty}
\end{equation*}
and $I'(z)=+\infty$.

For $z\in D'\setminus L'$, combining with Lemma 3.1, we have
\begin{equation*}
\begin{array}{rl}
I'(z)&=\displaystyle{\int_{\mathbb{R}}\int_{\partial\Omega'}\int_{L'} \tau(y) \frac{\delta(t-c^{-1}|x-y|)}{4\pi|x-y|} \frac{\delta(t-c^{-1}|x-z|)}{4\pi|x-z|} {\rm d}s(y){\rm d}s(x){\rm d}t}\\
&=\displaystyle{\int_{L'}\tau(y) \int_{\partial\Omega'}\frac{\delta(c^{-1}|x-y|-c^{-1}|x-z|)} {16\pi^2|x-y||x-z|}{\rm d}s(x){\rm d}s(y)}\\
&\leq\displaystyle{\frac{1}{16\pi^2d_1^2} \int_{L'}\tau(y)} \displaystyle{\int_{\partial\Omega'}\delta(c^{-1}|x-y|-c^{-1}|x-z|) {\rm d}s(x){\rm d}s(y)}\\
&<+\infty,
\end{array}
\end{equation*}
which completes the proof.
\end{proof}

Then we give the proof in $\mathbb{R}^3$ without a priori information.

\begin{theorem}
Assume that $u(x,t)$ is the solution to the wave equation
\begin{equation*}
c^{-2}\partial_{tt}u(x,t)-\triangle{u(x,t)}= \lambda(t)\tau(x)\delta_L(x)
\end{equation*}
with a piecewise smooth source curve $L\subset\mathbb{R}^3$, $\tau(x)\geq\tau_0>0$ is a bounded intensity function and $\lambda(t)=\delta(t)$ is the signal function. Assume that $\Omega$ is a spherical region of radius $R$ such that $L\subset\Omega$. The sampling region $D\subset \Omega$ satisfies $L\subset D$ and $D\cap\partial\Omega=\varnothing$. Then the indicator function $I(z)$ defined by (\ref{indicator function}) satisfies
\begin{equation*}
 I(z)
 \left\{
 \begin{array}{ll}
 = +\infty, &z\in L, \\
 < +\infty, &z\in D\setminus L.
 \end{array}
 \right.
 \end{equation*}
\end{theorem}

\begin{proof}
The proof is similar to that of Theorem 3.5. Denote $d_1=\inf\limits_{x\in\partial\Omega, z\in D}|x-z|$ and $d_2=\sup\limits_{x\in\partial\Omega, z\in D}|x-z|$. For $z=y_0\in L$, combining with (\ref{line solution}), (\ref{indicator function}) and Lemma 3.3, we have
\begin{equation*}
I(z)\geq\displaystyle{\frac{\tau_0}{16\pi^2d_2^2}\int_{L}  \int_{\partial\Omega} \delta(c^{-1}|x-y|-c^{-1}|x-y_0|) {\rm d}s(x) {\rm d}s(y)}=+\infty.
\end{equation*}
For $z\in D\setminus L$, Lemma 3.3 implies
\begin{equation*}
I(z)\leq\displaystyle{\frac{1}{16\pi^2d_1^2} \int_{L}\tau(y)} \displaystyle{\int_{\partial\Omega}\delta(c^{-1}|x-y|-c^{-1}|x-z|) {\rm d}s(x){\rm d}s(y)}<+\infty,
\end{equation*}
which completes the proof.
\end{proof}

The direct sampling method to solve the inverse source problems with either point sources or sources of form $\lambda(t)\tau(x)\delta_L(x)$ is shown in Algorithm 1.

\begin{table}[!bhtp]
  \small
  \centering
  \begin{tabular}{ll}
    \toprule
    \multicolumn{2}{c}{\textbf{Algorithm 1}. The direct sampling method.}\\
    \midrule
    \textbf{Step 1}~ & Choose a convex region $\Omega$, a signal function $\lambda(t)$, a set of point sources or\\
    & sources on a curve. Collect the wave data $u(x_i,t_k)$ for the sensing points \\
    & $x_i\in\partial \Omega$ $(i=1,\ldots,N_x)$ and the discrete time steps $t_k\in[0,T]~(k=1,\ldots,N_t)$, \\
    & where $T$ is a chosen terminal time.\\
    \textbf{Step 2}~ & Choose a sampling region $D\subset\Omega$ such that the sources are included in $D$\\
    & and $D\cap\partial \Omega=\varnothing$. Select a grid of sampling points $z_l~(l=1,\ldots,N_z)$ in $D$. \\
    & Compute\\
    &~~~~~~~~~~~~~~~~~~$\displaystyle{I(z_l)=\sum_{k=1}^{N_t} \sum_{i=1}^{N_x}u(x_i,t_k) G*\lambda(x_i,t_k;z_l)\Delta s(x_i)\Delta t_k}$,\\
    & where $\Delta s(x_i)$ is the area of the grid cell that contains $x_i$ in it, $\Delta t_k$ is the \\
    & length of the time steps that contains $t_k$ in it.\\
    \textbf{Step 3}~ & Mesh $I(z_l)$ on the sampling grid. The locations of the sources are given by \\
    & the locations of $z_l$ for which $I(z_l)$ are relatively large.\\
    \bottomrule
  \end{tabular}
\end{table}

\section{Numerical experiments}

In this section, numerical experiments are provided based on the direct sampling method and detailed algorithms to get quantized positions and intensities of the sources.

\subsection{The numerical scheme}

We divide the numerical scheme to reconstruct the sources into three steps. The scheme is given for the three-dimensional case, and the case of the two-dimensional cross section is similar.

Step $1$: Choose isometric sampling grid in $D\subset\mathbb{R}^3$ with step size of $\Delta l_z$ and utilize the direct sampling method to obtain $I(z_l)$ and the image of the indicator function.

Step $2$: Get the location of the sources with the local maximums of $I(z_l)$. The specific procedure is provided by the following process:

\noindent Step 2-1. Let $p=1$, chose a reference intensity $T_s>0$ and the total steps $N_p\in\mathbb{N^\ast}$.

\noindent Step 2-2. Find the global maximum value point $z_{k_p}$ of $I(z_{l})$.

\noindent Step 2-3. If $I(z_{k_p})>T_s$, the corresponding sampling point $z_{k_p}$ is regarded as a source point. If $I(z_{k_p})<T_s$, ignore the point and end the procedure.

\noindent Step 2-4. Denote $z_l=(z_{1,l},z_{2,l},z_{3,l})$. Redefine $I(z_l)=0$ if
\begin{equation*}
\max\Big(|z_{1,l}-z_{1,k_p}|, |z_{2,l}-z_{2,k_p}|, |z_{3,l}-z_{3,k_p}|\Big)\leq N_1\Delta l_z,
\end{equation*}
where $N_1\in \mathbb{N^\ast}$ is a parameter that depends on the choice of sampling grid.

\noindent Step 2-5. Redefine $p=p+1.$ If $p \leq N_p,$ go back to Step 2-2. If $p>N_p$, end the procedure.

We can get several reconstructed source points $s_j^r$ $(j=1,2,\ldots,K')$ for some $K'\in \mathbb{N^\ast}$ in this step for both the case of multiple point sources and the case of sources on a curve.

Step $3$: For the case of multiple point sources, to get the intensities of the reconstructed point sources $s_j^r$, we utilize a general numerical solution method (refer to \cite{lan48method}) as follows.

The general numerical solution method is to find the minimum of the residual square cost function
\begin{equation}\label{residual}
g=\sum_{k=1}^{N_t} \sum_{i=1}^{N_x}(u_{i,k}-u(x_i,t_k))^2,
\end{equation}
where $u_{i,k}$ is the numerical approximation of $u(x_i,t_k)$.

Let $u_j(x,t), j=1, 2, \cdots, K'$ be the solution of the wave equation with a single point source $s_j^r$ with the intensity $1$. Denote by $\eta_j$ the unknown intensity of the point source $s_j^r$, we have
\begin{equation}\label{expan}
u_{i,k}=\sum_{j=1}^{K'}\eta_{j}u_j(x_i,t_k).
\end{equation}

Combining (\ref{residual}) and (\ref{expan}), we can get
\begin{equation*}
g=\sum_{k=1}^{N_t} \sum_{i=1}^{N_x}\left(\sum_{j=1}^{K'} \eta_{j}u_j(x_i,t_k)-u(x_i,t_k)\right)^2.
\end{equation*}
Then the general numerical solution method is to find $\eta_{j}~(j=1, 2, \cdots, K')$ that minimizes $g$. We can get a system of linear equations as follows to solve the intensities $\eta_j$ of the point sources:
\begin{equation*}
\frac{\partial g}{\partial\eta_j}=0,\quad j=1,2,\cdots,K'.
\end{equation*}

Step $3'$: For the case of the curve source of the form $\lambda(t)\tau(x)\delta_L(x)$, the position of the curve source is reconstructed by fitting the reconstructed source points  $s_j^r$ obtained in Step $2$ using the polynomial fitting of the MATLAB toolbox ``cftool''.

In this paper, we choose $T_s=2.1$, $N_p=10$ and $N_1=2$ in all the experiments to reconstruct multiple point sources, and choose $T_s=0.2$, $N_p=20$ and $N_1=2$ in all the experiments to reconstruct the sources on a curve.

\subsection{Numerical examples}

We give several numerical examples to demonstrate the effectiveness of the proposed method. In all the following examples, the radiated field is collected for $t\in[0,T]$, in which the terminal time is chosen as $T=15$. The time discretization is
$$t_k=k\frac{T}{N_T},\quad k=0,1,2,\cdots,N_T,$$
where $N_T=64.$ Random noise is added to the data with
$$u_\epsilon=(1+\epsilon r)u,$$
where $\epsilon>0$ is the noise level and $r$ are uniformly distributed random numbers in $[-1,1]$.

Although the signal function is chosen as $\lambda(t)=\delta(t)$ in the theoretical analysis, the $\delta$ distribution is not a suitable choice for the numerical computation. Due to the widespread use of Gaussian signals in practical applications, in this paper, we assume that $\lambda(t)$ is a causal Gaussian modulated sine pulse of the form
\begin{equation*}
 \lambda(t)=
 \left\{
 \begin{array}{ll}
 0, &t<0, \\
 sin(\omega t)e^{-\sigma(t-t_0)^2}, \quad &t\geq 0,
 \end{array}
 \right.
 \end{equation*}
where $\omega>0$ is the center frequency, $\sigma>0$ is the frequency bandwidth parameter and $t_0$ is the time-shift parameter concerning the pulse peak time. In the following examples, we choose $\omega=12$, $\sigma=0.01$ and $t_0=3$.

The sound speed of the homogeneous background medium is chosen as $c=1$. The sensing points are selected as
 \begin{equation*}
x(i,j)=(5sin{\varphi_i}cos{\theta_j}, 5sin{\varphi_i}sin{\theta_j}, 5cos{\varphi_i}),
\end{equation*}
where $\displaystyle{\varphi_i=\frac{2i-1}{32}\pi, i=1,2,\ldots,16}$ and $\displaystyle{\theta_j=\frac{j}{8}\pi, j=0,1,\ldots,15}$. The sampling points are selected as $45\times 45\times 45$ uniform discrete points in $[-2,2]\times[-2,2]\times[-2,2]$.

\newcommand{\myfont}{\textit{\textbf{\textsf{Fancy Text}}}}

\noindent \textbf{Example 1.}
This example is concerned with the reconstruction of point sources on a two-dimensional cross section with different intensities. The source points are selected as $(1.4, 0.7, 0)$, $(0.1, 1.4, 0)$, $(-1.5, 1.1, 0)$, $(-1, -1, 0)$ and $(1, -0.7, 0)$ with relative intensities $3$, $2$, $4$, $3$ and $4$, respectively.  The reconstruction utilizing the direct sampling method is shown in Figure \ref{dian5tu}. The numerical scheme proposed in Section 4.1 is utilized to reconstruct the positions and intensities of the point sources. See Table \ref{dian5b} for the specific reconstructed data.

\begin{figure}[!ht]
 \center
 \scriptsize
 \begin{tabular}{ccc}
  \includegraphics[width=4.6cm]{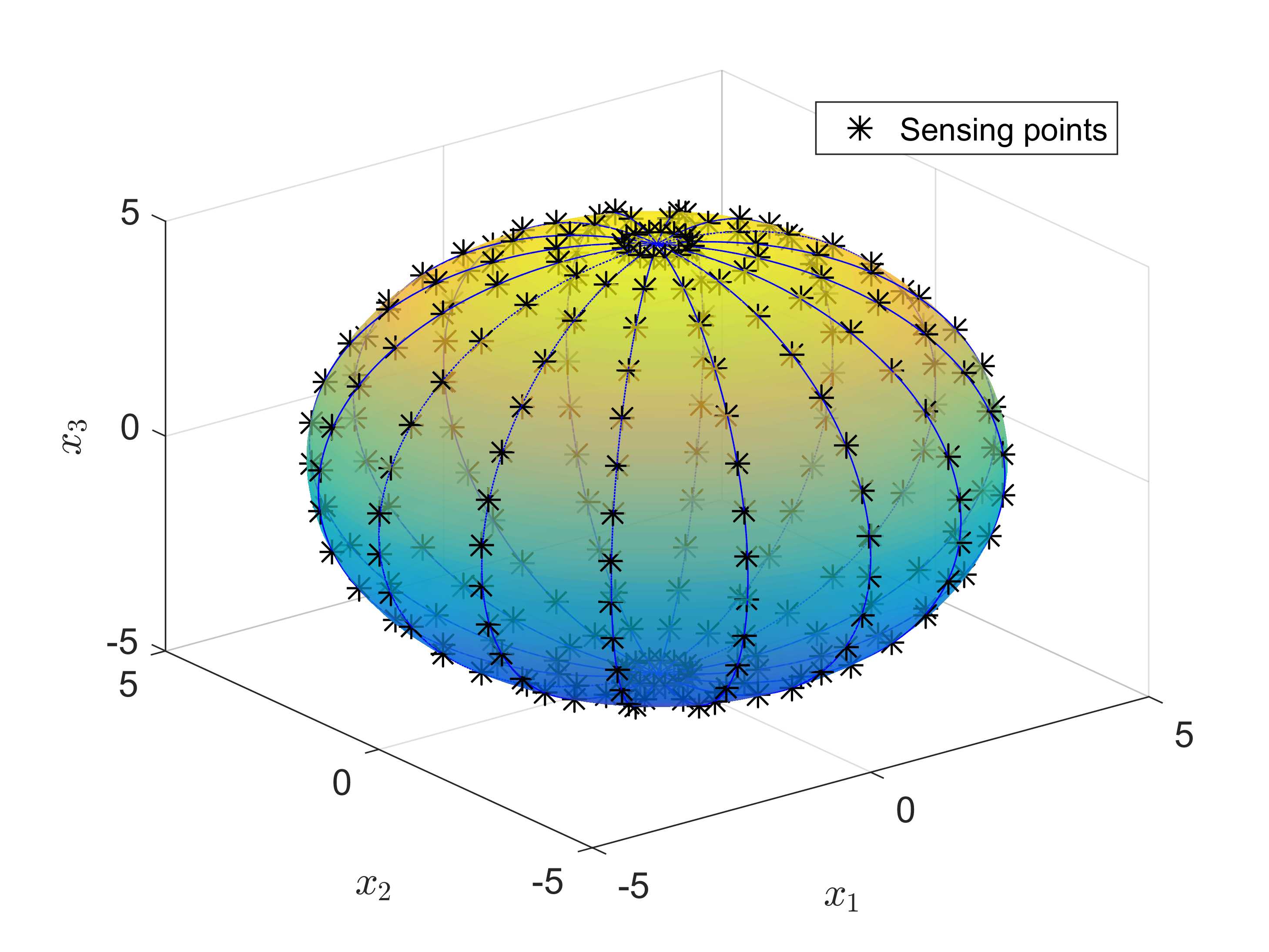} & \includegraphics[width=4.6cm]{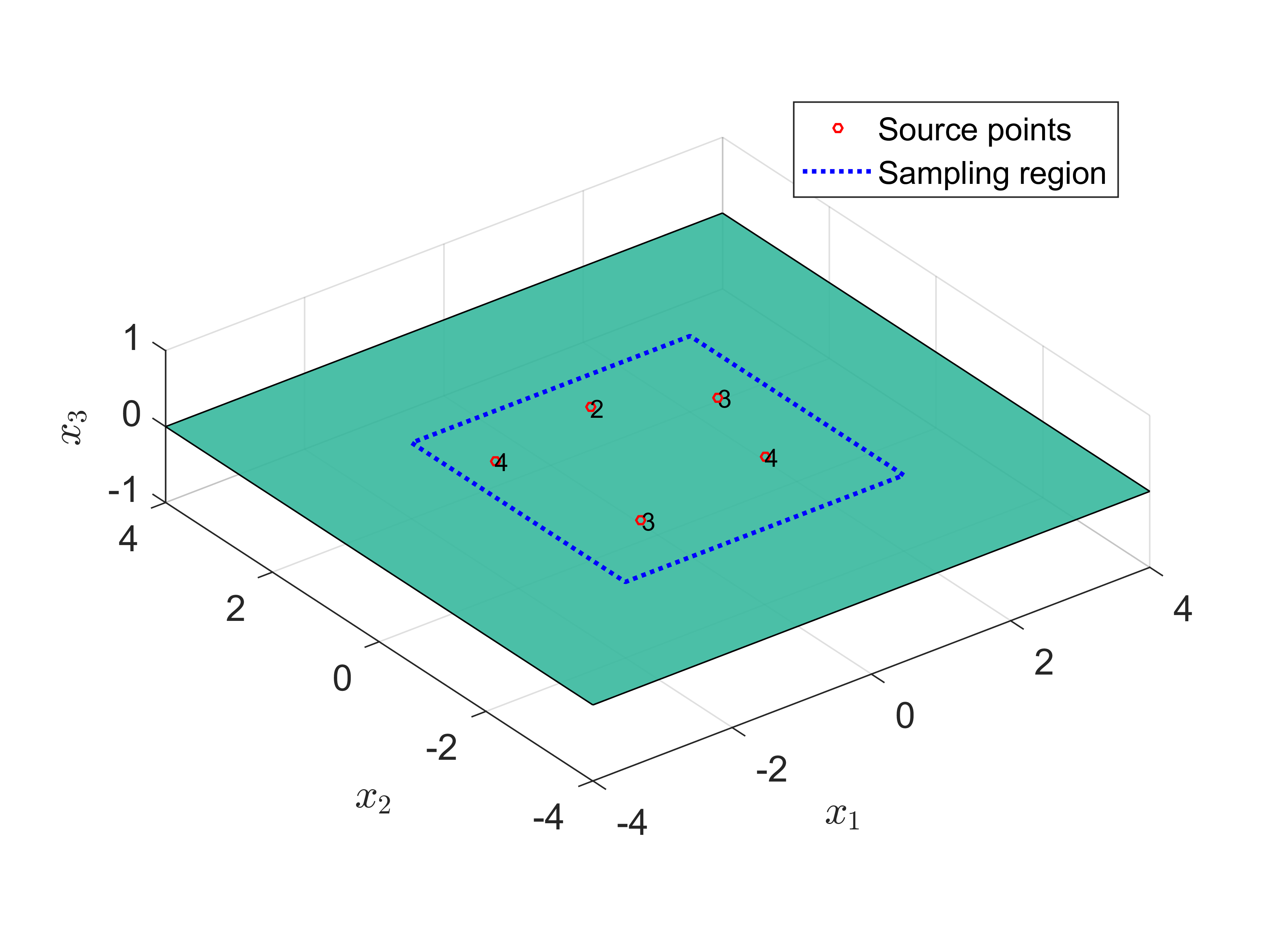} & \includegraphics[width=4.6cm]{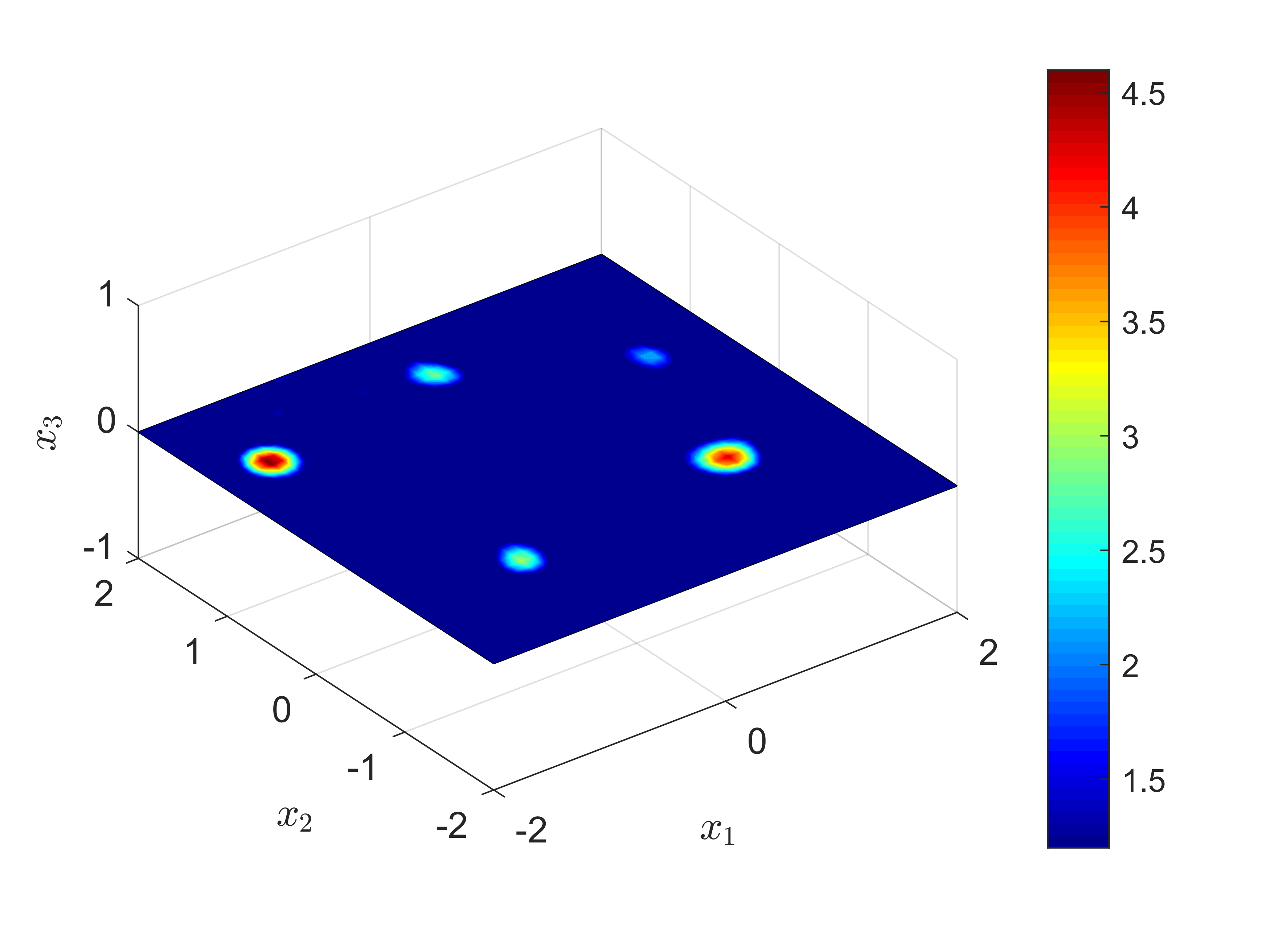} \\
  (a) & (b) & (c) \\
  \includegraphics[width=4.6cm]{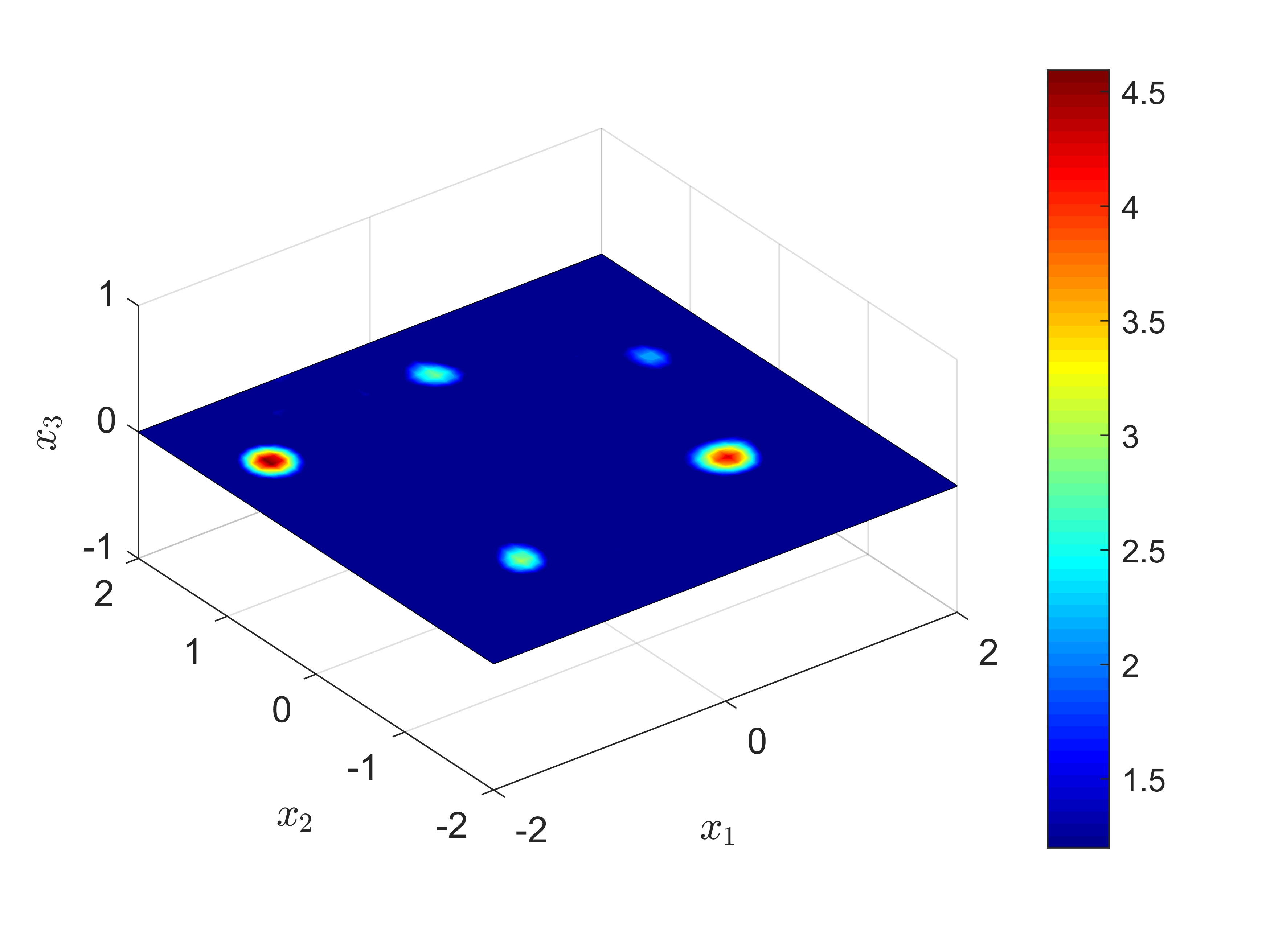} & \includegraphics[width=4.6cm]{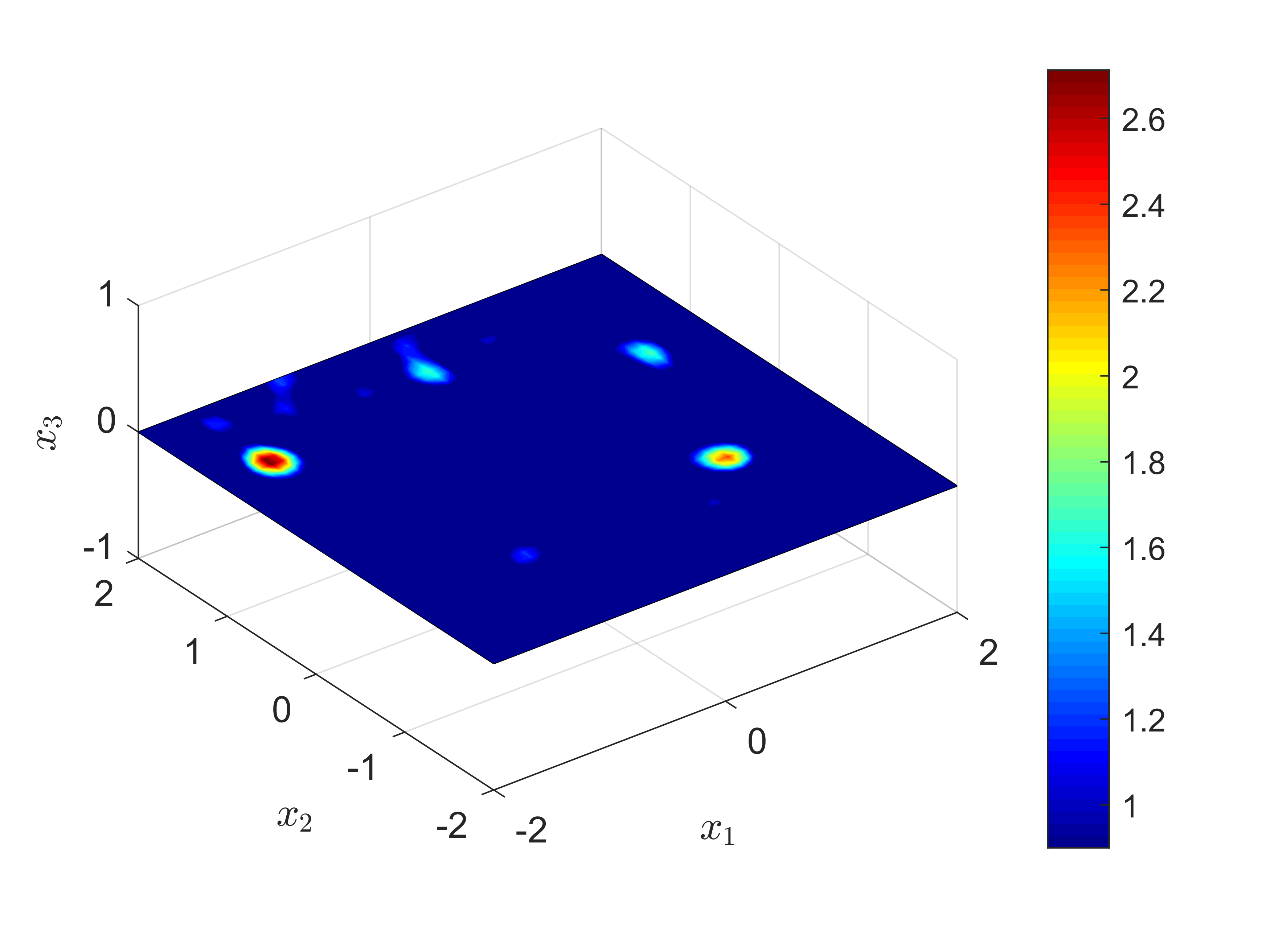} \\ 
  (d) & (e) \\
 \end{tabular}
 \caption{Reconstruction of 5 point sources with different intensities. (a) Location of the sensors. (b) Sketch of the example. (c) Reconstruction with all the sensors, $\epsilon=10\%$ (d) Reconstruction with all the sensors, $\epsilon=5\%$. (e) Reconstruction with the left half of the sensors, $\epsilon=5\%$.}
\label{dian5tu}
 \vspace{-0.5em}
\end{figure}

\begin{table}
\centering
\caption{Specific reconstructed data of Example 1.}
\footnotesize
\begin{tabular}{@{}clll}
\br
 & & The actual & Reconstruction with all \\
No. &Items & point sources & the sensors, $\epsilon=5\%$ \\
\mr
1 & Location & $(1.4, 0.7, 0)$ & $(1.4545, 0.7273, 0)$ \\
  & Intensity & 3 & 2.6141 \\
2 & Location & $(0.1, 1.4, 0)$ & $(0.0909, 1.3636, 0)$ \\
  & Intensity & 2 & 1.8030 \\
3 & Location & $(-1.5, 1.1, 0)$ & $(-1.5454, 1.0909, 0)$ \\
  & Intensity & 4 & 3.8393 \\
4 & Location & $(-1, -1, 0)$ & $(-1, -1, 0)$ \\
  & Intensity & 3 & 2.8258 \\
5 & Location & $(1, -0.7, 0)$ & $(1, -0.7273, 0)$ \\
  & Intensity & 4 & 3.8344 \\
\br
\label{dian5b}
\end{tabular}
\end{table}

\noindent\textbf{Example 2.}
In this example, the reconstruction of the point sources located at $(-1, 1, -1)$, $(0.5, 1.1, -1)$, $(1.2, -1, 0.8)$ and $(0.5, -1, 1.2)$ with respectively the intensities $2$, $3$, $2$ and $3$ is considered. The reconstruction utilizing the direct sampling method is shown in Figure \ref{dian4tu}. See Table \ref{dian4b} for specific reconstructed data.
\begin{figure}[!ht]
 \center
 \scriptsize
 \begin{tabular}{cc}
  \includegraphics[width=6.1cm]{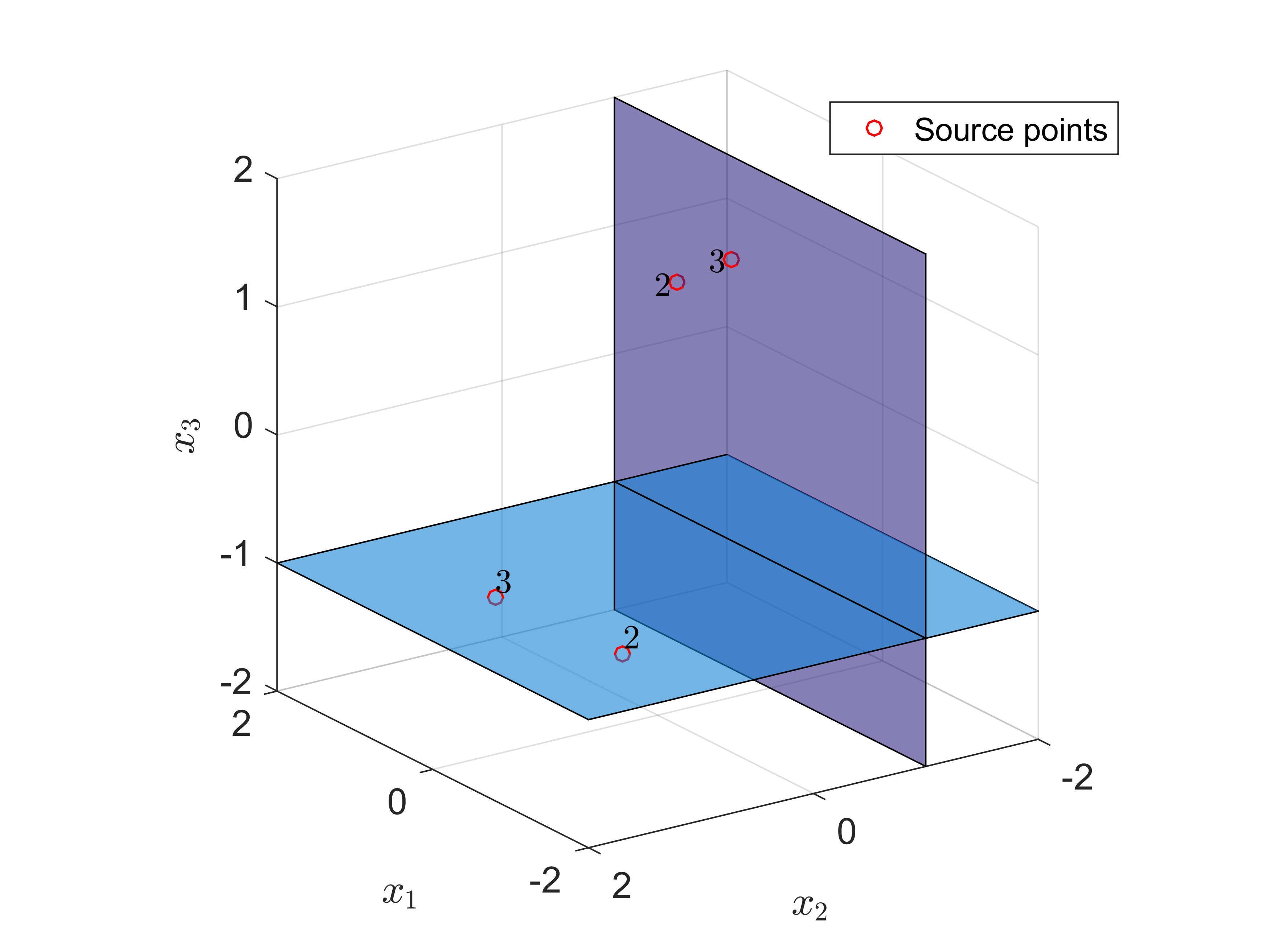} & \includegraphics[width=6.1cm]{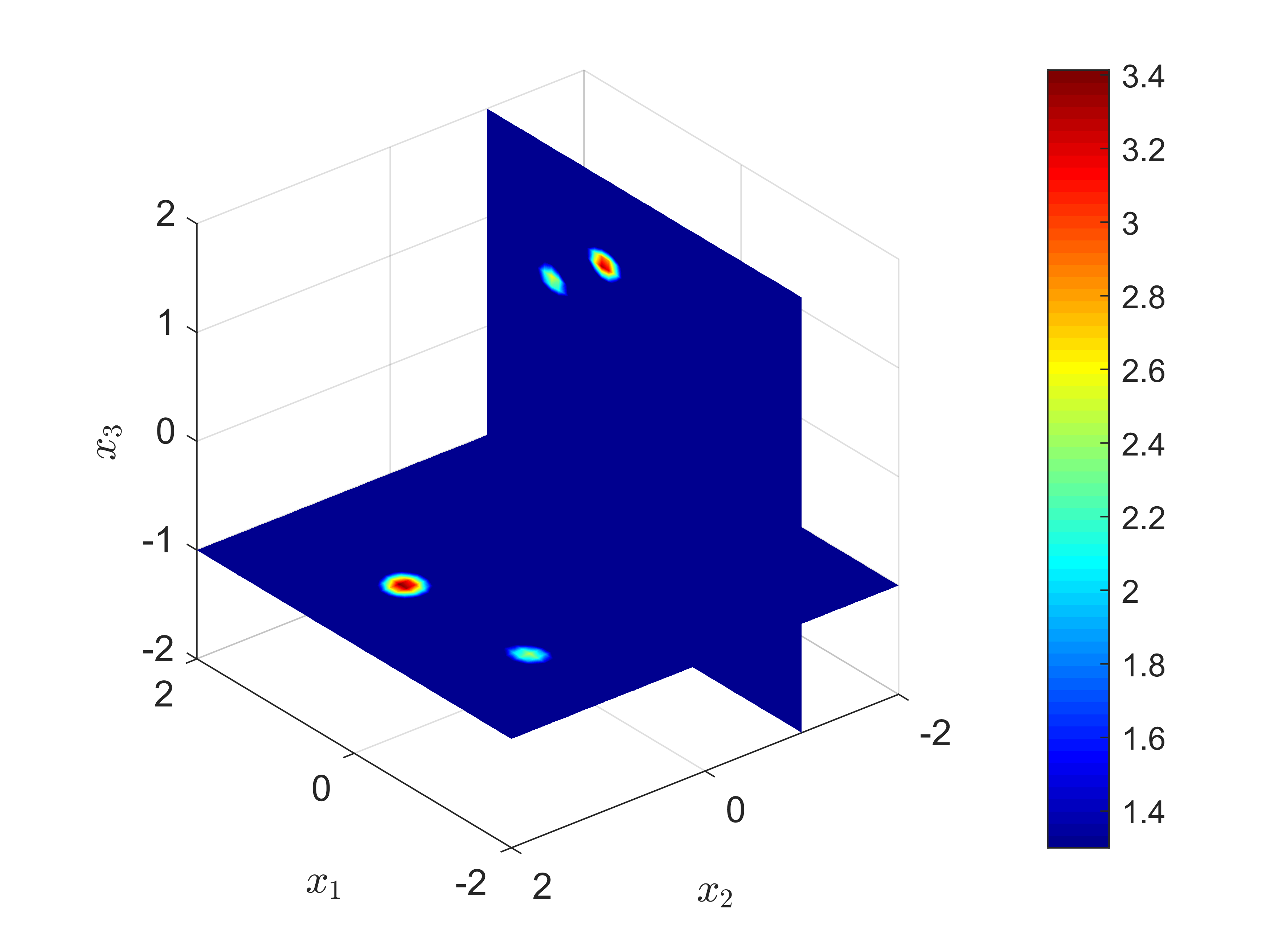} \\
  (a) & (b) \\
  \includegraphics[width=6.1cm]{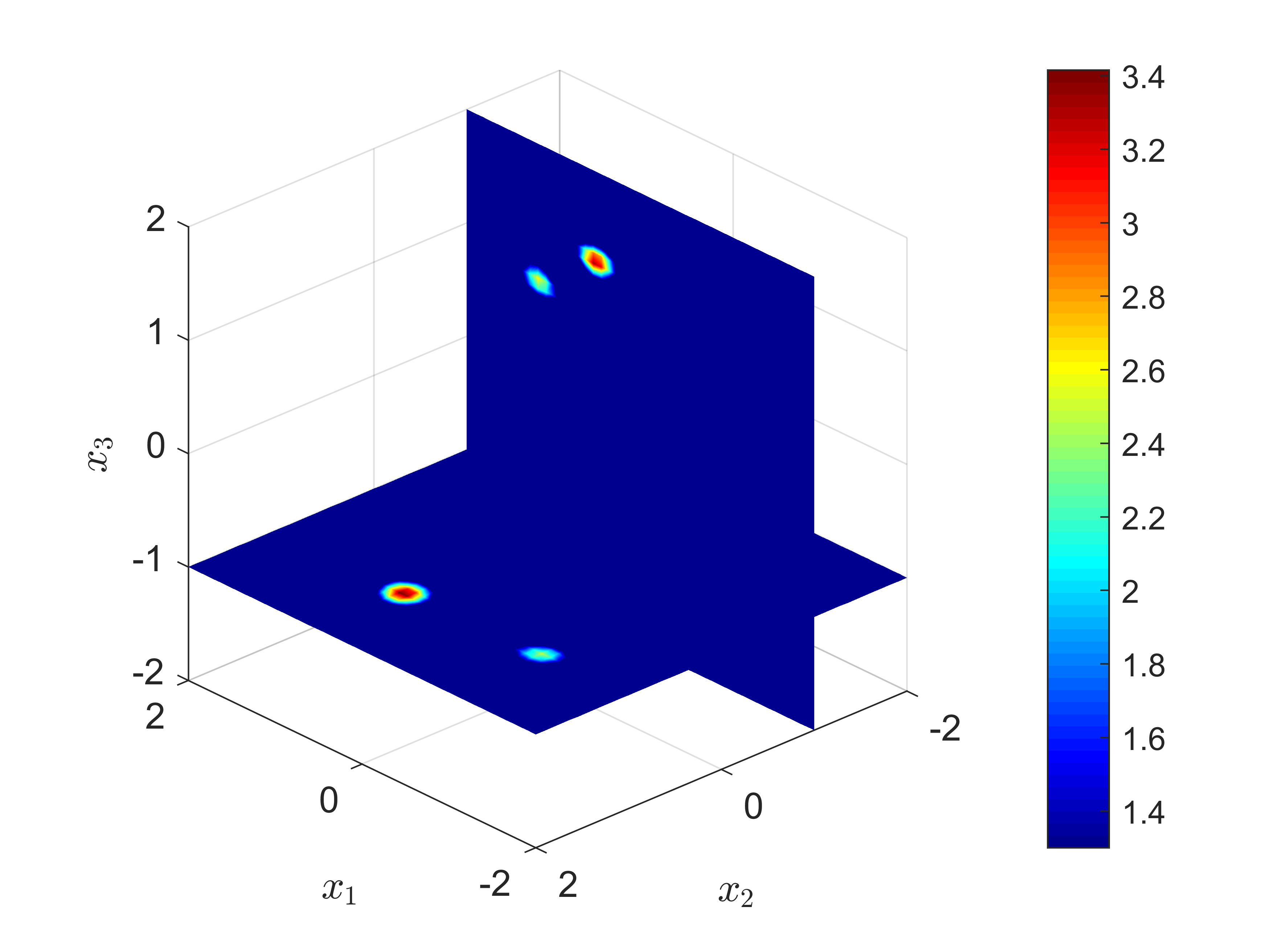} & \includegraphics[width=6.1cm]{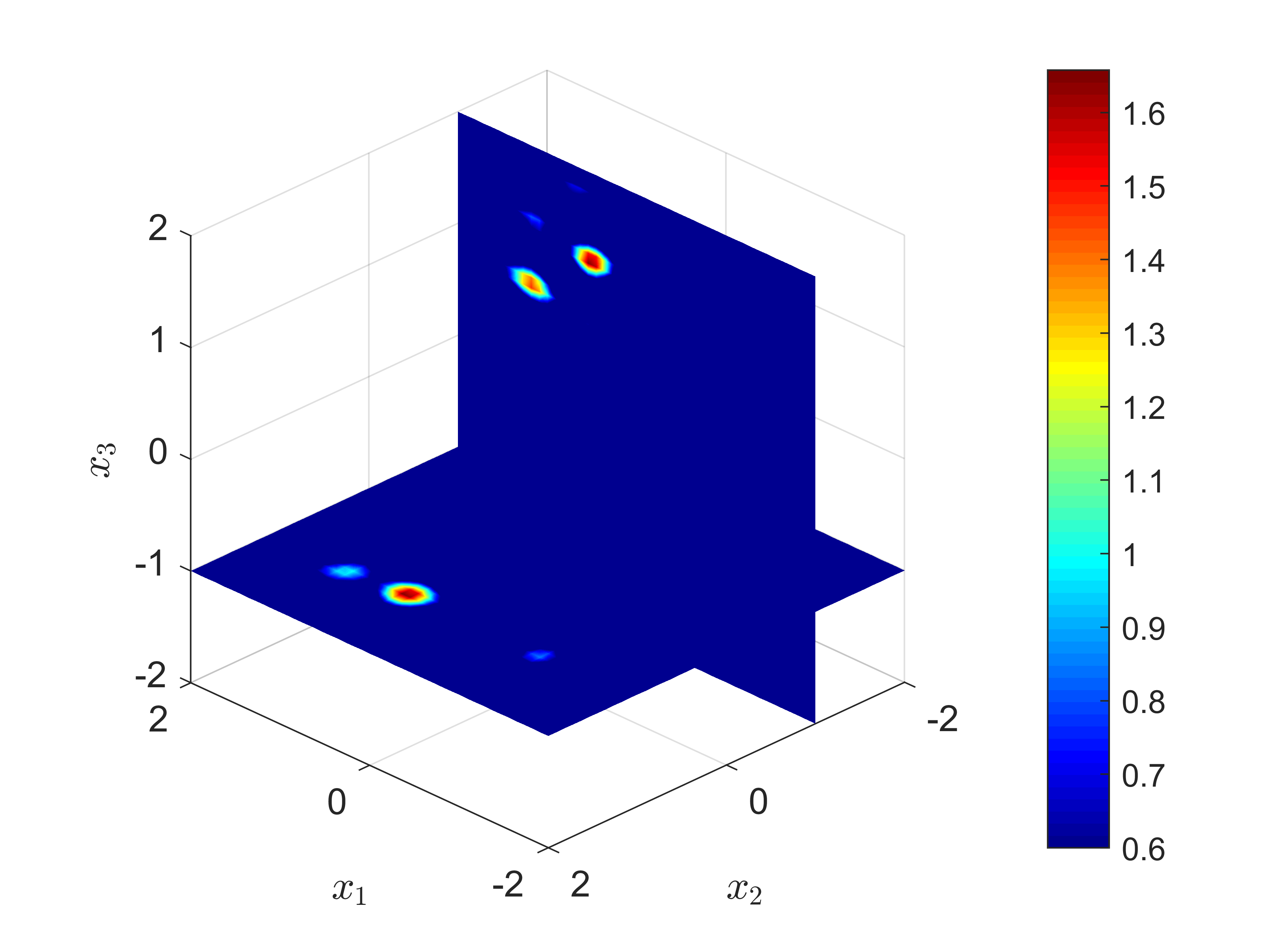} \\
  (c) & (d) \\
 \end{tabular}
 \caption{Reconstruction of 4 point sources with different intensities in $\mathbb{R}^3$. (a) Sketch of the example. (b) Reconstruction with all the sensors, $\epsilon=10\%$.  (c) Reconstruction with all the sensors, $\epsilon=5\%$. (d) Reconstruction with the left half of the sensors, $\epsilon=5\%$.}
\label{dian4tu}
 \vspace{-0.5em}
\end{figure}

\begin{table}
\centering
\caption{Specific reconstructed data of Example 2.}
\footnotesize
\begin{tabular}{@{}clll}
\br
 & & The actual & Reconstruction with all \\
No. &Items & point sources & the sensors, $\epsilon=5\%$ \\
\mr
1 & Location & $(-1, 1, -1)$ & $(-1,1,-1)$ \\
  & Intensity & 2 & 1.9093 \\
2 & Location & $(0.5, 1.1, -1)$ & $(0.5455,1.0909,-1)$ \\
  & Intensity & 3 & 2.9560 \\
3 & Location & $(1.2, -1, 0.8)$ & $(1.1818,-1,0.8182)$ \\
  & Intensity & 2 & 1.8969 \\
4 & Location & $(0.5, -1, 1.2)$ & $(0.5455,-1,1.1818)$ \\
  & Intensity & 3 & 2.7989 \\
\br
\label{dian4b}
\end{tabular}
\end{table}

\noindent\textbf{Example 3.}
In this example, the reconstruction of sources of the form $\lambda(t)\tau(x)\delta_L(x)$ on a two-dimensional cross section is considered. The source curve is chosen as
\begin{equation*}
  L:
  \left\{
  \begin{array}{ll}
  x_2=x_1^2-1, & \quad x_1\in[-0.9, 1.4],\\
  x_3=0 &
  \end{array}
  \right.
 \end{equation*}
with relative intensities $\tau(x)=x_1+3$.

The numerical scheme proposed in Section 4.1 is utilized to reconstruct the sources. In Step 1, the direct sampling method is utilized and the reconstruction can be seen in Figure \ref{meutu}. In Step 2, we get some local maximums of the indicator function $I(z)$ as the reconstructed source points. Note that some of the reconstructed source points, which are inconsistent with the visual results of Figure \ref{meutu}, are ignored while utilizing the fitting of polynomials in Step 3.

In the practice of the fitting of polynomials, we first get the polynomials respectively of the degree $1$, $2$, $3$ and $4$ which fit the data best in a least-squares sense. Then we choose the best fitting polynomial which maximizes
\begin{equation*}
R_{Adj}=1-\frac{n-1}{n-p-1}\cdot\frac{\mathop{\sum}\limits_{i=1}^{n} (y_i-\widehat{y_i})^2}{\mathop\sum\limits_{i=1}^{n} (y_i-\overline{y})^2},
\end{equation*}
where $n$ is the sample size, $p$ is the number of independent variables, $y_i$ is the true value of the dependent variable, $\widehat{y_i}$ is the predicted value of the dependent variable and $\displaystyle{\overline{y}=\frac{1}{n}\mathop{\sum}\limits_{i=1}^{n}y_i}$.

The fitting result is shown in Table \ref{meub}. According to the value of $R_{Adj}$, we choose the quadratic polynomial fitting in this example. See Table \ref{meub2} for a specific reconstruction of the source curve $L$.

\begin{table}
\centering
\caption{Polynomial fitting results in Example 3.}
\footnotesize
\begin{tabular}{@{}clll} 
\br
Degree & Polynomial fitting result & $R_{Adj}$ \\
\mr
1 & $x_2=0.5109x_1-0.6217$ & 0.3684 \\
2 & $x_2=0.9993x_1^2+0.008412x_1-1.006$ & 0.994 \\
3 & $x_2=0.02451x_1^3+0.9824x_1^2-0.006792x_1-1.001$ & 0.9932 \\
4 & $x_2=-0.03038x_1^4+0.05233x_1^3+1.007x_1^2-0.02195x_1-1.004$ & 0.992 \\
\br
\label{meub}
\end{tabular}
\end{table}

\begin{figure}[!ht]
 \center
 \scriptsize
 \begin{tabular}{ccc}
  \includegraphics[width=4.6cm]{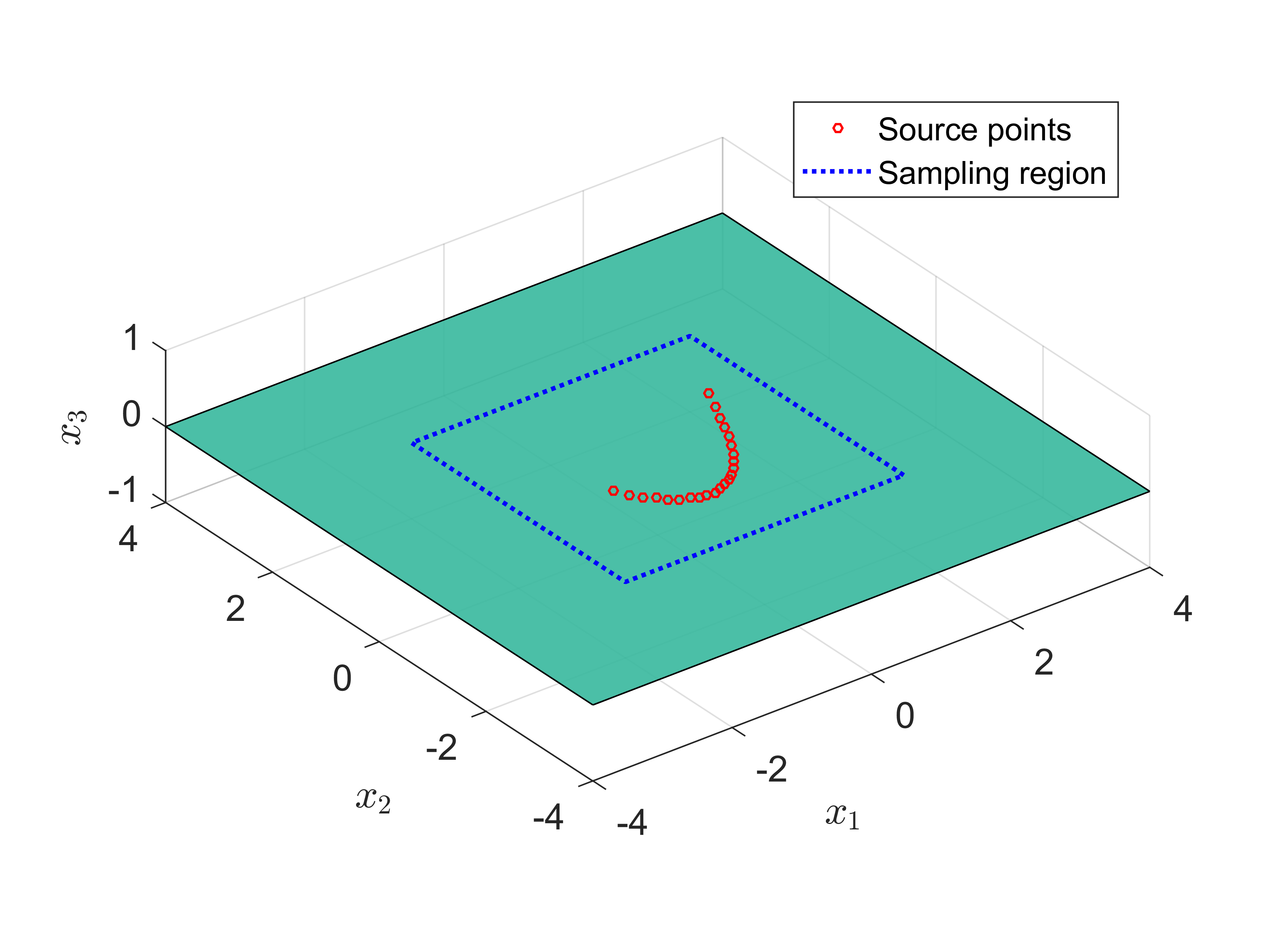} & \includegraphics[width=4.6cm]{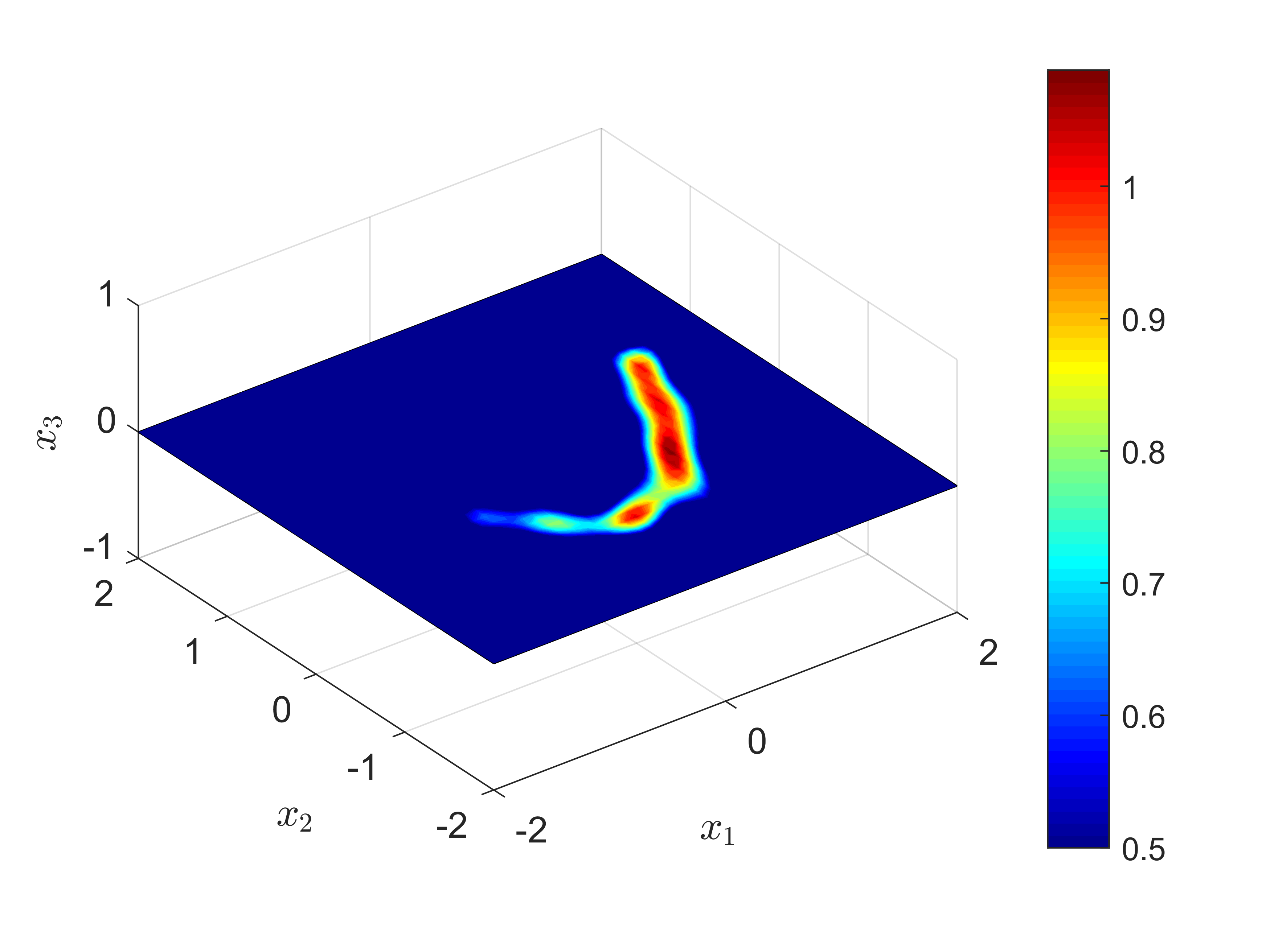} & \includegraphics[width=4.6cm]{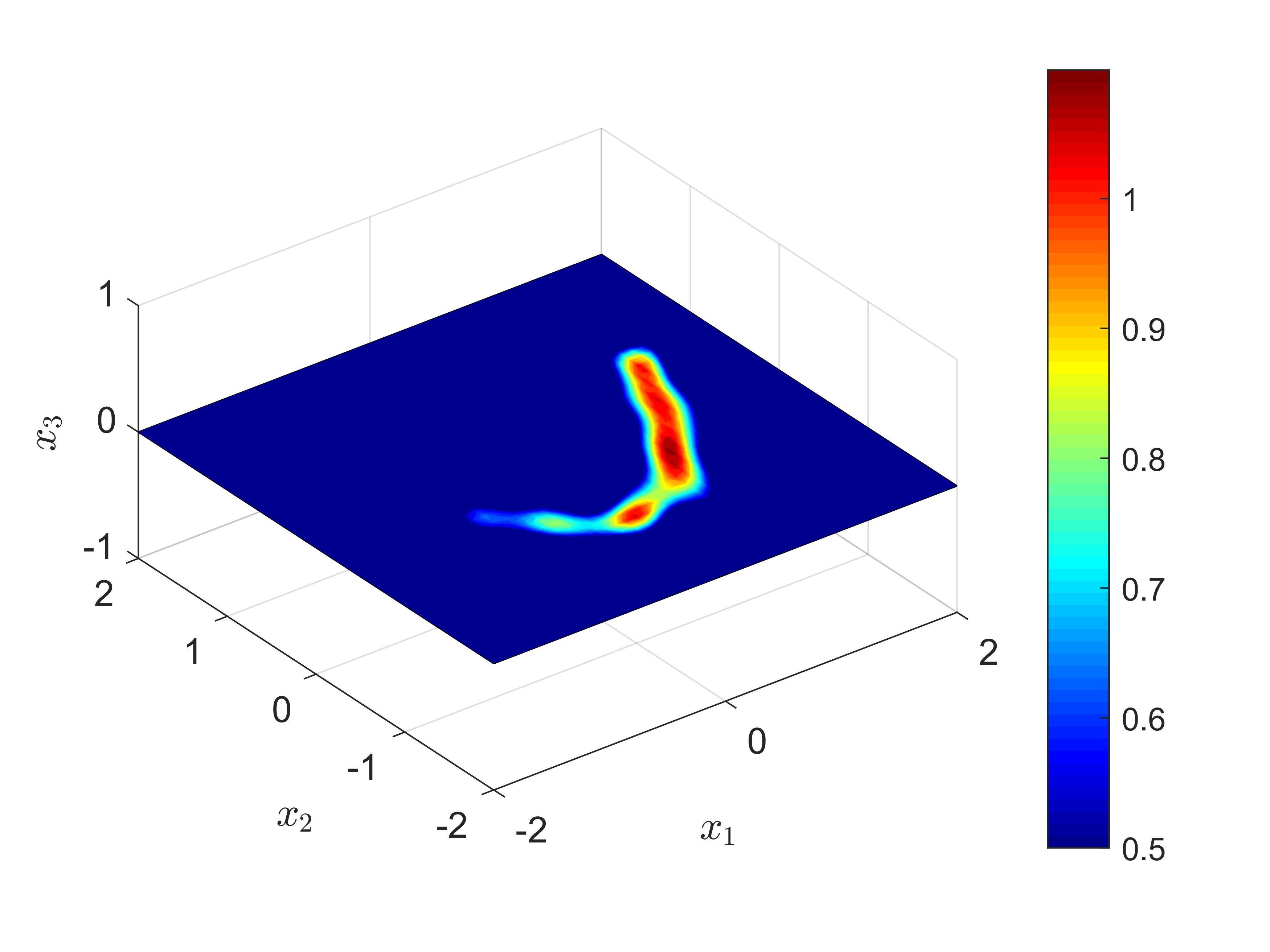} \\
  (a) & (b) & (c)  \\
  \includegraphics[width=4.6cm]{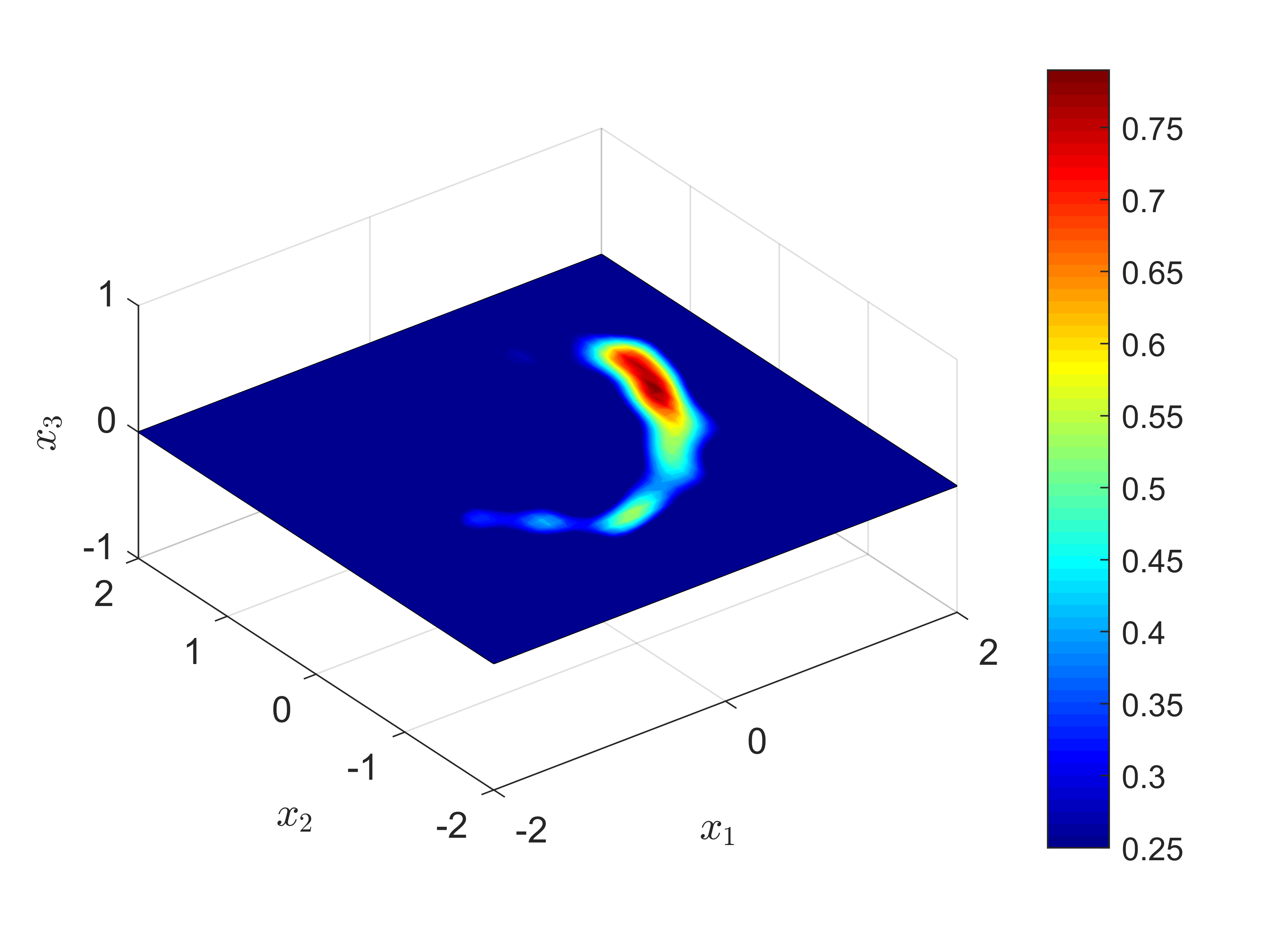} & \includegraphics[width=4.6cm]{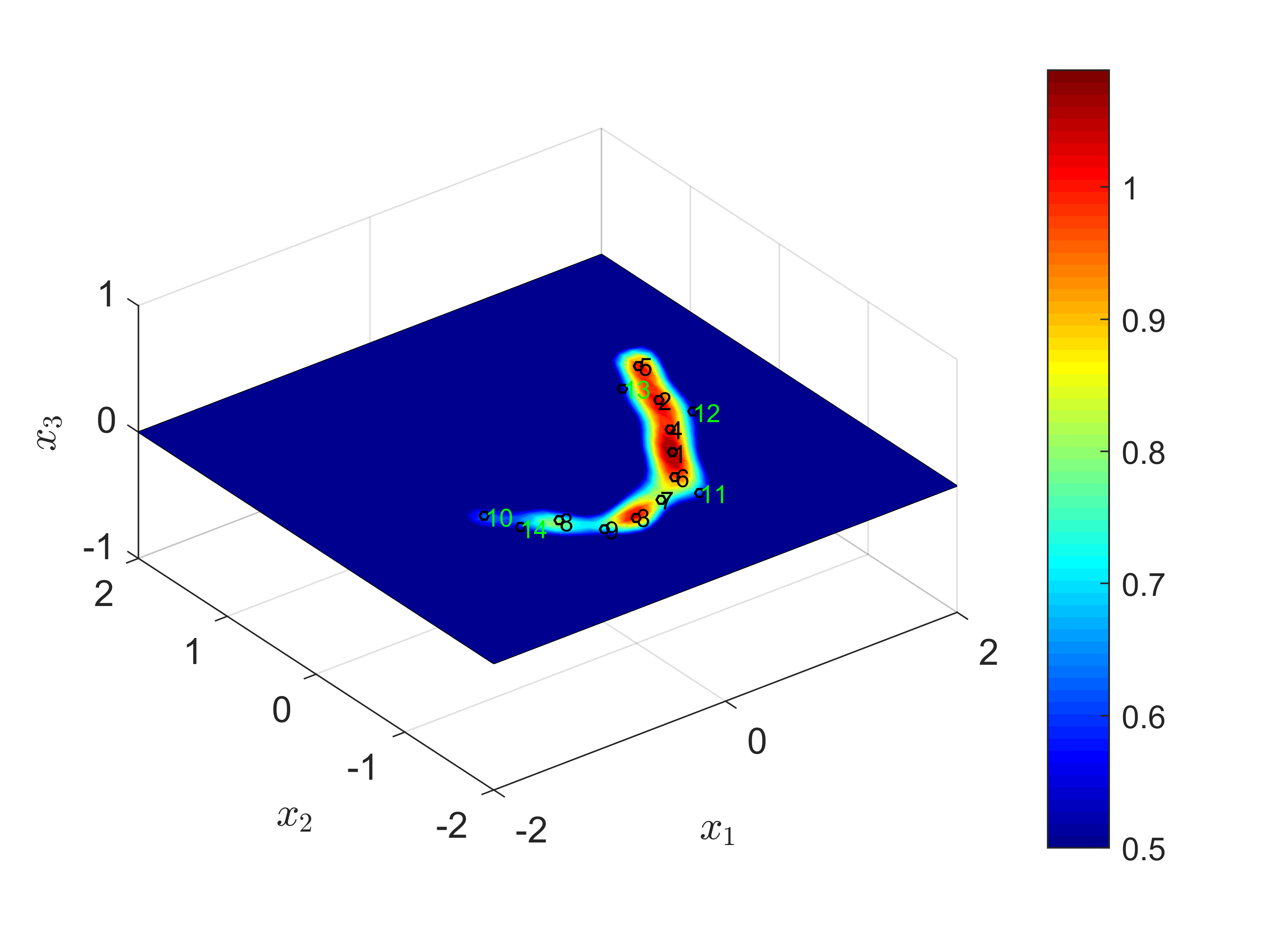} & \includegraphics[width=4.6cm]{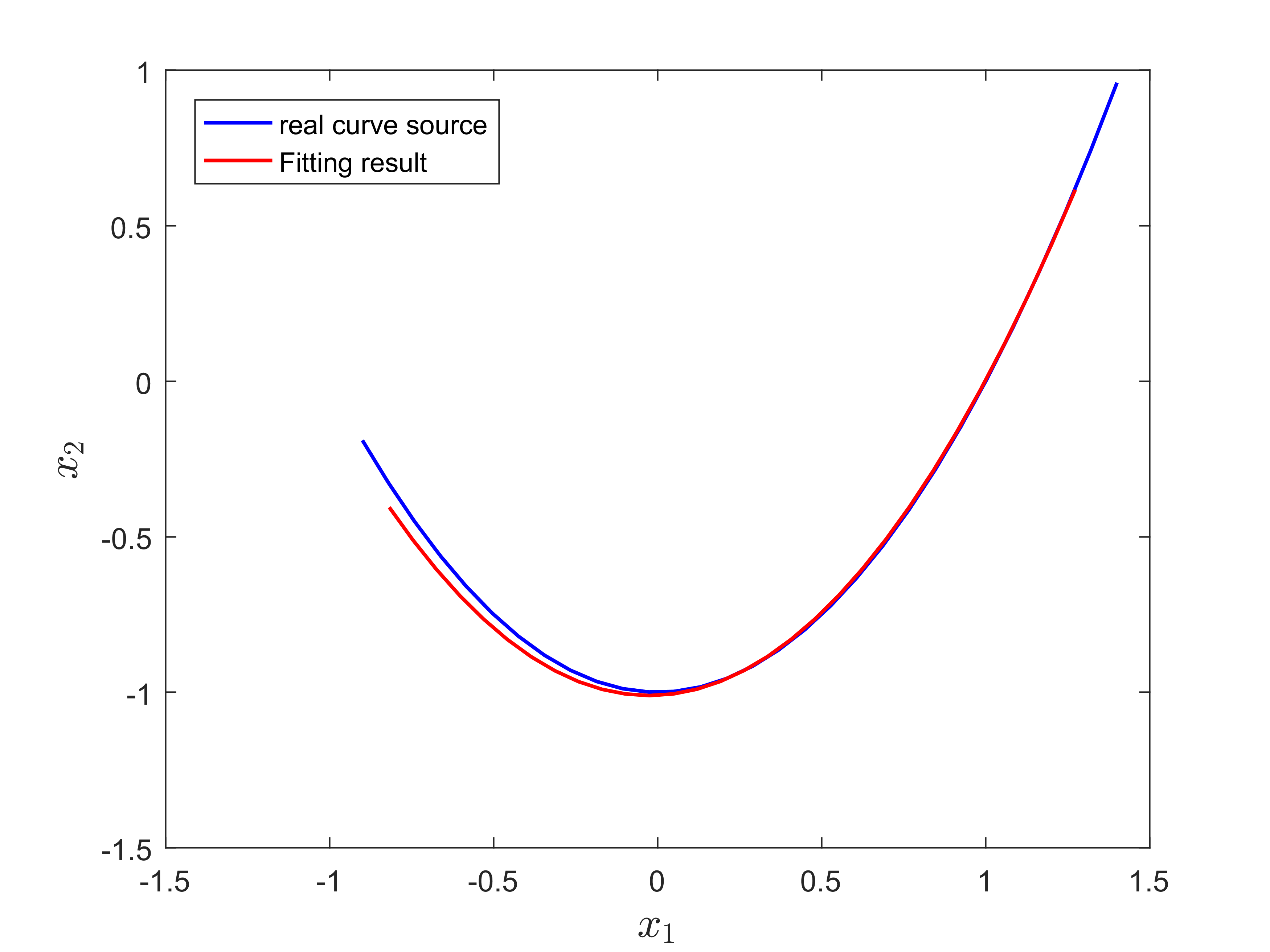} \\
  (d) & (e) & (f) \\
 \end{tabular}
 \caption{Reconstruction of the source of the form $\lambda(t)\tau(x)\delta_L(x)$. (a) Sketch of the example. (b) Reconstruction with all the sensors, $\epsilon=10\%$. (c) Reconstruction with all the sensors, $\epsilon=5\%$. (d) Reconstruction with the left half of the sensors, $\epsilon=5\%$. (e) The approximate source points with all the sensors, $\epsilon=5\%$. (f) Polynomial fitting result.}
\label{meutu}
 \label{meu}
 \vspace{-0.5em}
\end{figure}

\newcommand{\tabincell}[2]{\begin{tabular}{@{}#1@{}}#2\end{tabular}}
\begin{table}
\centering
\caption{Reconstruction of the source curve in Example 3.}
\footnotesize
\begin{tabular}{@{}llll} 
\br
\tabincell{c}{The actual curve source} & \tabincell{c}{Reconstruction with all the sensors, $\epsilon=5\%$} \\
\mr
$x_2=x_1^2-1, x_1\in[-0.9, 1.4]$ & $x_2=0.9993x_1^2+0.008412x_1-1.006, x_1\in[-0.8182, 1.2727]$ \\
$x_3=0$ & $x_3=0$ \\
\br
\label{meub2}
\end{tabular}
\end{table}

\noindent\textbf{Example 4.}
This example is to study the reconstruction of multiple curve sources. The relative intensities are chosen as $\tau(x)=x_1+3$ for all the cases bellow.

Case 1. The source curve is selected as
\begin{equation*}
  L:
  \left\{
  \begin{array}{ll}
  x_2=\frac{1}{x_1}, & \quad x_1\in[-1.4, -0.6]\cup[0.6, 1.4],\\
  x_3=0. &
  \end{array}
  \right.
 \end{equation*}
The reconstruction is shown in Figure \ref{xiaxietu}. The numerical scheme is similar to that in Example 3. Note that Figure \ref{xiaxietu} shows that the sources are composed of two parts. Therefore, the reconstructed point sources are divided into two parts and the fitting of polynomials is utilized separately for two parts. See Table \ref{xiaxieb} for the specific reconstructed data.

\begin{figure}[!ht]
 \center
 \scriptsize
 \begin{tabular}{ccc}
  \includegraphics[width=4.6cm]{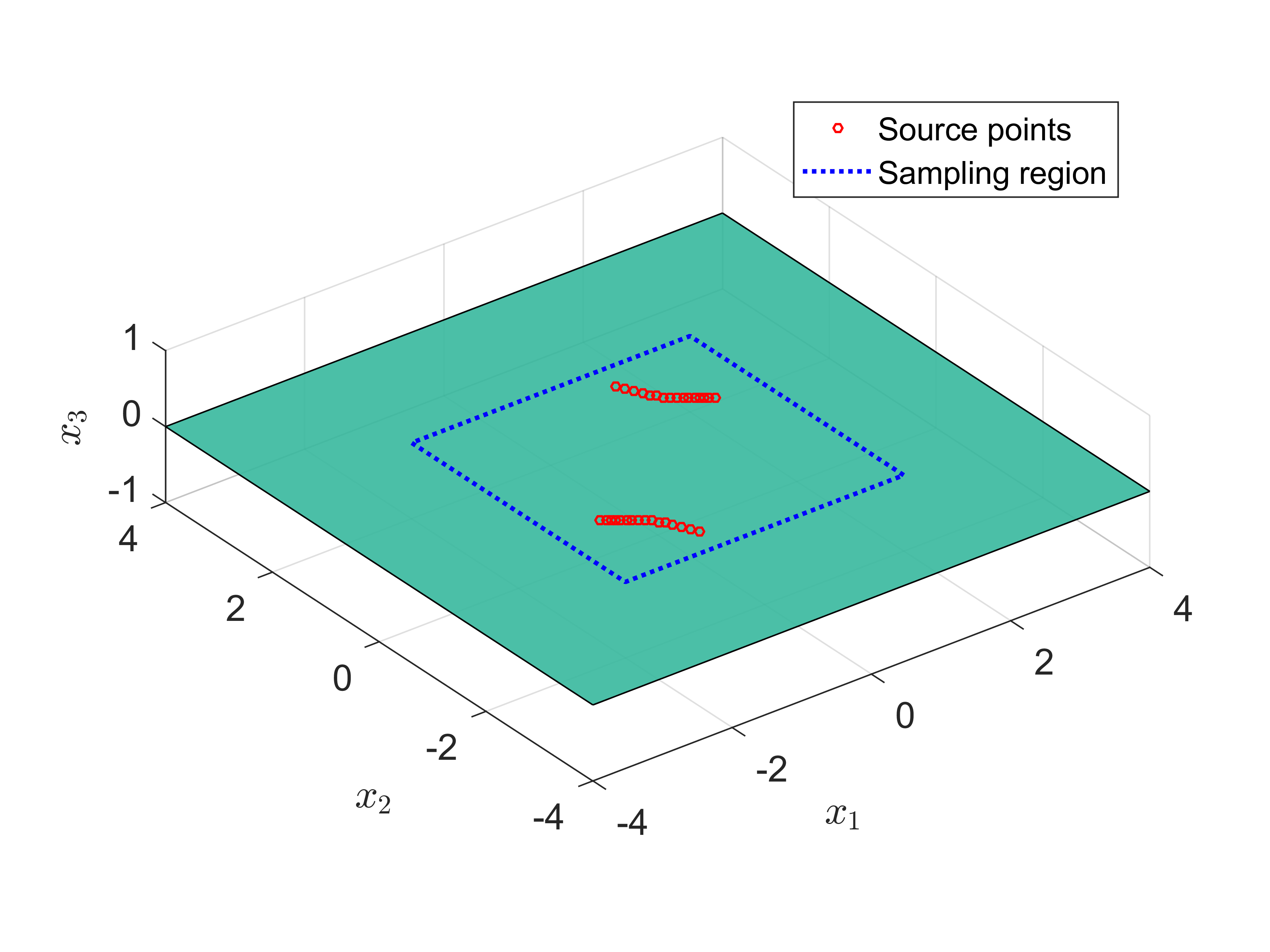} & \includegraphics[width=4.6cm]{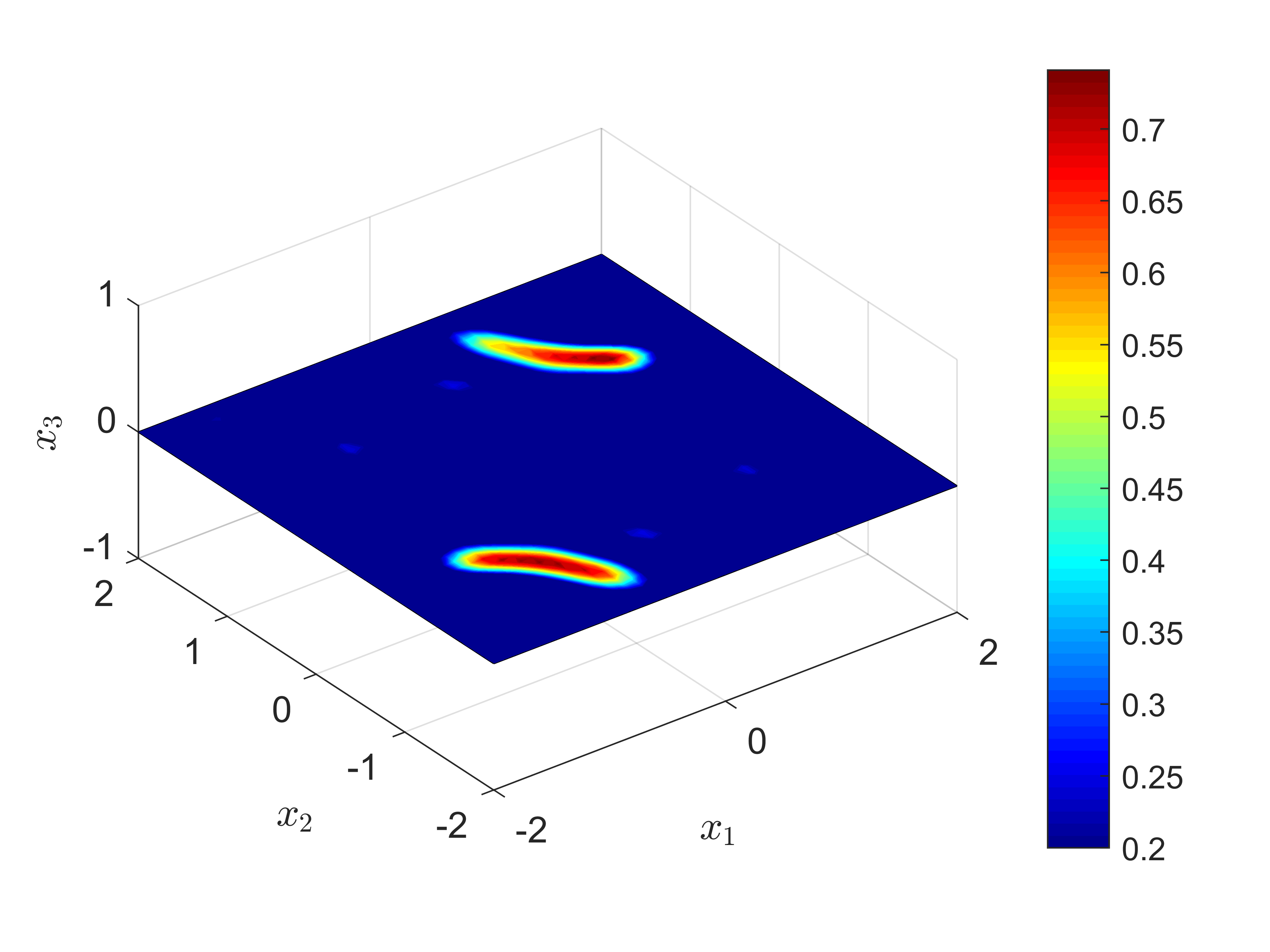} & \includegraphics[width=4.6cm]{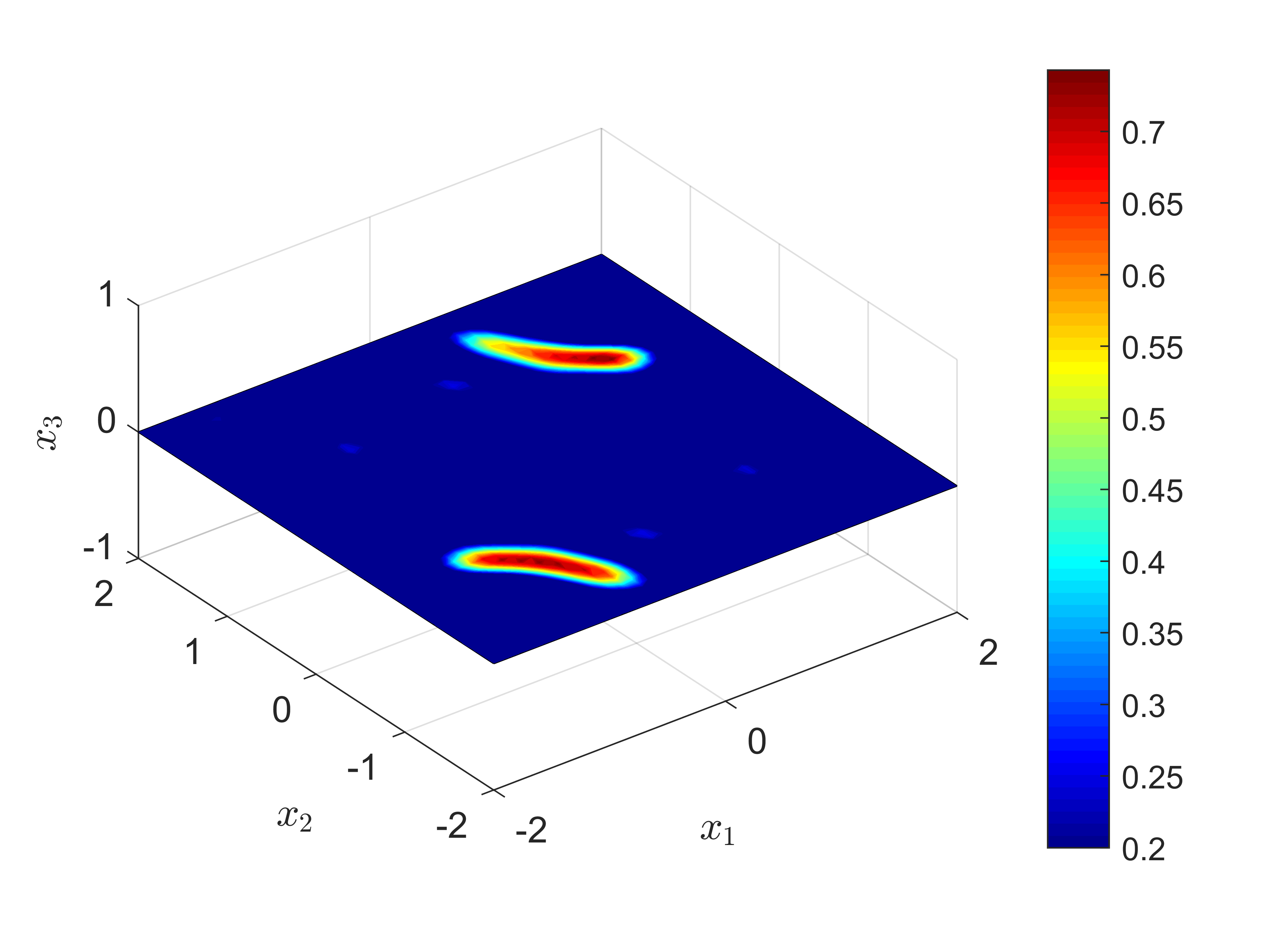} \\
  (a) & (b) & (c) \\
  \includegraphics[width=4.6cm]{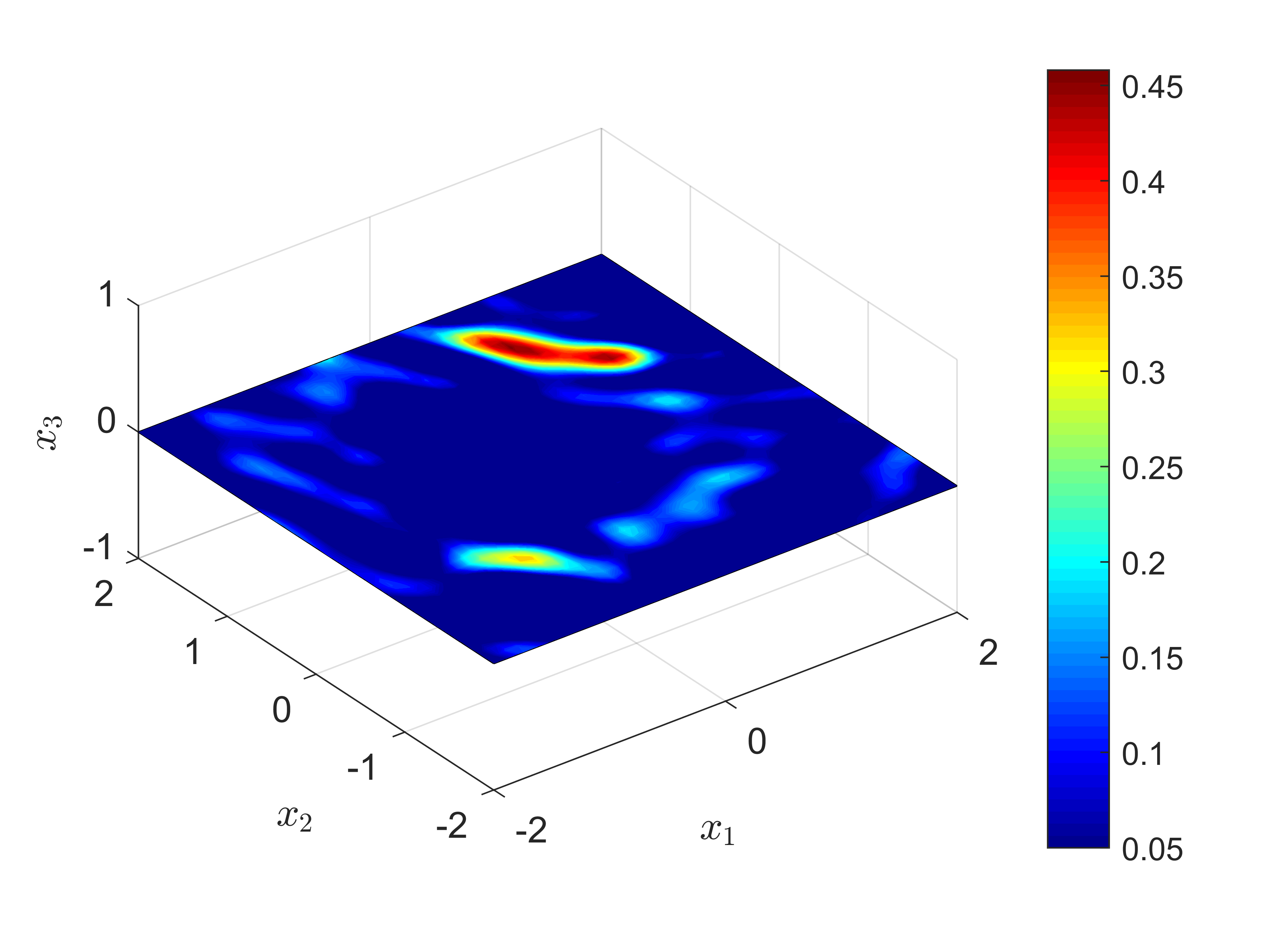} & \includegraphics[width=4.6cm]{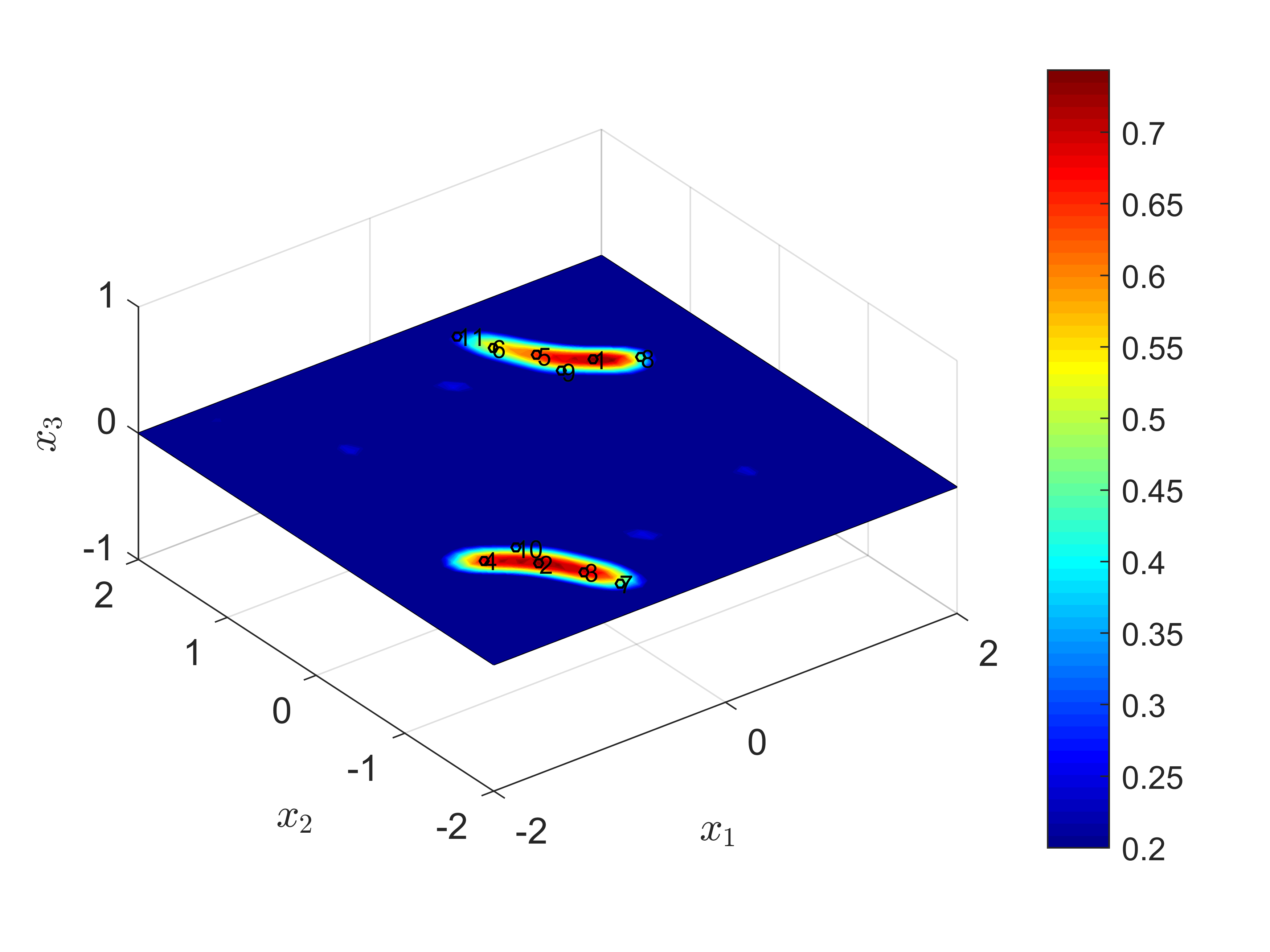} & \includegraphics[width=4.6cm]{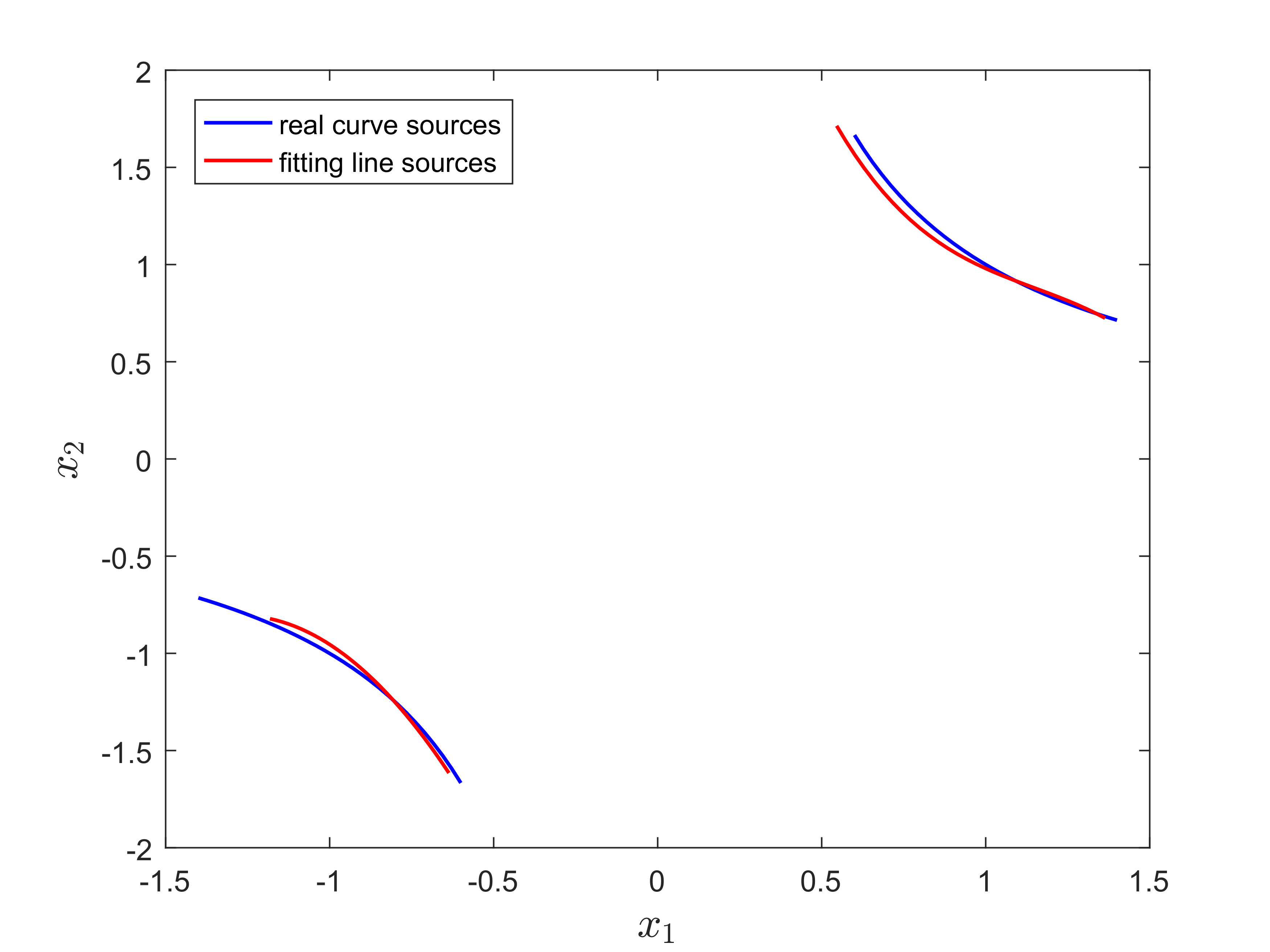} \\
  (d) & (e) & (f)  \\
 \end{tabular}
 \caption{Reconstruction of multiple curve sources, Case 1. (a) Sketch of the example. (b) Reconstruction with all the sensors,  $\epsilon=10\%$. (c) Reconstruction with all the sensors, $\epsilon=5\%$. (d) Reconstruction with the left half of the sensors, $\epsilon=5\%$. (e) The approximate source points with all the sensors, $\epsilon=5\%$. (f) Polynomial fitting result.}
 \label{xiaxietu}
 \vspace{-0.5em}
\end{figure}

\begin{table}
\centering
\caption{Reconstruction of multiple source curves in Example 4, Case 1.}
\footnotesize
\begin{tabular}{@{}llll} 
\br
No. & \tabincell{c}{The actual curve sources} & \tabincell{c}{Reconstruction with all the sensors, $\epsilon=5\%$} \\
\mr
1 & $x_2=\frac{1}{x_1}, x_1\in[-1.4, -0.6]$ & $x_2=-2.001x_1^2-5.093x_1-4.046, x_1\in[-1.1818, -0.6364]$ \\
  & $x_3=0$ & $x_3=0$ \\
2 & $x_2=\frac{1}{x_1}, x_1\in[0.6, 1.4]$ & $x_2=0.794x_1^2-2.579x_1+2.77, x_1\in[0.5455, 1.3636]$\\
  & $x_3=0$ & $x_3=0$ \\
\br
\label{xiaxieb}
\end{tabular}
\end{table}

Case 2. The source curves are chosen as the combination of
\begin{equation*}
  L_1:
  \left\{
  \begin{array}{ll}
  x_2=-2x_1^2+1, & \quad x_1\in[0.2, 1],\\
  x_3=0 &
  \end{array}
  \right.
 \end{equation*}
and
 \begin{equation*}
  L_2:
  \left\{
  \begin{array}{ll}
  x_2=x_1^2-1, & \quad x_1\in[-0.9, 1.4],\\
  x_3=0. &
  \end{array}
  \right.
 \end{equation*}
The reconstruction is shown in Figure \ref{xianjiaochatu}. See Table \ref{xianjiaochab} for the specific reconstructed data.

\begin{figure}[!ht]
 \center
 \scriptsize
 \begin{tabular}{ccc}
  \includegraphics[width=4.6cm]{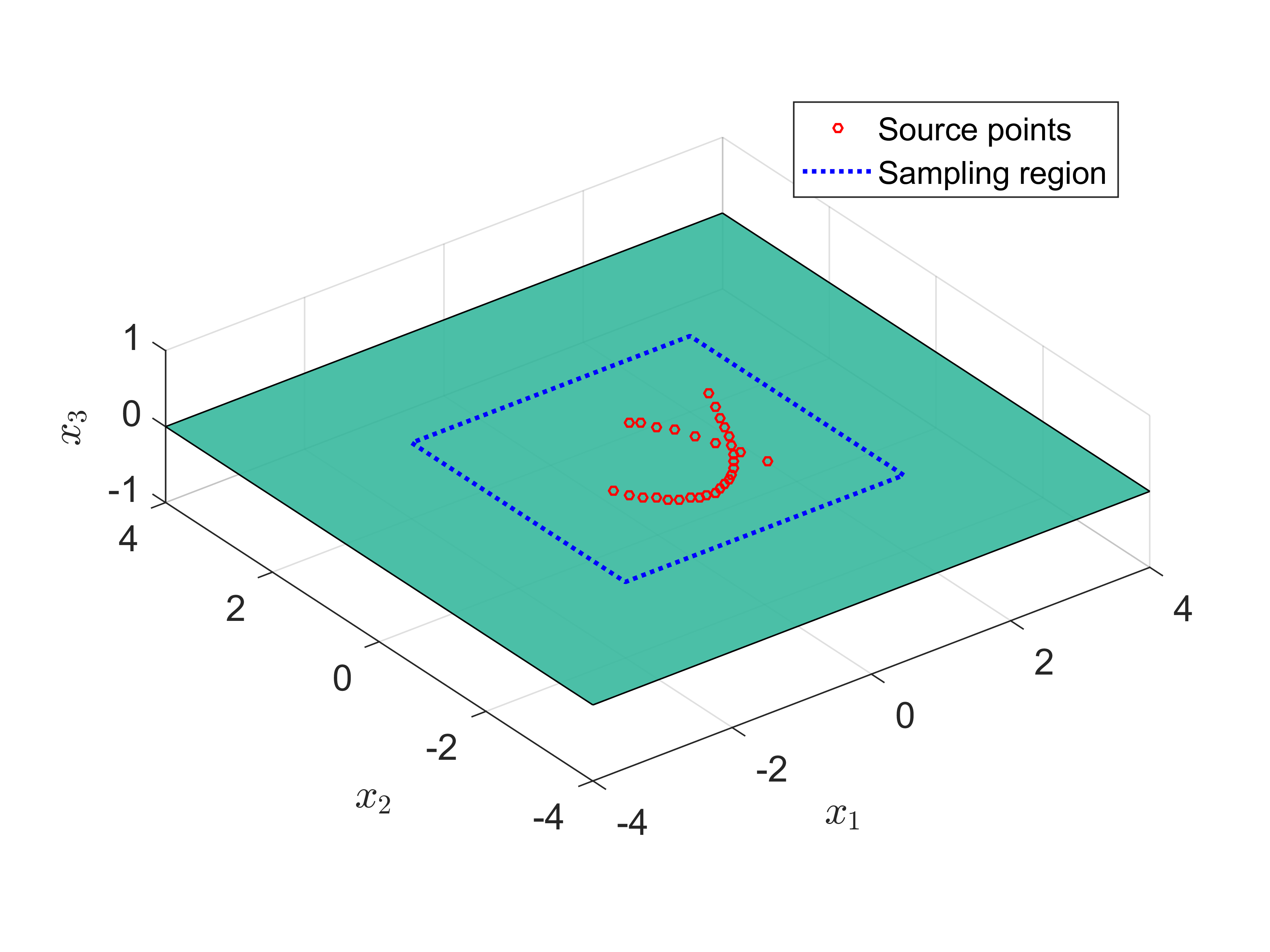} & \includegraphics[width=4.6cm]{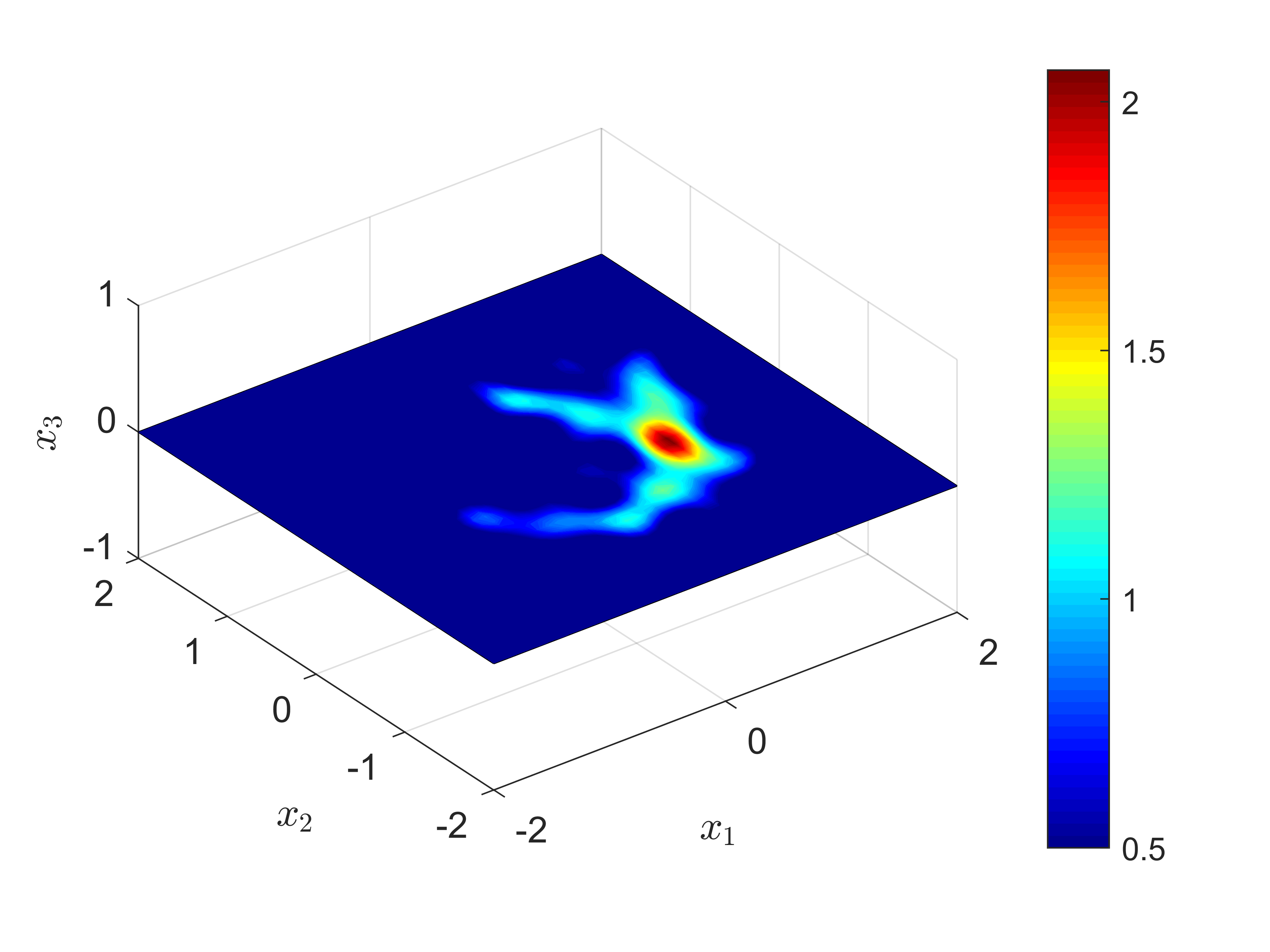} & \includegraphics[width=4.6cm]{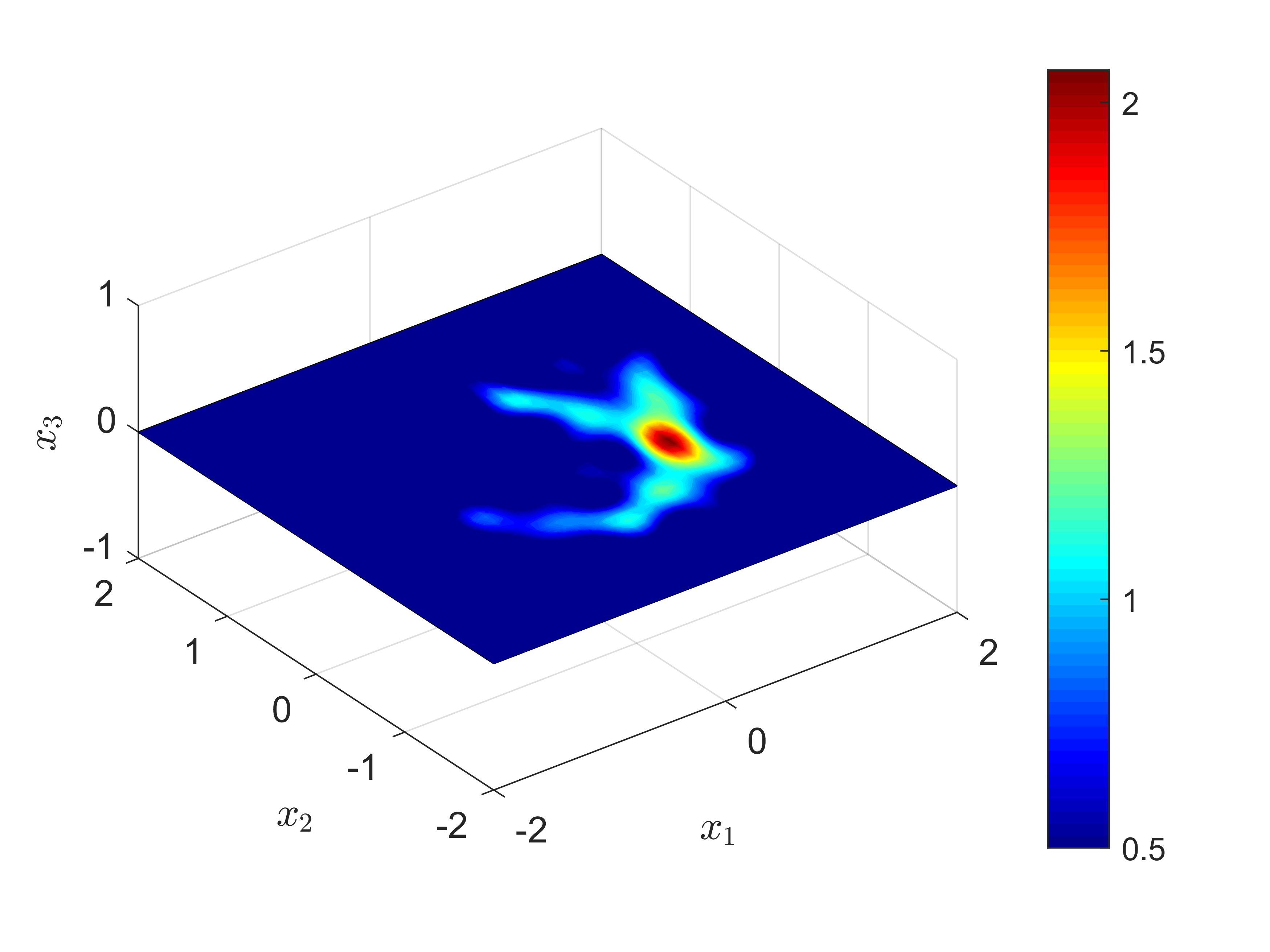} \\
  (a) & (b) & (c) \\
  \includegraphics[width=4.6cm]{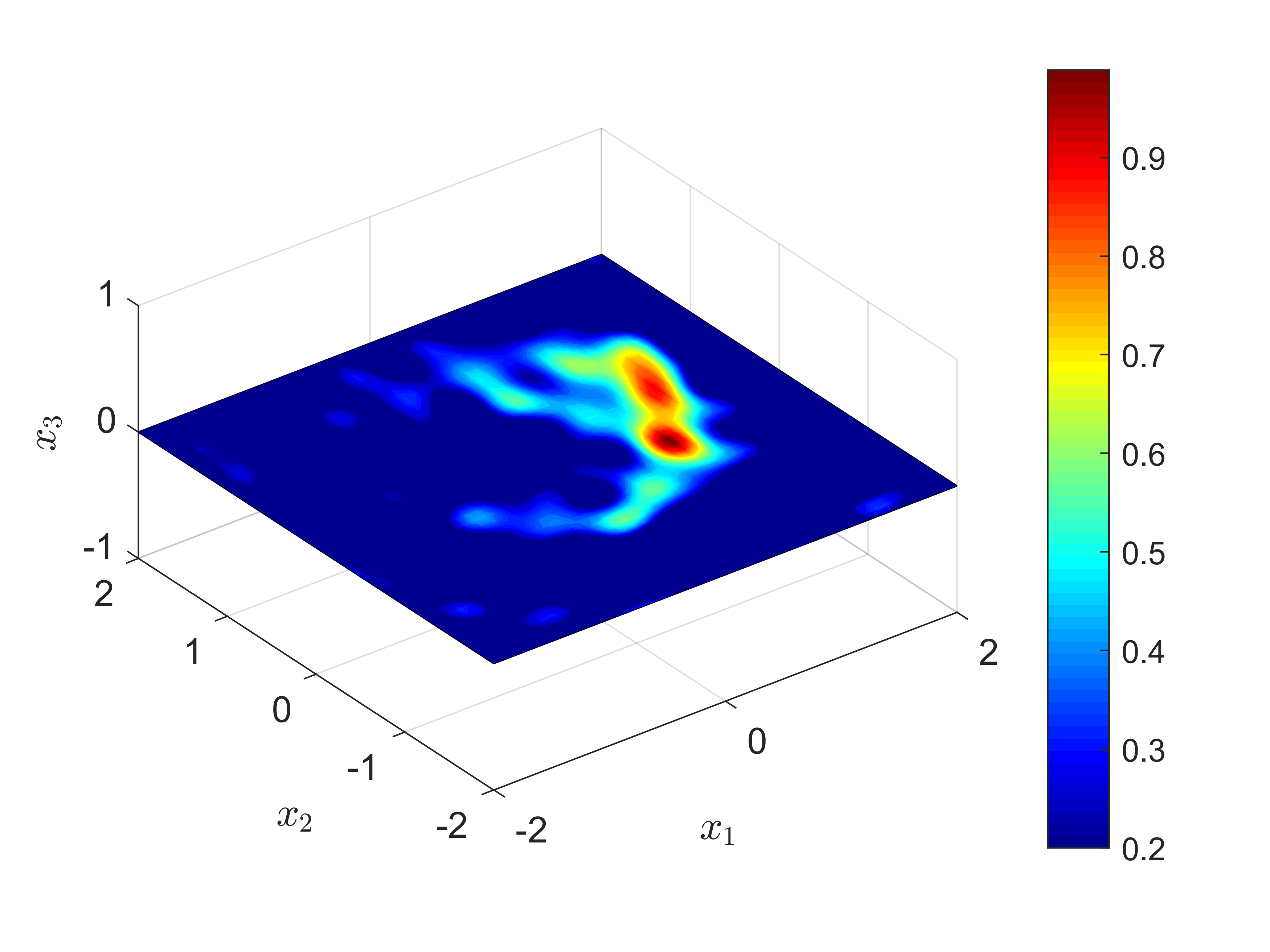} & \includegraphics[width=4.6cm]{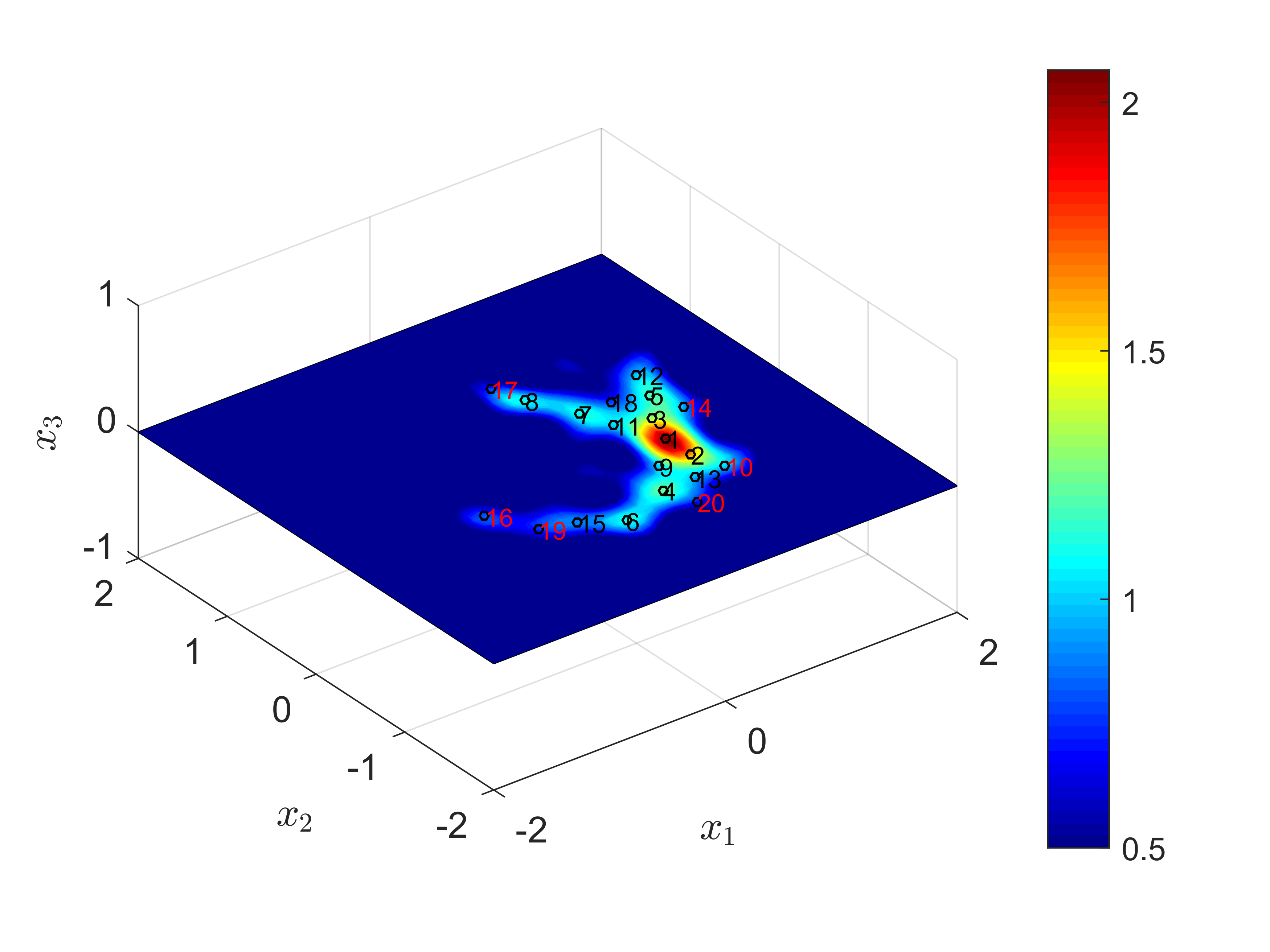} & \includegraphics[width=4.6cm]{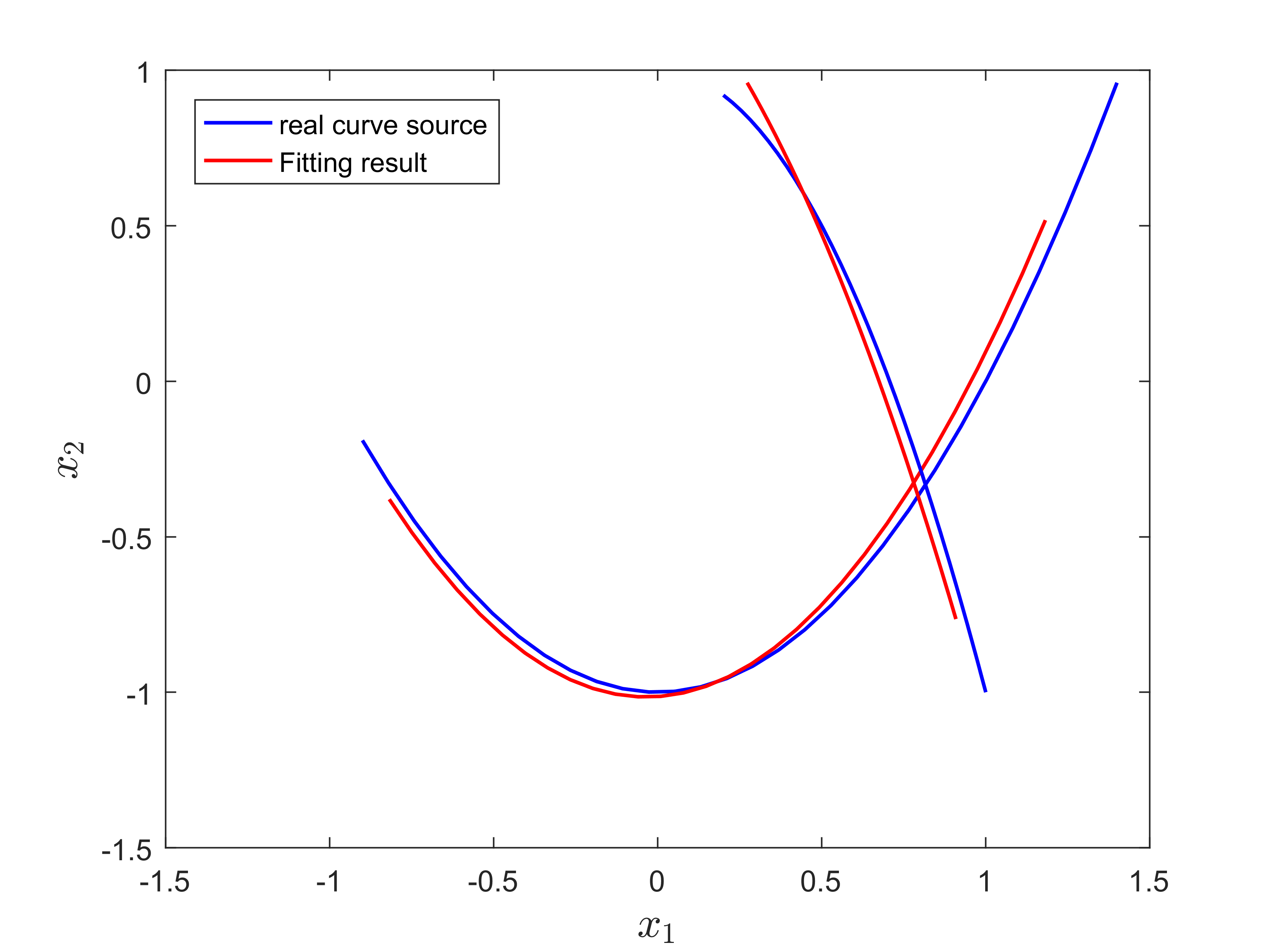} \\
  (d) & (e) & (f)  \\
 \end{tabular}
 \caption{Reconstruction of multiple curve sources, Case 2. (a) Sketch of the example. (b) Reconstruction with all the sensors, $\epsilon=10\%.$ (c) Reconstruction with all the sensors, $\epsilon=5\%.$ (d) Reconstruction with the left half of the sensors, $\epsilon=5\%$. (e) The approximate source points with all the sensors, $\epsilon=5\%$. (f) Polynomial fitting result.}
 \label{xianjiaochatu}
 \vspace{-0.5em}
\end{figure}

\begin{table}
\centering
\caption{Reconstruction of multiple source curves in Example 4, Case 2.}
\footnotesize
\begin{tabular}{@{}llll} 
\br
No. & \tabincell{c}{The actual curve sources} & \tabincell{c}{Reconstruction with all the sensors, $\epsilon=5\%$} \\
\mr
1 & $x_2=x_1^2-1, x_1\in[-0.9, 1.4]$ & $x_2=1.142x_1^2-0.0482x_1-1, x_1\in [-0.8182, 1.1818]$ \\
  & $x_3=0$  & $x_3=0$ \\
2 & $x_2=-2x_1^2+1, x_1\in[0.2, 1]$ & $x_2=-2.152x_1^2-0.01441x_1+1.01, x_1\in[0.2727, 0.9091]$ \\
  & $x_3=0$  & $x_3=0$ \\
\br
\label{xianjiaochab}
\end{tabular}
\end{table}

Case 3. The source curves are selected as the combination of
\begin{equation*}
  L_1:
  \left\{
  \begin{array}{ll}
  x_2=\frac{1}{2}x_1^2-\frac{1}{4}, & \quad x_1\in[-0.9, 1.5],\\
  x_3=0 &
  \end{array}
  \right.
 \end{equation*}
and
 \begin{equation*}
  L_2:
  \left\{
  \begin{array}{ll}
  x_3=\sqrt{x_1},   & \quad x_1\in[0.1, 1.5],\\
  x_2=0. &
  \end{array}
  \right.
 \end{equation*}
The reconstruction is shown in Figure \ref{liangmianxiantu}. See Table \ref{liangmianxianb} for the specific reconstructed data.
\begin{figure}[!ht]
 \center
 \scriptsize
 \begin{tabular}{ccc}
  \includegraphics[width=4.6cm]{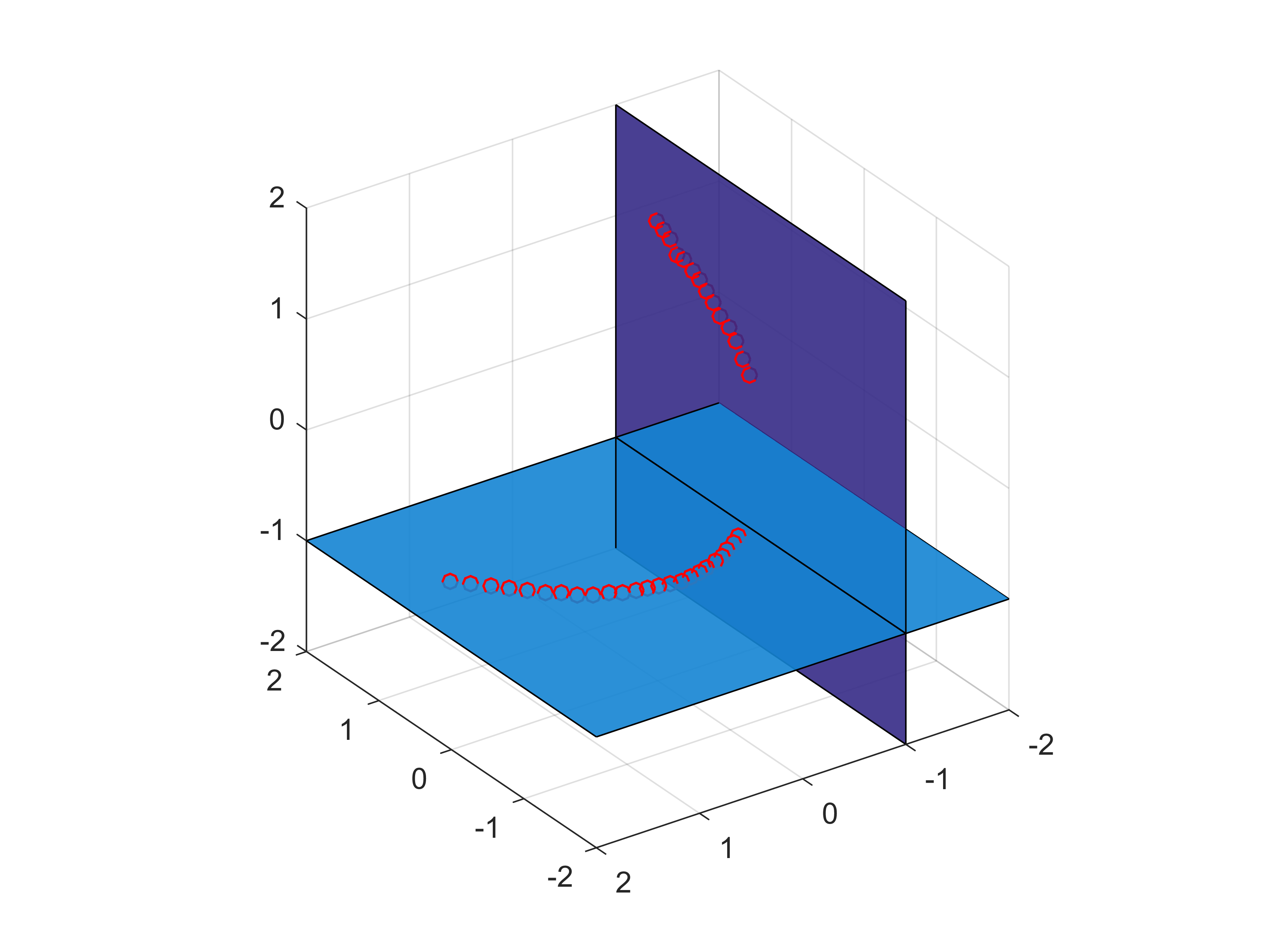} & \includegraphics[width=4.6cm]{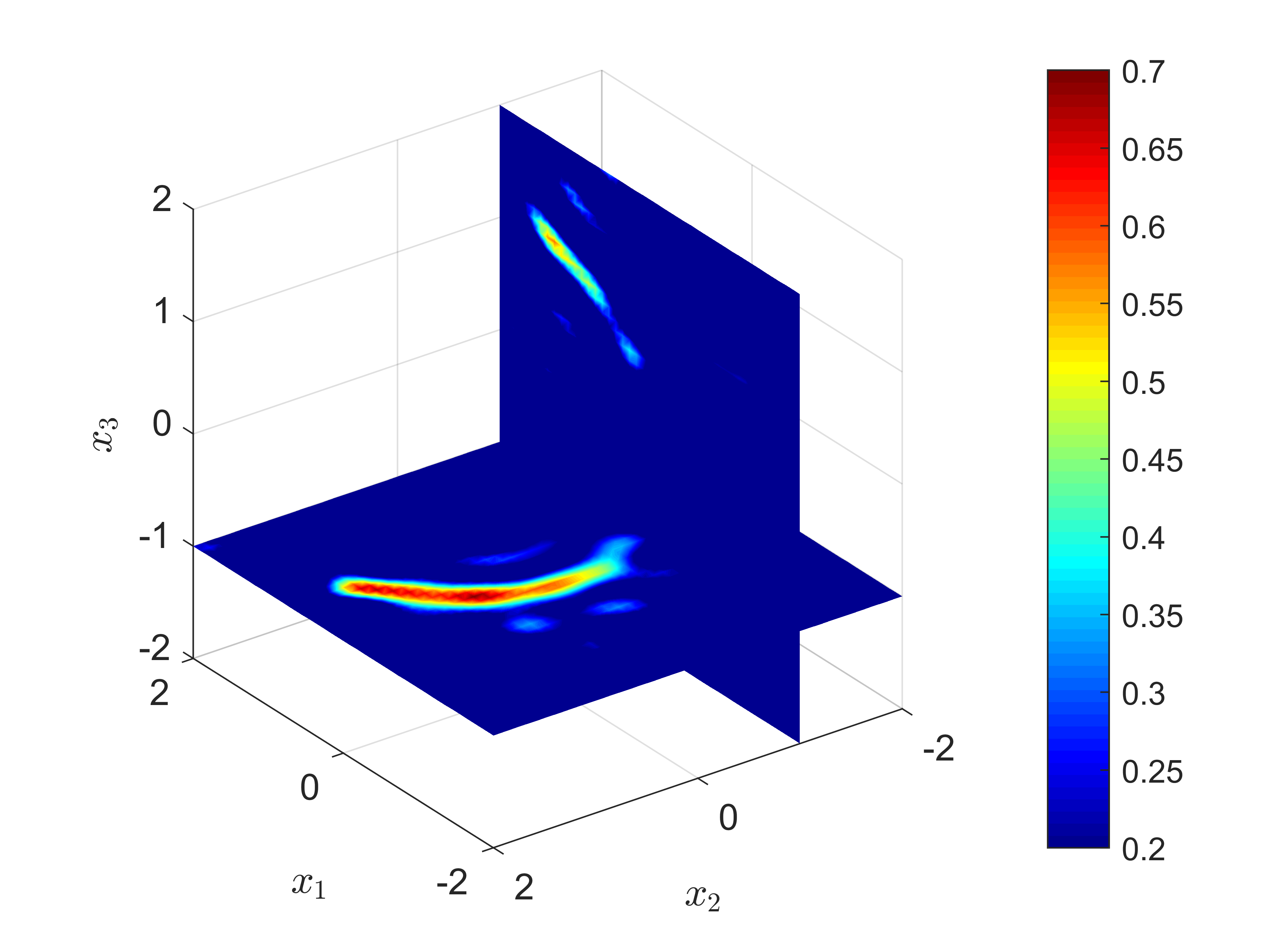} & \includegraphics[width=4.6cm]{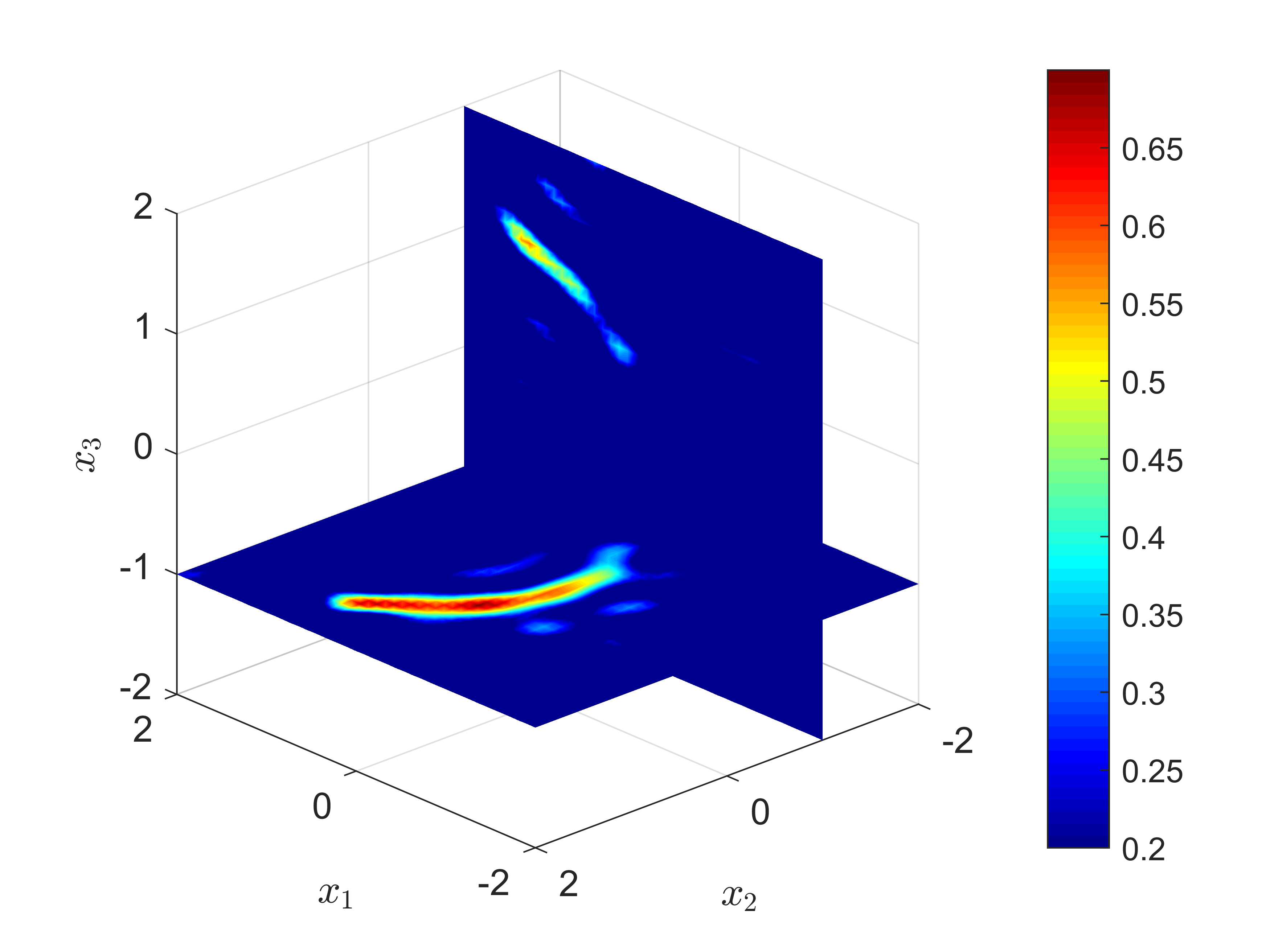} \\
  (a) & (b) & (c) \\
  \includegraphics[width=4.6cm]{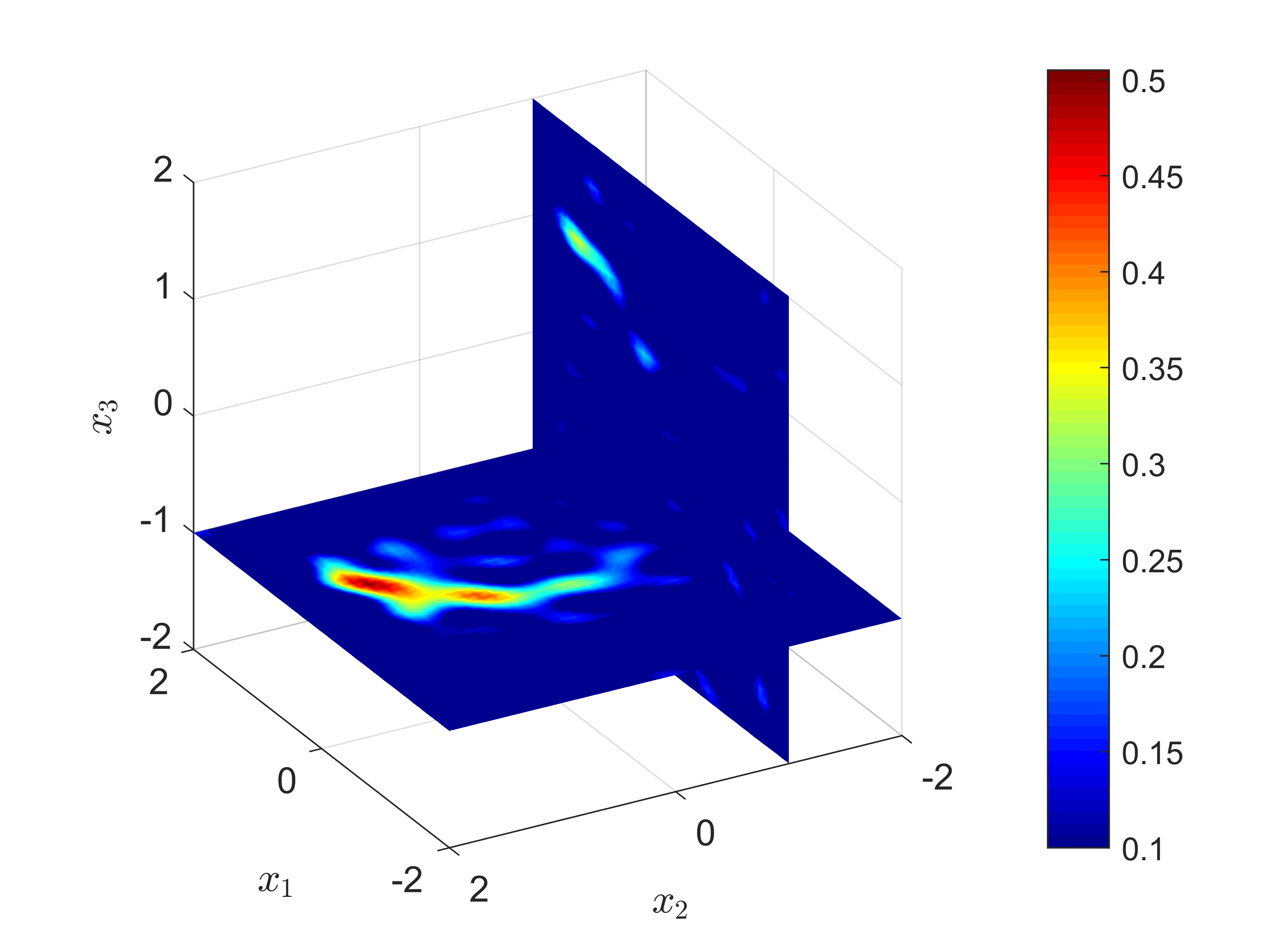} & \includegraphics[width=4.6cm]{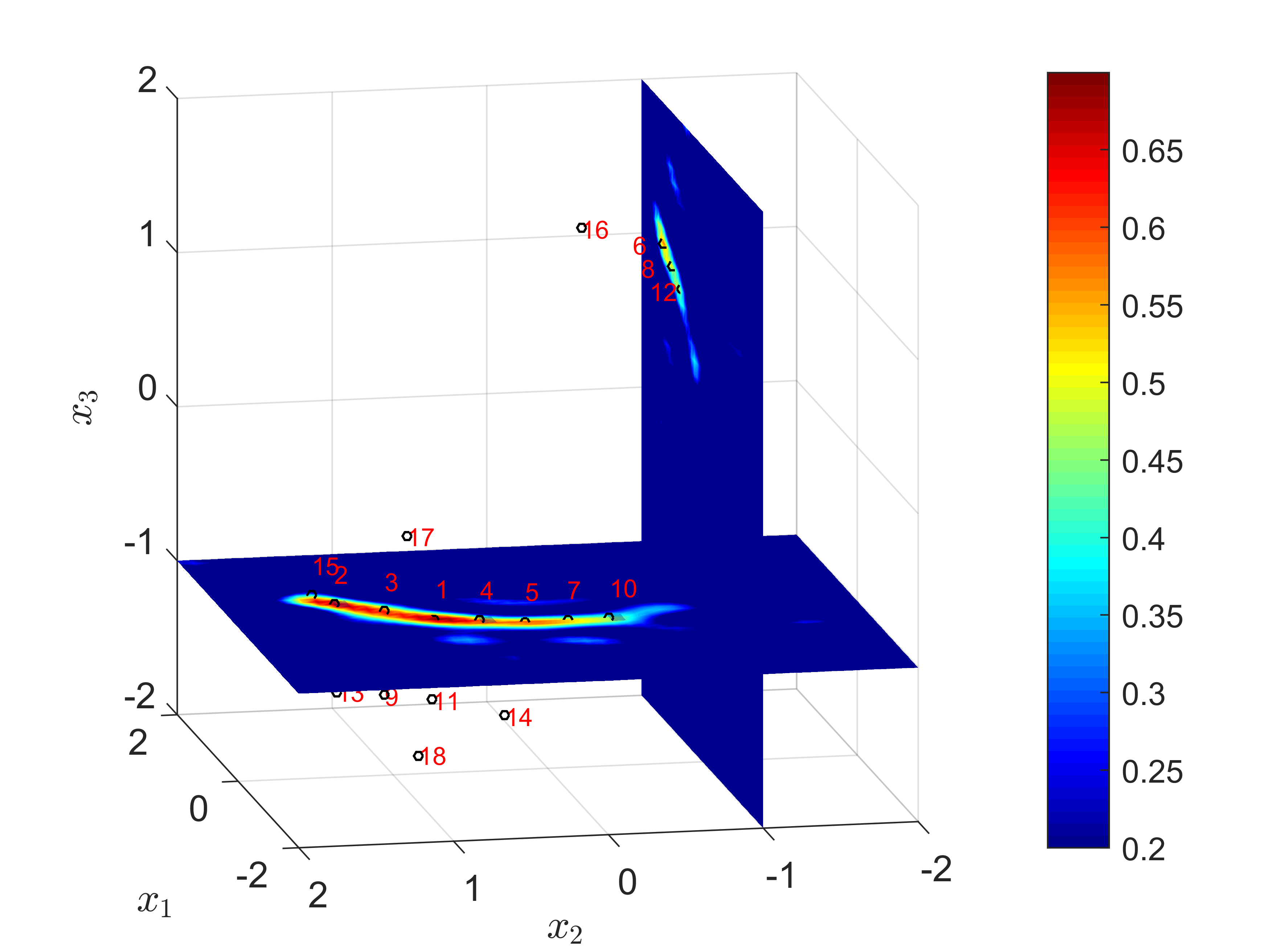} & \includegraphics[width=4.6cm]{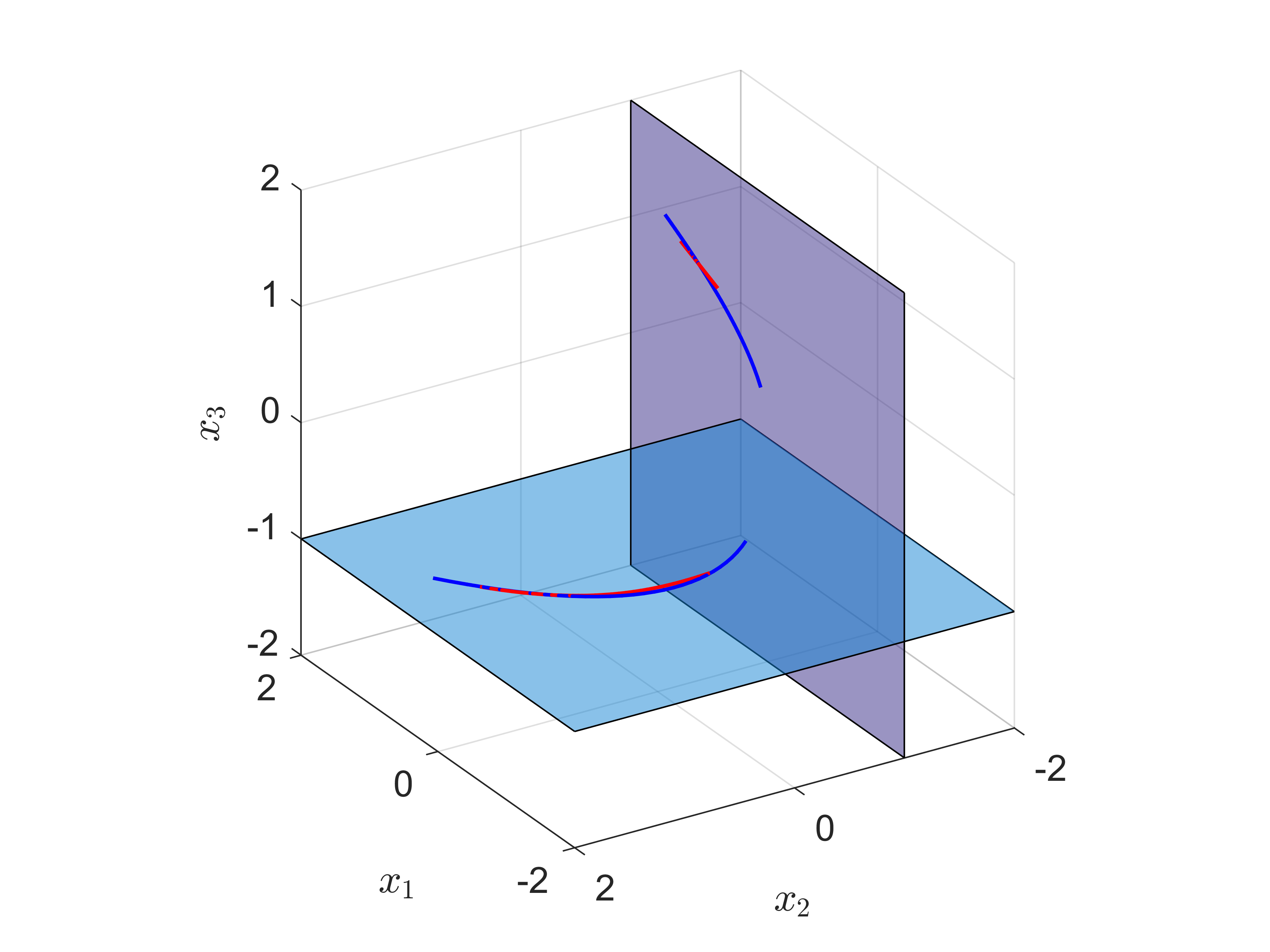} \\
  (d) & (e) & (f)  \\
 \end{tabular}
 \caption{Reconstruction of multiple curve sources, Case 3. (a) Sketch of the example. (b) Reconstruction with all the sensors, $\epsilon=10\%$. (c) Reconstruction with all the sensors, $\epsilon=5\%$. (d) Reconstruction with the left half of the sensors, $\epsilon=5\%$. (e) The approximate source points with all the sensors, $\epsilon=5\%$. (f) Polynomial fitting result.}
 \label{liangmianxiantu}
 \vspace{-0.5em}
\end{figure}

\begin{table}
\centering
\caption{Reconstruction of multiple source curves in Example 4, Case 3.}
\footnotesize
\begin{tabular}{@{}llll} 
\br
No. & \tabincell{c}{The actual curve sources} & \tabincell{c}{Reconstruction with all the sensors, $\epsilon=5\%$} \\
\mr
1 & $x_2=\frac{1}{2}x_1^2-\frac{1}{4}, x_1\in[-0.9, 1.5]$ & $x_2=0.1247x_1^3+0.2665x_1^2+0.04406x_1-0.1907, $ \\
  & $x_3=0$ & $x_1 \in[-0.3636, 1.2727], x_3=0$ \\
2 & $x_3=\sqrt{x_1}, x_1\in[0.1, 1.5]$ & $x_3=0.3333x_1+0.6667, x_1\in[0.7273, 1.2727]$ \\
  & $x_2=0$ & $x_2=0$ \\
\br
\label{liangmianxianb}
\end{tabular}
\end{table}

\noindent\textbf{Example 5.}
This example is to study the simultaneous reconstruction of a curve source and a point source. The source point is located at $(-1, -1, 0)$ with the intensity $1$ and the curve source is chosen as
\begin{equation*}
  L:
  \left\{
  \begin{array}{ll}
  x_2=-2x_1^2+1, & \quad x_1\in[0.2, 1],\\
  x_3=0 &
  \end{array}
  \right.
 \end{equation*}
with relative intensities $\tau(x)=x_1+3.$ The reconstruction is shown in Figure \ref{dianxian}. See Table \ref{dianxianb} for the specific reconstructed data.
\begin{figure}[!ht]
 \center
 \scriptsize
 \begin{tabular}{ccc}
  \includegraphics[width=4.6cm]{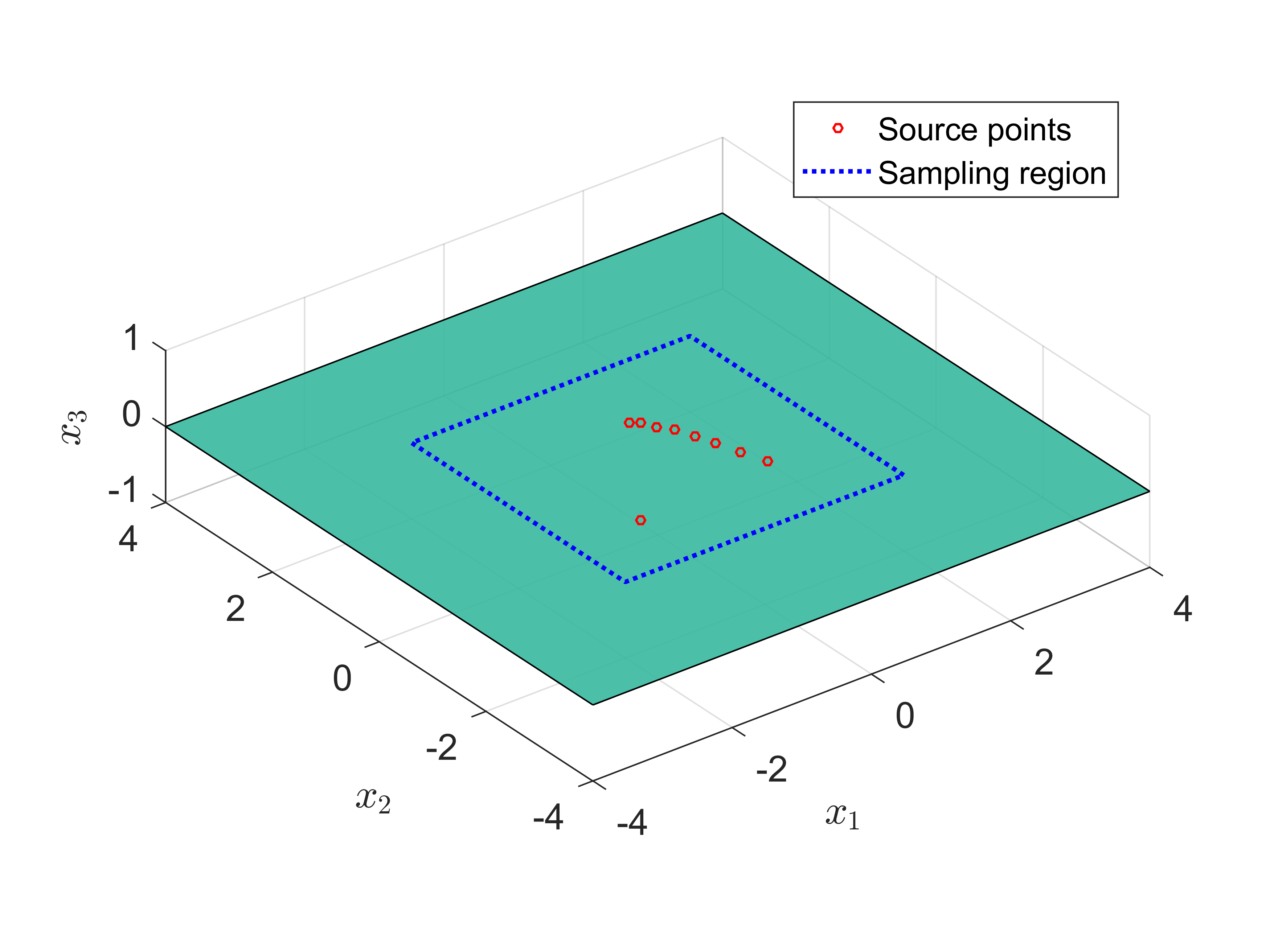} & \includegraphics[width=4.6cm]{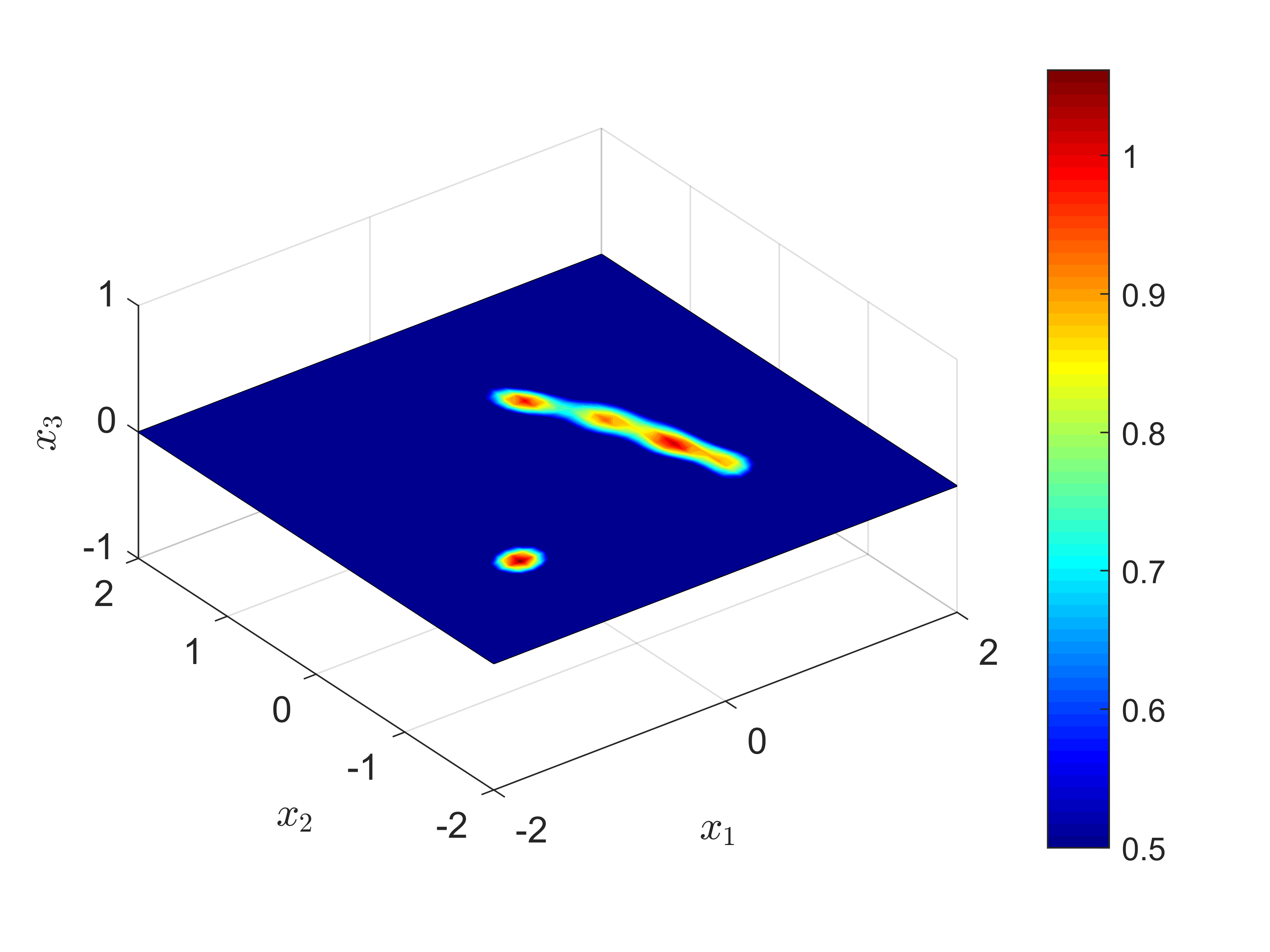} & \includegraphics[width=4.6cm]{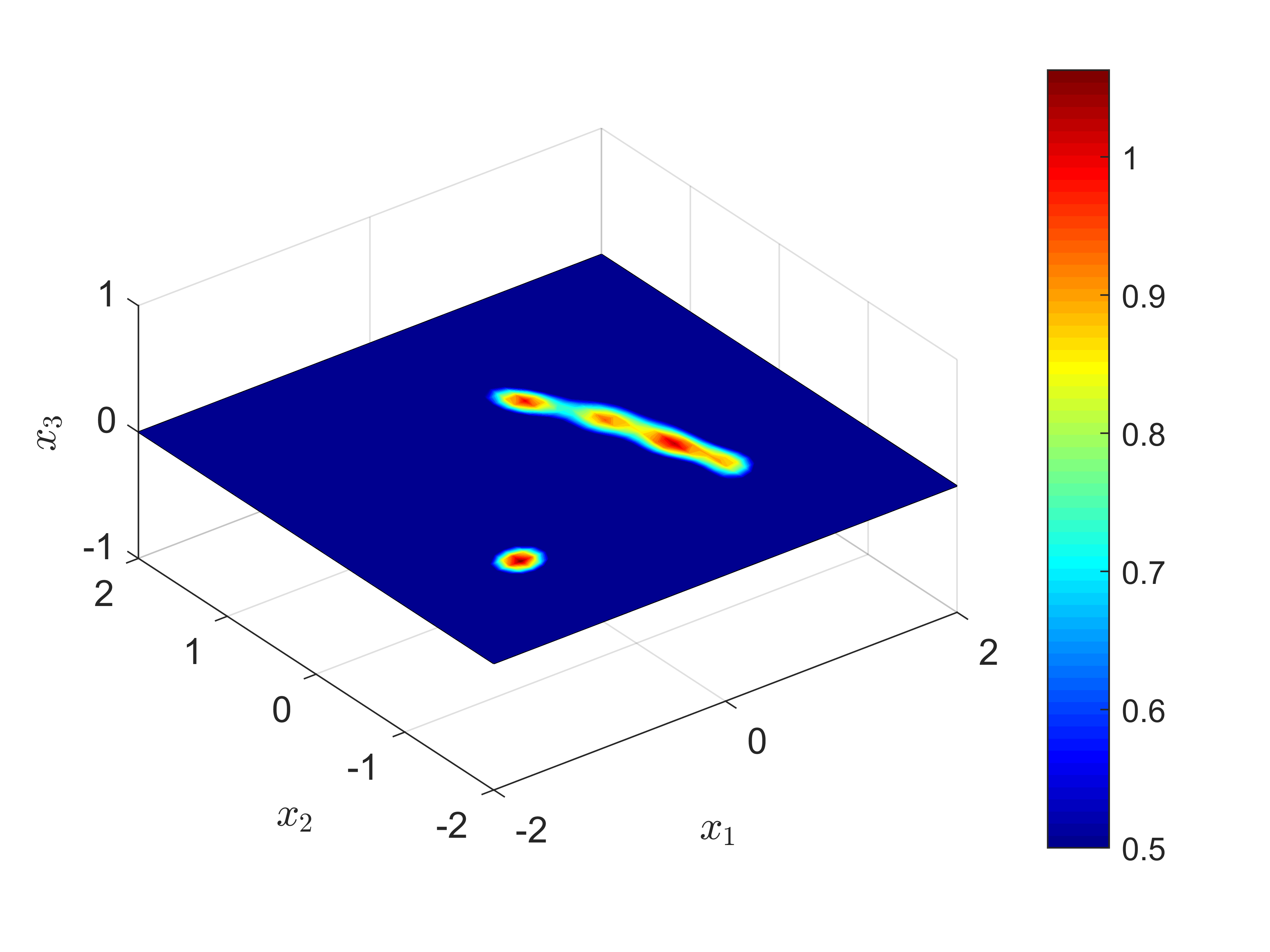}\\
  (a) & (b) & (c)\\
  \includegraphics[width=4.6cm]{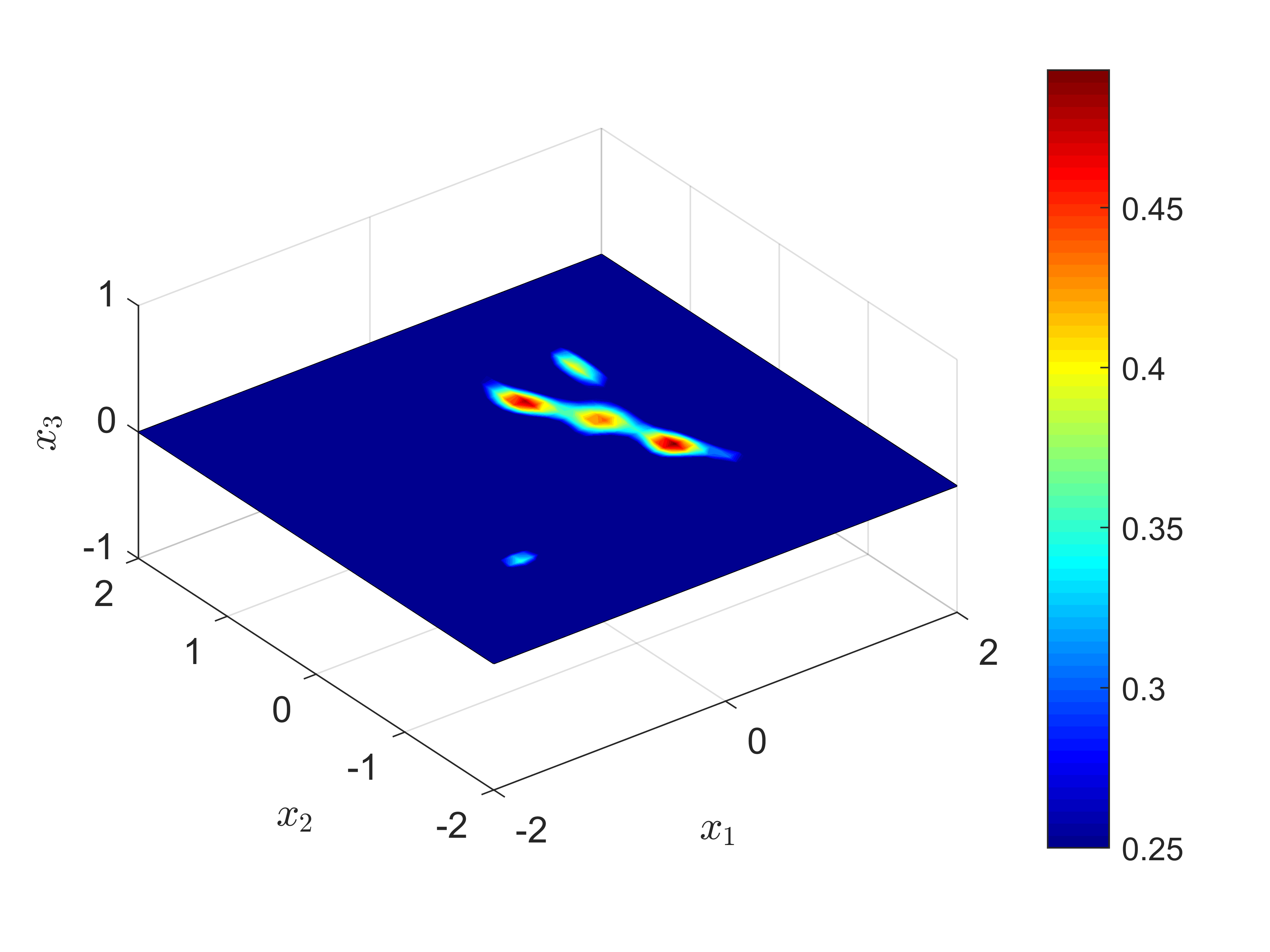} & \includegraphics[width=4.6cm]{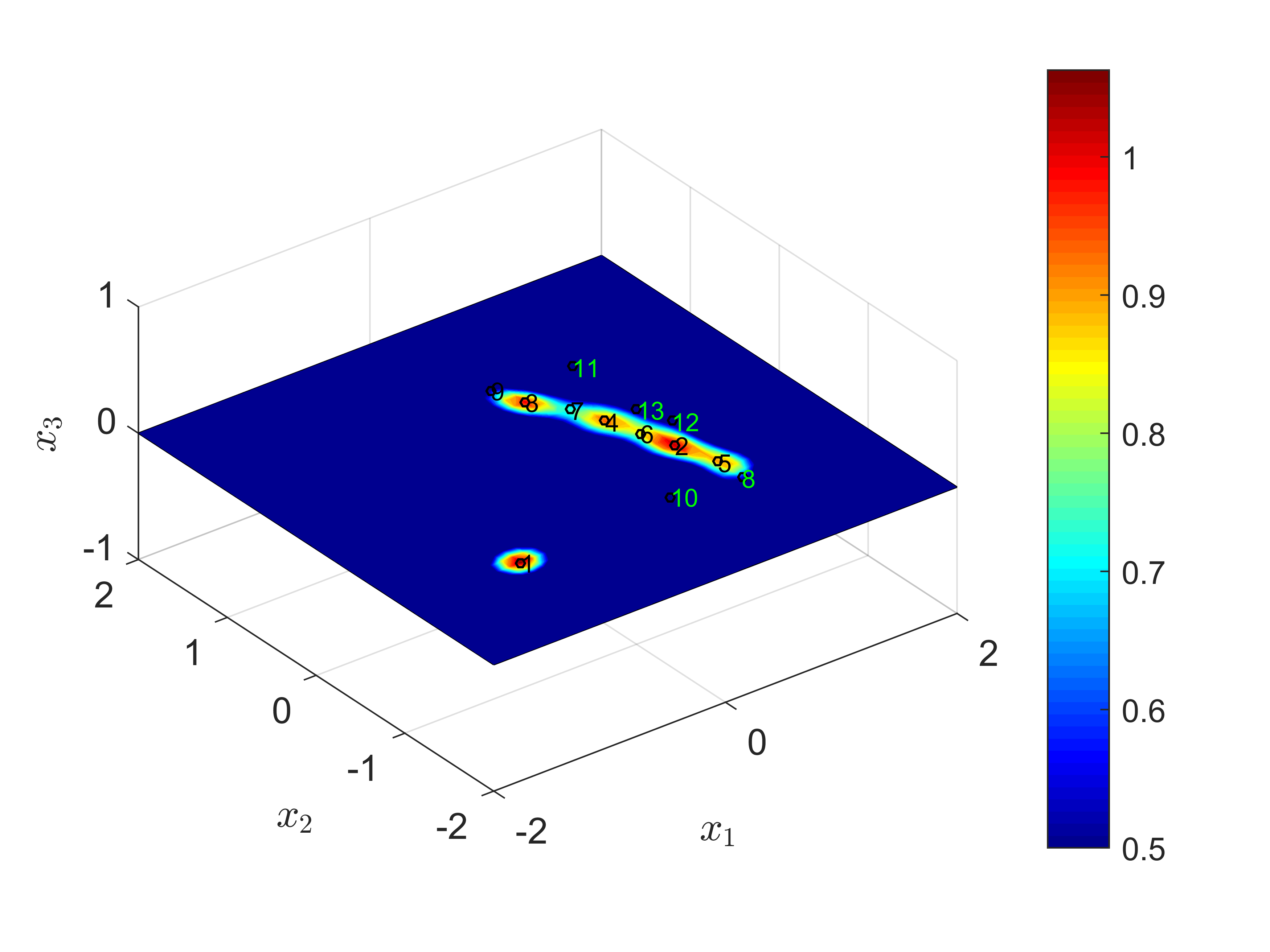} & \includegraphics[width=4.6cm]{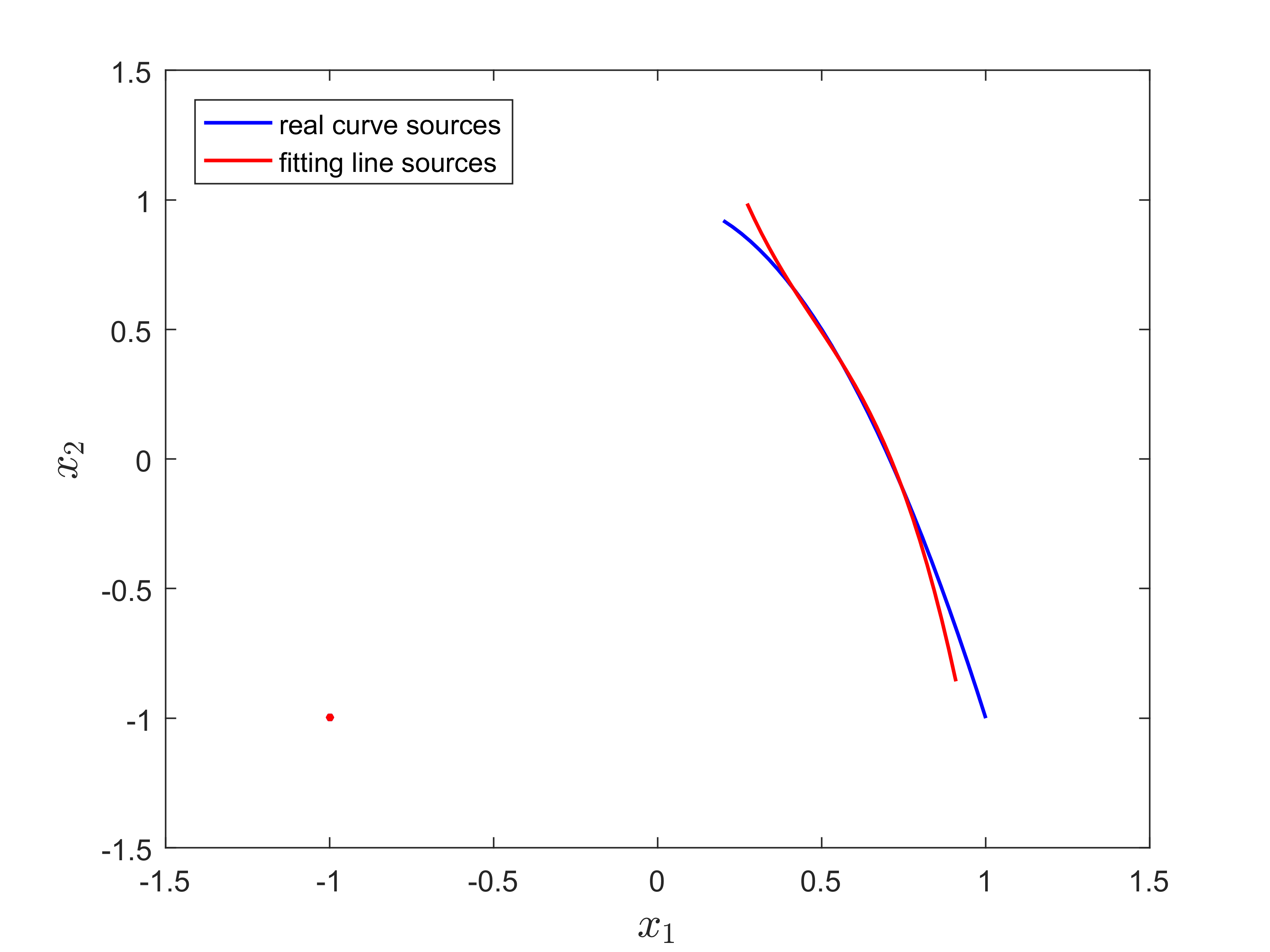} \\
  (d) & (e) & (f)  \\
 \end{tabular}
 \caption{Simultaneous reconstruction of a curve source and a point source. (a) Sketch of the example. (b) Reconstruction with all the sensors, $\epsilon=10\%.$ (c) Reconstruction with all the sensors, $\epsilon=5\%.$ (d) Reconstruction with the left half of the sensors, $\epsilon=5\%$. (e) The approximate source points with all the sensors, $\epsilon=5\%$. (f) Polynomial fitting result.}
 \label{dianxian}
 \vspace{-0.5em}
\end{figure}

\begin{table}
\centering
\caption{Reconstruction of curve source location in example 5.}
\footnotesize
\begin{tabular}{@{}llll} 
\br
No. & \tabincell{c}{The actual curve source} & \tabincell{c}{Reconstruction with all the sensors, $\epsilon=5\%$} \\
\mr
1 & $x_2=-2x_1^2+1,x_1\in[0.2, 1]$ & $x_2=-7.124x_1^3+10.22x_1^2-6.791x_1+2.222,x_1\in[0.2727, 0.9091]$ \\
  & $x_3=0$ & $x_3=0$ \\
2 & $(-1, -1, 0)$ & $(-1, -1, 0)$ \\
\br
\label{dianxianb}
\end{tabular}
\end{table}

\section{Conclusion}
In this paper, we have proposed a simple sampling method with a novel indicator function to reconstruct both multiple point sources and sources on a curve. The effectiveness of the proposed method has been proved by both theoretical analysis and numerical experiments.

\vspace{1em}
\noindent\textbf{Acknowledgments}\\
The work of Bo Chen was supported by the National Natural Science Foundation of China (No. 12101603), the Fundamental Research Funds for the Central Universities (Special Project for Civil Aviation University of China, No. 3122021072) and the supplementary project of Civil Aviation University of China (No. 3122022PT19).

\vspace{1em}
\noindent\textbf{ORCID iDs}\\
Bo Chen: https://orcid.org/0000-0003-1708-9542\\
Yao Sun: https://orcid.org/0000-0002-9165-1735

\vspace{1em}
\noindent\textbf{References}

\bibliographystyle{plain}
\bibliography{cankaowenxian}

\end{document}